\documentclass{article}

\usepackage{hyperref}

\usepackage{booktabs} % for much better looking tables
\usepackage{array} % for better arrays (eg matrices) in maths
\usepackage{paralist} % very flexible & customisable lists (eg. enumerate/itemize, etc.)
\usepackage{verbatim} % adds environment for commenting out blocks of text & for better verbatim
\usepackage{subfig} % make it possible to include more than one captioned figure/table in a single float
\usepackage{amsmath, amsthm, amssymb} 
\usepackage{amsmath, amssymb} 
\usepackage{mathtools}
\usepackage{mathrsfs}
\usepackage{graphicx} % support the \includegraphics command and options

       \usepackage{amssymb}%
     \usepackage{amsbsy}%}

\usepackage[utf8]{inputenc} % set input encoding (not needed with XeLaTeX)

\title{\bfseries Mass-conserving diffusion-based dynamics on graphs
}
\author{Jeremy Budd$^1$, Yves van Gennip$^1$\\
\\
\small{$^1$Delft Institute of Applied Mathematics (DIAM)}\\
\small{Technische Universiteit Delft,} \\\small{Delft, The Netherlands.}\\
   \small{ \textup{\textsf{j.m.budd-1@tudelft.nl \qquad y.vangennip@tudelft.nl}}}\\
}
\date{}
\numberwithin{equation}{section}
\newtheorem{thm}{Theorem}[section]
\newtheorem{mydef}[thm]{Definition}
\newtheorem{prop}[thm]{Proposition}

\newtheorem{lem}[thm]{Lemma}
\newtheorem{cor}[thm]{Corollary}
\newtheorem*{nb}{Note}
\newtheorem*{keywords}{Key Words}
\newtheorem*{subjclass}{2010 AMS Classification}
\newcommand{\be}{\begin{equation}}
\newcommand{\ee}{\end{equation}}

\newcommand{\V}{\mathcal{V}}
\newcommand{\bigO}{\mathcal{O}}
\DeclareMathOperator*{\argmin}{argmin}
\DeclareMathOperator*{\argmax}{argmax}

\DeclareMathOperator{\Ext}{Ext}

\DeclareMathOperator{\GL}{GL_{\varepsilon}}

\DeclarePairedDelimiter\ceil{\lceil}{\rceil}

\DeclarePairedDelimiter\ip{\langle}{\rangle_\V}
\DeclarePairedDelimiter\ipp{\langle\langle}{\rangle\rangle_\V}
\begin{document}
\label{firstpage}
\maketitle
\begin{abstract}
An emerging technique in image segmentation, semi-supervised learning, and general classification problems concerns the use of phase-separating flows defined on finite graphs. This technique was pioneered in Bertozzi and Flenner (2012), which used the Allen--Cahn flow on a graph, and was then extended in Merkurjev, Kosti\'c and Bertozzi (2013) using instead the Merriman--Bence--Osher (MBO) scheme on a graph. In previous work by the authors, Budd and Van Gennip (2019), we gave a theoretical justification for this use of the MBO scheme in place of Allen--Cahn flow, showing that the MBO scheme is a special case of a ``semi-discrete" numerical scheme for Allen--Cahn flow.   

In this paper, we extend this earlier work, showing that this link via the semi-discrete scheme is robust to passing to the mass-conserving case. Inspired by Rubinstein and Sternberg (1992), we define a mass-conserving Allen--Cahn equation on a graph. Then, with the help of the tools of convex optimisation, we show that our earlier machinery can be applied to derive the mass-conserving MBO scheme on a graph as a special case of a semi-discrete scheme for mass-conserving Allen--Cahn. We give a theoretical analysis of this flow and scheme, proving various desired properties like existence and uniqueness of the flow and convergence of the scheme, and also show that the semi-discrete scheme yields a choice function for solutions to the mass-conserving MBO scheme. Finally, we exhibit initial work towards extending to the multi-class case, which in future work we seek to connect to recent work on multi-class MBO in Jacobs, Merkurjev and Esedo\= glu (2018).
\end{abstract}

\begin{keywords}\emph{
Allen--Cahn equation, threshold dynamics, graph dynamics, mass constrained motion, convex optimisation. }
\end{keywords}

\begin{subjclass}
\emph{34B45, 35R02, 34A12, 65N12, 
05C99 }
\end{subjclass}
\section{Introduction}
In this paper, we will investigate variants of the Allen--Cahn equation and Merriman--Bence--Osher (MBO) scheme on a graph, modified to ensure that mass is conserved along trajectories. First, we formulate on a graph the mass-conserving Allen--Cahn flow devised by Rubinstein and Sternberg \cite{RS}, noticing that mass conservation continues to hold in the discrete setting. Next, following our earlier work in \cite{Budd} and drawing on work in Van Gennip \cite{OKMBO}, we show that a formulation of a mass-conserving MBO scheme arises naturally as a special case of a semi-discrete scheme for the mass-conserving Allen--Cahn flow with the double-obstacle potential. We then examine various theoretical properties of this mass-conserving semi-discrete scheme.

Finally, we exhibit results towards generalising to the multi-class case, formulating mass-conserving Allen--Cahn flow with a ``multi-obstacle" potential and thereby deriving a mass conserving multi-class semi-discrete scheme which we hope to link to the multi-class MBO scheme. 
\subsection{Contributions of this work}
In this paper we have:
\begin{itemize}
\item Following \cite{RS}, defined a mass-conserving graph Allen--Cahn flow with double-obstacle potential (Definition \ref{mACdef}) and proved that it conserves mass (Proposition \ref{mACmc}).
\item Extended the analysis in \cite{Budd} to this new flow, proving a weak form, an explicit form, and uniqueness and existence theory for this flow (Theorems \ref{ACobsweak}, \ref{explicitthm}, \ref{uniqueness}, and \ref{existence}, respectively) and via the semi-discrete scheme proved that solutions exhibit monotonic decrease of the Ginzburg--Landau energy, and Lipschitz regularity (Theorems \ref{GLthm} and \ref{mACLips}, respectively).
\item Defined a mass-conserving semi-discrete scheme for this flow (Definition \ref{mSDdef}) and as in \cite{Budd} proved that this scheme is equivalent to a variational scheme of which the MBO scheme is a special case (Theorems \ref{obsMMprop} and \ref{obsMMconverse}).
\item Used the tools of convex optimisation to characterise the solutions of this variational scheme (Theorems \ref{lamb1soln} and \ref{lamble1soln}) and proved that in the MBO limit the mass-conserving semi-discrete solutions converge to an MBO solution, providing a choice function for the mass-conserving MBO solutions (Theorem \ref{sdMBOconv}).
\item Following \cite{Budd}, derived a Lyapunov functional for the mass-conserving semi-discrete scheme (Theorem~\ref{Lyapthm}) and proved convergence of the scheme to the Allen--Cahn trajectory (Theorem \ref{SDlimit}), giving a novel proof of a key lemma from the method in \cite{Budd}.
\item Defined non-mass-conserving and mass-conserving graph Allen--Cahn flows with a \emph{multi}-obstacle potential, and corresponding multi-class semi-discrete schemes (Definitions \ref{mcACdef} and \ref{mcSDdef}, respectively).
\end{itemize}

Though we worked in the framework of \cite{Budd}, this paper extends upon \cite{Budd} in a number of key ways. Most directly, we have shown a new result, that shows that the link discovered in \cite{Budd} between Allen--Cahn flow and the MBO scheme is robust in the prescence of a further constraint. Moreover, this was not a trivial extension: the mass conservation condition substantially increased the difficulty of some of the key results of \cite{Budd}. In particular, finding the solutions of the variational form and thereby proving the equivalence to the semi-discrete scheme for Allen--Cahn, which are both fairly straightforward in \cite{Budd}, required a substantial employment of the tools of convex optimisation. Other results, such as Theorems \ref{ACobsweak} and \ref{sdMBOconv}, also required non-trivial extensions to the proofs of their counterparts in \cite{Budd} (indeed, the latter being in that context sufficiently clear as to not be needed to be stated). Furthermore, for the proof of convergence we have exhibited a novel proof technique for one of the key lemmas. Finally, the final section on multi-obstacle Allen--Cahn was entirely new work.
\subsection{Background} 
The primary background for this work is \cite{Budd}, in which the authors developed a general framework for linking graph Allen--Cahn flow and the graph MBO scheme via a semi-discrete scheme. We showed that the MBO scheme was a special time-discretisation of Allen--Cahn flow with a double-obstacle potential, and investigated properties of this Allen--Cahn flow and time-discretisation scheme. This paper will follow that framework, introducing a mass constraint.

Mass conservation (a.k.a. volume preservation) as a constraint on the MBO scheme and on Allen--Cahn flow arises in a number of contexts, which we shall here briefly survey. For a wider survey of general MBO schemes and Allen--Cahn flow in both the continuum and graph contexts, see \cite{Budd} and \cite{vGGOB} and the references therein.  

In the continuum context, mass-conserving dynamics of the Ginzburg--Landau energy have a long history, dating back to \cite{Cahn} and \cite{CH} and the development of the Cahn--Hilliard equation. In the 1990s, Rubinstein and Sternberg \cite{RS} devised a mass-conserving variant of the Allen--Cahn equation as an alternative to the Cahn--Hilliard equation. We will use this alternative equation as the basis for our mass-conserving graph Allen--Cahn equation.
% [[give physical justification of mass-constrained Ginzburg--Landau, mention Cahn--Hilliard and Rubinstein and Sternberg]]. 

Just as the original MBO scheme was introduced as a method for mean curvature flow in Merriman, Bence, and Osher \cite{MBO1992}, mass-constrained MBO schemes in the continuum have been investigated as methods for studying mass-constrained mean curvature flow. It was first introduced as such in Ruuth and Wetton \cite{RW}, and the convergence of this scheme has been recently studied by Laux and Schwartz \cite{LauxSchwartz}, who showed that as the time-step goes to zero the algorithm of Ruuth and Wetton converges (up to a subsequence) to the weak formulation of mass-constrained mean curvature flow defined in \cite{MSS}. 
%[[Mention mass-constrained MBO as a method for mass-constrained MCF]].

%Mention references in BvG and vGGOB 
Turning to the graph context, recently Van Gennip \cite{OKMBO} studied a graph analogue of the Ohta--Kawasaki functional, and devised a modified graph MBO scheme (with the ordinary MBO scheme as a special case) and mass-conserving graph MBO scheme as a method for minimising this functional without and with a mass conservation constraint, respectively. We will show that the mass-conserving MBO scheme yielded by applying the technique from \cite{Budd} to the Rubinstein and Sternberg Allen--Cahn equation on a graph coincides with this definition of the mass-conserving MBO scheme on graphs (up to non-uniqueness of MBO solutions).%[[more Ohta--Kawasaki stuff?]].

Finally, graph Allen--Cahn flow and MBO schemes have received much attention in the last decade as algorithms for image processing and semi-supervised learning, stemming from pioneering work by Bertozzi and Flenner \cite{BF} and Merkurjev, Kostić and Bertozzi \cite{MKB}, respectively, and extended to the multi-class case in  Merkurjev \emph{et. al.} \cite{mcBertozzi}. Bae and Merkurjev \cite{BM} studied the effect of mass conservation constraints on these algorithms, inspiring Jacobs, Merkurjev, and Esedo\=glu \cite{Auction} to employ ``auction dynamics" as a novel way to solve a mass-conserving multi-class graph MBO scheme. In this work we extend the link developed in \cite{Budd} between these image-processing algorithms to this mass-conserving setting in the two-class case, and demonstrate how to define our framework in the multi-class case. In future work we seek to extend the theory of this paper to the multi-class case, and so link up with this body of work.
\subsection{Paper outline}

We here give a brief overview of the rest of this paper.

In section \ref{gwork}, we outline our notation and key definitions, and then briefly describe the link from \cite{Budd} that we shall be extending to the mass-conserving case in this paper.

In section \ref{mAC}, we define mass-conserving graph Allen--Cahn flow, following Rubinstein and Sternberg's \cite{RS} definition of mass-conserving Allen--Cahn flow in the continuum. We extend the analysis in \cite{Budd} to this mass-conserving setting, rigorously defining this flow with the double-obstacle potential, and proving explicit and weak forms, as well as existence and uniqueness, for this flow.

In section \ref{mSD}, we define the mass-conserving semi-discrete scheme, which we formulate variationally, as in \cite{Budd}, to link to the MBO scheme. We then use the tools of convex optimisation to solve this variational equation, first in the case where the objective function is linear (i.e. MBO) and next (using strong duality) in the general semi-discrete case where the objective function is strictly convex. This task of solving the variational equation is significantly more involved than its counterpart in \cite{Budd}. We show that as the strictly convex case converges to the MBO case, the corresponding solutions converge to a unique MBO solution, providing a choice function for the mass-conserving MBO scheme. We lastly follow \cite{Budd} in deriving a Lyapunov functional for the mass-conserving semi-discrete scheme and thereby discussing the long-time behaviour of the scheme.

In section \ref{SDconvsec}, we follow the method of \cite{Budd} to prove convergence of the mass-conserving semi-discrete scheme to mass-conserving Allen--Cahn flow as the time-step tends to zero. We also give a novel proof of one of the lemmas in this proof. We then use this convergence to prove monotonicity of the Ginzburg--Landau functional along mass-conserving Allen--Cahn trajectories, and prove the Lipschitz regularity of those trajectories. 

Finally, in section \ref{mcAC} we make first steps towards future work concerning a multi-class MBO scheme, defining a graph Allen--Cahn flow with the ``multi-obstacle" potential and a corresponding multi-class semi-discrete scheme, with and without mass conservation.
%[[Too brief??]]

 %[[Classification and machine learning stuff, Auction Dynamics]].    
\section{Groundwork}\label{gwork}

We here rewrite the abridged summary of \cite{vGGOB} from \cite{Budd}. We henceforth consider graphs $G:=(V,E)$ which are finite, simple, connected, undirected and positively weighted, with vertex set $V$, edge set $E\subseteq V^2$ and with weights $\{\omega_{ij}\}_{ij\in E}$ satisfying $\omega_{ij} = \omega_{ji}$ and $\omega_{ij}\geq 0$ for all $ij \in E$. We extend $\omega_{ij}=0$ when $ij\notin E$. We define function spaces on $G$  (where $X\subseteq \mathbb{R}$, and $T\subseteq\mathbb{R}$ an interval): 
\begin{align*}
	&\V := \left\{ u: V\rightarrow\mathbb{R} \right\} , &\V_{X} := \left\{ u: V\rightarrow X \right\}&,  &\mathcal{E} := \left\{ \varphi: E\rightarrow\mathbb{R} %|\varphi_{ij} = -\varphi_{ji} 
\right\}.&\\
	&\V_{t\in T} := \left\{ u: T\rightarrow\V \right\} , &\V_{X,t\in T} := \left\{ u: T\rightarrow \V_X \right\}.&
\end{align*}
We introduce a Hilbert space structure on these function spaces.  For $r\in [0,1]$, and writing $d_i:=\sum_j \omega_{ij}$ for the degree of vertex $i$,  we define inner products on $\V$ and $\mathcal{E}$%\footnote{We here take, so as to preserve the $\Gamma$-convergence result in \cite{vGB}, $q=1$ in the definitions in \cite{vGGOB}.}
\begin{align*}
	&\ip{u,v} := \sum_{i\in V} u_i v_i d_i^r, &\langle\varphi,\phi\rangle_\mathcal{E}:=\frac{1}{2}\sum_{i,j\in V} \varphi_{ij} \phi_{ij}\omega_{ij}&
\end{align*} and define the inner product on $\V_{t\in T}$ (or $\V_{X,t\in T}$) \[ (u,v)_{t\in T}:=\int_T \left\langle u(t),v(t)\right\rangle_\V \;dt = \sum_{i\in V} d_i^r\, (u_i,v_i)_{L^2(T;\mathbb{R})}.\] 
We then induce inner product norms $||\cdot||_\V$, $||\cdot||_\mathcal{E}$ and $||\cdot||_{t\in T}$ and also define on $\V$ the norm $ ||u||_\infty := \max_{i\in V} |u_i|$.
Next, we define the $L^2$ and $L^\infty$ spaces: \begin{align*} &L^2(T;\V) : =\left\{ u\in\V_{t\in T} \mid ||u||_{t\in T} <\infty \right\}, \\&L^\infty(T;\V) : =\left\{ u\in\V_{t\in T}\mid \exists C\in\mathbb{R}, || u(t)||_{\infty} <C\text{ for a.e. } t \in T  \right\}.\end{align*}
Finally, for $T$ an open interval, we define the {Sobolev space} $H^1(T;\V)$ as the set of $u \in L^2(T;\V)$ with weak derivative $du/dt \in L^2(T;\V)$ such that  
\[  \forall \varphi\in C^\infty_c(T;\V)\:\:\left(u,\frac{d\varphi}{dt}\right)_{t\in T} = -\left(\frac{du}{dt},\varphi\right)_{t\in T} \] 
where $C^\infty_c(T;\V)$ denotes the set of elements of $\V_{t\in T}$ that are infinitely differentiable with respect to time and compactly supported in $T$. By \cite[Proposition 1]{Budd},
$u\in H^1(T;\V)$ if and only if $u_i \in H^1(T;\mathbb{R})$ for each $i\in V$. Then $H^1(T;\V)$ has inner product:\[ (u,v)_{H^1(T;\V)} := (u,v)_{t\in T} + \left(\frac{du}{dt},\frac{dv}{dt}\right)_{t\in T} = \sum_{i\in V} d_i^r (u_i,v_i)_{H^1(T;\mathbb{R})}.\] We also define the local $H^1$ space on any interval $T$: \[ H^1_{loc}(T;\V) :=\left\{u\in \V_{t\in T}\,\middle|\,\forall a,b\in T, \: u\in H^1((a,b);\V) \right\}\] and likewise define $L^2_{loc}(T;\V)$ and $L^\infty_{loc}(T;\V)$. 

We introduce some notation: for $A\subseteq V$, define $\chi_A\in\V$ by
\[ (\chi_A)_i := \begin{cases} 1, & \text{if }i\in A,\\
 0, & \text{if }i\notin A.\end{cases}\]
Next, we introduce the graph gradient and Laplacian:
\begin{align*}
	&(\nabla u)_{ij}:=\begin{cases}u_j -u_i, & ij\in E\\ 0, &\text{otherwise} \end{cases} &(\Delta u)_i:=d_i^{-r}\sum_{j\in V}\omega_{ij}(u_i-u_j).&
\end{align*}
We note that $\Delta$ is positive semi-definite and self-adjoint with respect to $\V$. From $\Delta$ we define the \emph{graph diffusion operator}: \[e^{-t\Delta}u:=\sum_{n\geq 0} \frac{(-1)^n t^n}{n!}\Delta^n u\] where $v(t)=e^{-t\Delta}u $ is the unique solution to ${dv}/{dt} = -\Delta v$ with $v(0) = u$. Note that $e^{-t\Delta}\mathbf{1} = \mathbf{1}$, where $\mathbf{1}$ is the vector of ones, so graph diffusion is mass-conserving, i.e. $\ip{e^{-t\Delta}u,\mathbf{1}}=\ip{u,\mathbf{1}}$.
	By \cite[Proposition 2]{Budd} if $u\in H^1(T;\V)$ and $T$ is bounded below, then $e^{-t\Delta}u\in H^1(T;\V)$ with \[\frac{d}{dt}\left(e^{-t\Delta}u\right) = e^{-t\Delta}\frac{du}{dt} -  e^{-t\Delta}\Delta u.\] 
We recall from functional analysis the notation, for any linear $F:\V\rightarrow\V$,
\begin{align*} &\rho(F):=\max\{|\lambda| :\text{$\lambda$ an eigenvalue of $F$}\}\\
&||F|| := \sup_{||u||_\V = 1} ||Fu||_\V
\end{align*}
and recall the standard result that if $F$ is self-adjoint then $||F|| = \rho(F)$.

Finally, we recall the notation from \cite{Budd}: for problems of the form \[ \underset{x}{\argmin}\: f(x) \] we write $f \simeq g$ and say $f$ and $g$ are \emph{equivalent} when $g(x) = af(x) + b$ for $a>0$ and $b$ independent of $x$. As a result, replacing $f$ by $g$ does not affect the minimisers.

%which we can rewrite with the equivalent functional\footnote{Can check that $\left\langle 1-2e^{-\tau\Delta}u_n,u\right\rangle_\V=\ip{u,1-u}+\ip{u-e^{-\tau\Delta}u_n,u-e^{-\tau\Delta}u_n}-\ip{e^{-\tau\Delta}u_n,e^{-\tau\Delta}u_n}$. Then suppress the constant (in $u$) $\ip{e^{-\tau\Delta}u_n,e^{-\tau\Delta}u_n}$ term and divide by $2\tau$.}:
%\begin{equation}
%		\label{MBOMM}
%	u_{n+1}\in\underset{u\in\V_{[0,1]}}{\argmin}\: \frac{1}{2\tau}\ip{1-u,u} + \frac{\left|\left|u-e^{-\tau\Delta}u_n\right|\right|^2_\V}{2\tau}. 
%\end{equation}
 To define \emph{graph Allen\textendash Cahn} (\emph{AC}) \emph{flow}, we first define the \emph{graph Ginzburg\textendash Landau functional} as in {\cite{Budd}} by 
\begin{equation}
	\label{GL}
	\GL(u) := \frac{1}{2}\left|\left|\nabla u\right|\right|_\mathcal{E}^2 +\frac{1}{\varepsilon}\left\langle W\circ u,\mathbf{1} \right \rangle_\V
\end{equation}
where $W$ is a double-well potential and $\varepsilon>0$ is a scaling parameter. AC flow is then the $\ip{\cdot,\cdot}$ gradient flow of $\GL$, which for $W$ differentiable is given by the ODE
\begin{equation}
	\label{AC}
	\frac{du}{dt} = -\Delta u - \frac{1}{\varepsilon} W'\circ u = -\nabla_\V\GL(u)
\end{equation}
where $\nabla_\V$ is the Hilbert space gradient on $\V$. 

In \cite{Budd} AC flow was linked to the MBO scheme via a discretisation of it by the {``semi-discrete"} implicit Euler scheme (with time step $\tau\geq 0$):
\begin{equation}
	\label{sdAC}
	u_{n+1}=e^{-\tau\Delta} u_n- \frac{\tau}{\varepsilon} W'\circ u_{n+1}.
\end{equation}
This obeys the variational equation: \begin{equation}
	\label{ACMM}
	\begin{split}
	u_{n+1}\in \underset{u\in \V}{ \argmin }  \: &\frac{1}{\varepsilon}\left\langle W\circ u,\mathbf{1} \right \rangle_\V + \frac{\left|\left|u-e^{-\tau\Delta}u_n\right|\right|^2_\V}{2\tau}.
	%\\							   \simeq
							   %\footnotemark\:&\frac{1}{\varepsilon}\left\langle W\circ u,1\right \rangle_\V +\left\langle u,\frac{u_n-e^{-\tau\Delta}u_n}{\tau}\right \rangle_\V + \frac{\left|\left|u-u_n\right|\right|^2_\V}{2\tau}.
	\end{split}
\end{equation}
We now define the MBO scheme.
\begin{mydef}[Mass-conserving graph MBO scheme]\label{mMBOdef} We define the \emph{mass-conserving graph Merriman\textendash Bence\textendash Osher (MBO) scheme} by the sequence of variational problems\emph{:} \[u_{n+1}\in \underset{\underset{\ip{u,\mathbf{1}}=\ip{u_n,\mathbf{1}}}{u\in \V_{[0,1]}}}{ \argmin }\: \left\langle \mathbf{1}-2e^{-\tau\Delta}u_n,u\right\rangle_\V.\]
This is motivated by recalling the result from \emph{\cite{vGGOB}} that the ordinary graph MBO scheme, defined as an iterative diffusion \emph{(}for a time $\tau$\emph{)} and thresholding scheme, is equivalent to the sequence of variational problems\emph{:}
\begin{equation*}
	\label{MBO}
	u_{n+1}\in \underset{u\in\V_{[0,1]}}{\argmin}\: \left\langle \mathbf{1}-2e^{-\tau\Delta}u_n,u\right\rangle_\V
\end{equation*} 
	to which we have added a mass conservation constraint on the minimiser. Note that we can suppress the now constant $\ip{\mathbf{1},u}$ term. 
\end{mydef}

To link the AC flow to the MBO scheme, as in \cite{Budd} take as $W$ the \emph{double-obstacle potential}: \begin{equation}
	\label{Wobs}
	W(x) := \begin{cases}
    \frac{1}{2}x(1-x), & \text{for } 0 \leq x \leq 1, \\
    \infty, & \text{otherwise.}  \end{cases}
\end{equation}
See also Blowey and Elliott \cite{BE1991,BE1992,BE1993} for study of this potential in the continuum context and Bosch, Klamt and Stoll \cite{BKS2018} for recent work in the graph context.
%\hrule

As $W$ is not differentiable, the AC flow has to be redefined via the sub-differential of $W$. As in \cite{Budd} we say that a pair $(u,\beta)\in\V_{[0,1],t\in T}\times\V_{t\in T}$ is a solution to double-obstacle AC flow for any interval $T$ when $u\in H_{loc}^1(T;\V)$ and for a.e. $t\in T$ and all $i\in V$: 
\begin{align}
\label{ACobs2}
	&\varepsilon \frac{du_i(t)}{dt} + \varepsilon(\Delta u(t))_i  +\frac{1}{2}-u_i(t)= \beta_i(t), &\beta(t)\in\mathcal{B}(u(t))
\end{align}
where $\mathcal{B}(u)$ is the set (for $I_{[0,1]}(x):=0$ if $x\in[0,1]$ and $I_{[0,1]}(x):=\infty$ otherwise)
\begin{equation}
\label{oldbeta}
	 \mathcal{B}(u) :=\left\{  \alpha\in\V\: \middle|\: \forall i\in V,  \alpha_i\in -\partial I_{[0,1]}(u_i)
	\right\}.
\end{equation}
That is, $\mathcal{B}(u)=\emptyset$ if $u\notin\V_{[0,1]}$, and for $u\in\V_{[0,1]}$ it is the set of $\beta\in\V$ such that 
\[\beta_i\in\begin{cases}
		%\{\infty\}, & u_i<0\\
		[0,\infty), & u_i=0,\\
		\{0\}, &0<u_i < 1,\\
		(-\infty,0], & u_i=1.
		%\{-\infty\}, &u_i>1.
	\end{cases}\]
The semi-discrete scheme thus becomes, where $\lambda:=\tau/\varepsilon$,
\begin{equation}
\label{SDobs}
(1-\lambda)(u_{n+1})_i -(e^{-\tau\Delta}u_n)_i+\frac{\lambda}{2} =\lambda(\beta_{n+1})_i
\end{equation}
where $\beta_{n+1}\in\mathcal{B}(u_{n+1})$. Then the key result of \cite[Theorem 3]{Budd} is the derivation of the MBO scheme from AC flow via the semi-discrete scheme, i.e. that for $\varepsilon=\tau$  the solutions to \eqref{SDobs} obey the variational equation:
\[
 \begin{split}
	u_{n+1}\in \underset{u\in \V_{[0,1]}}{ \argmin }  \: &\left\langle u,\mathbf{1} -u\right \rangle_\V +  \left|\left|u-e^{-\tau\Delta}u_n\right|\right|^2_\V\\
							   %=\:&\ip{u,1}-\ip{u,u}+\ip{u,u}-2\left\langle u,e^{-\tau\Delta}u_n\right \rangle_\V + \left\langle e^{-\tau\Delta}u_n,e^{-\tau\Delta}u_n\right \rangle_\V\\
							\simeq\:&\left\langle u,\mathbf{1} -2e^{-\tau\Delta}u_n\right \rangle_\V
\end{split}
\]
and thus the solutions are MBO trajectories.

This paper will follow this method to derive the mass-conserving MBO scheme as a special case of a semi-discrete scheme for a mass-conserving double-obstacle AC flow.
\section{Mass-conserving AC flow}\label{mAC}
In \cite{RS}, Rubinstein and Sternberg define a mass-conserving Allen--Cahn flow (on a domain $\Omega$) as the non-local reaction-diffusion PDE, where $u:\Omega\rightarrow\mathbb{R}$, \be \frac{\partial u}{\partial t} = \Delta u - W'(u) + \frac{1}{|\Omega|}\int_\Omega W'(u) \; dx\ee
with Neumann boundary conditions. We can readily formulate this on a graph, noting the differing sign convention on $\Delta$ and introducing our scaling, as the ODE \be \label{mAC0} \frac{du}{dt} = -\Delta u - \frac{1}{\varepsilon}W'\circ u +\frac{1}{\varepsilon}\frac{\ip{W'\circ u, \mathbf{1} }}{\ip{\mathbf{1} ,\mathbf{1} }}\mathbf{1} .\ee
Finally, as above in \eqref{ACobs2} we account for the non-differentiability of $W$ to arrive at:
\begin{align}
\label{mACobs0}
	&\varepsilon \frac{du}{dt} + \varepsilon\Delta u(t)-u(t)+\frac{\ip{u(t),\mathbf{1} }}{\ip{\mathbf{1} ,\mathbf{1} }}\mathbf{1} = \beta(t) - \frac{\ip{\beta(t),\mathbf{1} }}{\ip{\mathbf{1} ,\mathbf{1} }}\mathbf{1} , &\beta(t)\in \mathcal{B}(u(t)).
\end{align}
We verify the mass conservation property for $u$ continuous and $H^1$. We first recall from \cite{Budd} a standard fact about continuous representatives of $H^1$ functions.
\begin{lem}[\text{See \cite[Lemma 4]{Budd}}]\label{H1AClem}
If $u\in H^1_{loc}(T;\V)\cap C^0(T;\V)$ or $u\in  H^1_{loc}(T;\mathbb{R})\cap C^0(T;\mathbb{R})$, then $u$ is locally absolutely continous on $T$. It follows that $u$ is differentiable a.e. in $T$, and the weak derivative equals the classical derivative a.e. in $T$.
\end{lem} 
\begin{mydef}
	Define the \emph{mass} of $u\in\V$ to be \be \mathcal{M}(u):=\ip{u,\mathbf{1} }.\ee Furthermore, define the \emph{average value} of $u\in \V$ to be
\be \bar u := \frac{\mathcal{M}(u)}{\mathcal{M}(\mathbf{1})}.\ee
\end{mydef}
\begin{prop}\label{mACmc}
	For any interval $T$ and $u\in H^1_{loc}(T;\V)\cap C^0(T;\V)$, if $u$ obeys \eqref{mACobs0} at a.e. $t\in T$, then for a.e. $t\in T$ \[\frac{d}{dt}\mathcal{M}(u(t))=0\] and so $\mathcal{M}(u(t))$ is constant.
\end{prop}
\begin{proof}
	First, note that $\mathcal{M}(u(t))\in  H^1_{loc}(T;\mathbb{R})\cap C^0(T;\mathbb{R})$ with \[\frac{d}{dt}\mathcal{M}(u(t))=\left\langle\frac{du}{dt},\mathbf{1} \right\rangle_\V\]
since for any $\varphi\in C^\infty_c(T;\mathbb{R})$ 
\[\begin{split}\int_T \ip{u(t),\mathbf{1}} \frac{d\varphi}{dt} \; dt & = \int_T \left\langle u(t),\frac{d\varphi}{dt}\mathbf{1}\right\rangle_\V \; dt\\&=-\int_T  \left\langle \frac{du}{dt},\varphi(t)\mathbf{1}\right\rangle_\V \; dt=-\int_T  \left\langle \frac{du}{dt},\mathbf{1}\right\rangle_\V \varphi(t)\; dt.\end{split}\]
 Then for almost every $t$, taking the mass of both sides of \eqref{mACobs0}: \[\varepsilon \left\langle\frac{du}{dt},\mathbf{1} \right\rangle_\V + \varepsilon\ip{\Delta u(t),\mathbf{1} } -\ip{u(t),\mathbf{1} }+\frac{\ip{u(t),\mathbf{1} }}{\ip{\mathbf{1} ,\mathbf{1} }}\ip{\mathbf{1} ,\mathbf{1} }= \ip{\beta(t),\mathbf{1} } - \frac{\ip{\beta(t),\mathbf{1} }}{\ip{\mathbf{1} ,\mathbf{1} }}\ip{\mathbf{1} ,\mathbf{1} }.\]
	So most of the terms cancel and we are left with\[ \left\langle\frac{du}{dt},\mathbf{1} \right\rangle_\V =-\ip{\Delta u(t),\mathbf{1} }=0\] with the final equality because $\Delta$ is self-adjoint and $\Delta \mathbf{1}  = \mathbf{0}$. Then by absolute continuity we infer that $\mathcal{M}(u(t))$ is constant.
\end{proof}
As in \cite{Budd} with the ordinary Allen--Cahn flow, not all values in the subdifferential are attained in valid trajectories. 
We use Lemma \ref{H1AClem} to characterise the validly attained $\beta$.
\begin{thm}\label{betathm}
Let $(u,\beta)$ obey \eqref{mACobs0} at a.e. $t\in T$, with $u\in H^1_{loc}(T;\V)\cap C^0(T;\V)\cap\V_{[0,1],t\in T}$. Then, for a.e. $t\in T$ and all $i\in V$, we have 
\begin{equation}\label{beta2}
\beta_i(t) -\bar\beta(t)= \begin{cases} \bar u+\varepsilon(\Delta u(t))_i, &\text{if }u_i(t) = 0,\\
-\bar\beta(t), &\text{if }u_i(t) \in (0,1),\\
\bar u-1+\varepsilon(\Delta u(t))_i, &\text{if }u_i(t) = 1.
\end{cases}
\end{equation}
%Therefore for a.e. $t\in T$, $\beta(t) -\bar\beta(t)\mathbf{1} \in \V_{[\bar u-1,\bar u]} \subseteq \V_{[-1,1]}$. If $\bar u\in(0,1) $, we furthermore have that for a.e. $t\in T$, $\beta(t) \in \V_{[-1,1]}$. 
%We recall that this entails that $\GL(u(t))$ is monotonically decreasing in $t$. 
\end{thm}
\begin{proof}
Since $\beta(t)\in\mathcal{B}(u(t))$ at a.e. $t\in T$, \eqref{beta2} holds at  a.e. $t\in T$ for which $u_i(t) \in(0,1)$. Let $\tilde T\subseteq T$ denote the times when $u$ is differentiable and has classical derivative equal to its weak derivative.  
Since $u_i(t)\in[0,1]$ at all times, when $t\in \tilde T$ and $u_i(t) \in\{ 0,1\}$ we have $du_i/dt = 0$. 
Consider first $u_i(t) = 0$. Then for a.e. such $t\in\tilde T$ \[0=\varepsilon \frac{du_i}{dt}(t) = -\varepsilon(\Delta u(t))_i + \beta_i(t)-\bar\beta(t) -\bar u \] so rearranging
\[\beta_i(t) -\bar\beta(t) =\bar u+\varepsilon(\Delta u(t))_i%\leq \bar u
\] 
% since $(\Delta u(t))_i\leq0$ as $i$ a minimiser of $u(t)$. 
Likewise for $u_i(t)=1$ we have for a.e. such $t\in \tilde T$
\[\beta_i(t) -\bar\beta(t) =\bar u-1+\varepsilon(\Delta u(t))_i%\geq \bar u -1
\]
so \eqref{beta2} holds at a.e. $t\in \tilde T$. By Lemma \ref{H1AClem}, $T\setminus\tilde T$ is  null, so  \eqref{beta2} holds at a.e. $t\in T$. 
%
%
%Note that by \eqref{beta2} and the sign properties of $\Delta u$, at a.e. $t\in T$ we have $\beta_i(t)-\bar\beta(t)\in[\bar u,\bar u-1]$ when $u_i(t)\in\{0,1\}$. Thus it suffices to show that $\bar\beta(t)\in[-\bar u,1-\bar u]$. Since $\beta(t)\in\mathcal{B}(u(t))$ at a.e $t\in T$, it follows that at such a $t$ if $\exists i\in V$ with $u_i(t)=0$ then $-\bar\beta(t)\leq \bar u-\beta_i(t)\leq \bar u$, if $\exists i\in V$ with $u_i(t)=1$ then $-\bar\beta(t)\geq \bar u-1-\beta_i(t)\geq \bar u-1$. 
%
%Now suppose at such a $t$ that $\bar\beta(t) > 1 -\bar u$: then by the above discussion $\forall i\in V$ $u_i(t)<1$, and so $\beta_i(t)\geq 0$ and $\bar u<1$. We have that \[ \beta_i(t) =     \begin{cases}
%0, &\text{if }u_i(t) \in (0,1),\\
%\bar u+\bar\beta(t)+\varepsilon(\Delta u(t))_i, &\text{if }u_i(t) = 0.
%\end{cases}\]
%By \eqref{mACobs0}, when $u_i(t)\in(0,1)$ we have
%\[\bar\beta(t) = u_i(t) -\bar u -\varepsilon(\Delta u(t))_i -\varepsilon\frac{du_i}{dt}\]
%and so for all such $i\in V$
%\[u_i(t) -\varepsilon(\Delta u(t))_i -\varepsilon\frac{du_i}{dt} > 1.\]
%
%
%
%Finally, consider the set $B(t) := \{ (\beta_i(t)\mid i\in V\}$. By above, $B(t) - \bar\beta(t)\subseteq[\bar u-1,\bar u]$, so we have that $\operatorname{diam} B(t) \leq 1$. Furthermore, $u(t)\neq\mathbf{0,1}$, so since $\beta(t)\in\mathcal{B}(u(t))$ we have $x,y\in B(t)$ such that $x\geq 0$ and $y\leq 0$. Therefore $B(t)\subseteq[-1,1]$.
\end{proof}
\begin{mydef}[Mass-conserving double-obstacle AC flow]\label{mACdef}
Let $T$ be any interval. A pair $(u,\beta)\in\V_{[0,1],t\in T}\times\V_{t\in T}$ is a solution to mass-conserving double-obstacle AC flow on $T$ when $u\in H^1_{loc}(T;\V)\cap C^0(T;\V)$ and for almost every $t\in T$
\begin{align}
\label{mACobs}
	&\varepsilon \frac{du}{dt} + \varepsilon\Delta u(t)-u(t)+\frac{\ip{u(t),\mathbf{1} }}{\ip{\mathbf{1} ,\mathbf{1} }}\mathbf{1} = \beta(t) - \frac{\ip{\beta(t),\mathbf{1} }}{\ip{\mathbf{1} ,\mathbf{1} }}\mathbf{1} , &\beta(t)\in\mathcal{B}(u(t)).
\end{align}
For brevity we will often refer to just $u$ as a solution to \eqref{mACobs}.
\end{mydef}
\subsection{Weak form and explicit integral form}
In this section, we prove first a weak form of mass-conserving AC flow, and then an explicit integral form.
\begin{thm}[\text{Cf. \cite[Proposition 10]{Budd}}]\label{ACobsweak}
	A function $u\in\V_{[0,1],t\in T}\cap H^1_{loc}(T;\V)$ \emph{(}and associated $\beta$% $\beta_i:= \varepsilon \frac{du_i}{dt} + \varepsilon(\Delta u)_i-u_i+\frac{1}{2}$
\emph{)} 
is a solution to \eqref{mACobs} if and only if for a.e. $t\in T$ and $\forall\eta\in\V_{[0,1]}$ such that $\mathcal{M}(\eta) = \mathcal{M}(u(t))$ \emph{(}i.e. $\eta - u(t) \bot \mathbf{1}$\emph{)}, the following hold
		\begin{subequations}
		\label{mACobsweak}
		\be		\label{mACobsweaka}
		\left\langle\varepsilon\frac{du}{dt}-u(t),\eta-u(t)\right\rangle_\V+\varepsilon\left\langle\nabla u(t),\nabla\eta-\nabla u(t)\right\rangle_\mathcal{E}\ge0,
		\ee
		\be 		\label{mACobsweakb}
		\left\langle\frac{du}{dt},\mathbf{1} \right\rangle_\V = 0.
		\ee
	\end{subequations}
\end{thm}
\begin{proof}
Let $u$ satisfy \eqref{mACobs}. Then for a.e. $t\in T$ we have \eqref{mACobsweakb} and $\beta(t)\in\mathcal{B}(u(t))$, so in particular $\beta_i(t) \geq 0$ and $\beta_i(t)\leq 0$ when $u_i(t)$ is 0 and 1 respectively. Therefore, for all $\eta\in\V_{[0,1]}$ with $\eta - u(t) \bot \mathbf{1}$, for a.e. $t\in T$ we verify \eqref{mACobsweaka}:
\[\begin{split}
	&LHS\\ &= \left\langle-\varepsilon\Delta u(t) -\frac{\ip{u(t),\mathbf{1} }}{\ip{\mathbf{1} ,\mathbf{1} }}\mathbf{1} +\beta(t) - \frac{\ip{\beta(t),\mathbf{1} }}{\ip{\mathbf{1} ,\mathbf{1} }}\mathbf{1} ,\eta-u(t)\right\rangle_\V+\varepsilon\left\langle\nabla u(t),\nabla\eta-\nabla u(t)\right\rangle_\mathcal{E}\\
	&=\ip{\beta(t),\eta-u(t)} \\
	&= \sum_{\{i|u_i(t)=0\}} d_i^r\beta_i(t)\eta_i + \sum_{\{i|u_i(t)=1\}} d_i^r\beta_i(t)(\eta_i-1)\geq 0.
\end{split}
\]	
	Now let $u\in\V_{[0,1],t\in T}\cap H^1_{loc}(T;\V)$ satisfy \eqref{mACobsweak}.	 Therefore by \eqref{mACobsweaka}, for a.e. $t\in T$ and all $\eta\in\V_{[0,1]}$ with $\eta - u(t) \bot \mathbf{1}$
\[  \left\langle\varepsilon\frac{du}{dt}-u(t)+\varepsilon\Delta u(t),\eta-u(t)\right\rangle_\V \ge 0 \]
and so for any $\theta:T\rightarrow\mathbb{R}$ and any $\eta$ as before,
\be\label{weaka2}  \left\langle\varepsilon\frac{du}{dt}-u(t)+\varepsilon\Delta u(t)+\frac{\ip{u(t),\mathbf{1} }}{\ip{\mathbf{1} ,\mathbf{1} }}\mathbf{1}+\theta(t)\mathbf{1} ,\eta-u(t)\right\rangle_\V \ge 0. \ee
For a specific $\theta$ to be determined later, define \be\label{b1}\beta(t) :=\varepsilon\frac{du}{dt}-u(t)+\varepsilon\Delta u(t)+\frac{\ip{u(t),\mathbf{1} }}{\ip{\mathbf{1} ,\mathbf{1} }}\mathbf{1}+\theta(t)\mathbf{1}.\ee
We will check that $\beta(t)\in\mathcal{B}(u(t))$. We consider certain valid test functions $\eta$ for \eqref{weaka2}. In particular, choose $i,j\in V$ and set $\eta_k = u_k(t) \in [0,1]$ for all $k\neq i,j$. Next, define the translated test function $\xi:=\eta-u(t)$, so $\xi_k=0$ for $k\neq i,j$. Then $\eta=u(t)+\xi$ is valid  if and only if  $\xi_i \in [-u_i(t),1-u_i(t)]$,  $\xi_j \in [-u_j(t),1-u_j(t)]$, and $\mathcal{M}(\xi) = 0 $, i.e.
\[ d_i^r\xi_i + d_j^r\xi_j=0. \]
\begin{nb} If $u_i(t) = 0$ and $u_j(t)>0$ then for $0<\alpha\leq 1$ sufficiently small
\begin{align*} &\xi_j = -\alpha u_j(t)\in[-u_j(t),0) &\xi_i = \alpha  d_i^{-r}d_j^ru_j \in (0,1-u_i(t)]\end{align*}
 is a valid $\xi$ with $\xi_i >0$. Likewise, if $u_i(t) = 1$ and $u_j(t)<1$ there is a valid $\xi$ with $\xi_i <0$, and if $u_i(t),u_j(t)\in(0,1)$ there are valid $\xi$ with $\xi_i >0$ and valid $\xi$ with $\xi_i < 0$. 
\end{nb}
For any valid $\xi$, by \eqref{weaka2} and \eqref{b1} we have that 
\[ d_i^r\xi_i\beta_i(t) + d_j^r\xi_j\beta_j(t)\geq 0 \]
and so since $d_i^r\xi_i + d_j^r\xi_j= 0$, 
\be \label{xiineq} d_i^r\xi_i(\beta_i(t)-\beta_j(t))\geq 0.\ee 

Next, first suppose $u_j(t) \in (0,1)$ for some $j\in V$. Then we fix such a $j$ and choose $\theta(t)$ so that $\beta_j(t) = 0$, and thus by \eqref{xiineq} for any $i\in V$ and valid $\xi$: 
\[\xi_i\beta_i(t)\geq 0.\]
Then by the above note, if we choose a valid $\xi$ with $\xi_i$ of the appropriate sign, \[\beta_i(t)\begin{cases}
		=0, & \text{if }u_i(t)\in (0,1),\\
		\leq 0, &\text{if }u_i(t)=1,\\
		\geq 0, &\text{if }u_i(t)=0.
	\end{cases}\]

Next, suppose no such $j$ exists. By above if $u_i(t) = 0$ and $u_j(t) = 1$ then we can choose $\xi_i > 0$ and so by \eqref{xiineq} we have that $\beta_j(t)\leq \beta_i(t)$. Thus we can choose $\theta(t)$ to add an appropriate constant to the values of $\beta(t)$ so that \[0\in\left[\max_{u_j(t)=1}\beta_j(t),\min_{u_i(t)=0}\beta_i(t)\right].\]
Hence we have 
\[\beta_i(t)\begin{cases}
		\leq 0, &\text{if }u_i(t)=1,\\
		\geq 0, &\text{if }u_i(t)=0,
	\end{cases}\]
so $\beta(t)\in\mathcal{B}(u(t))$.

Note finally that whatever the choice of $\theta$, by \eqref{mACobsweakb} and \eqref{b1} we have
\[\ip{\beta(t),\mathbf{1}}=\theta(t)\ip{\mathbf{1},\mathbf{1}}.\] 
Hence by \eqref{b1}
\[ \varepsilon \frac{du}{dt} + \varepsilon\Delta u(t)-u(t)+\frac{\ip{u(t),\mathbf{1} }}{\ip{\mathbf{1} ,\mathbf{1} }}\mathbf{1} = \beta(t) - \frac{\ip{\beta(t),\mathbf{1} }}{\ip{\mathbf{1} ,\mathbf{1} }}\mathbf{1}\] and we chose $\theta(t)$ so that our choice of $\beta(t)\in\mathcal{B}(u(t))$. Hence $(u,\beta)$ solves \eqref{mACobs}.
\end{proof}
\begin{thm}\label{explicitthm}
For $u\in \V_{[0,1],t\in T}$ and $\beta \in \V_{t\in T}$, $(u,\beta)$ is a solution to \eqref{mACobs} if and only if $\beta -\bar\beta \mathbf{1}$ is locally essentially bounded and locally integrable \emph{(}where by ``locally" we mean on each bounded subinterval of $T$\emph{)}, $\beta(t)\in\mathcal{B}(u(t))$ for a.e. $t\in T$, and for all $t\in T$
 \be\label{mACobsexplicit}
u(t) = \bar u \mathbf{1} + e^{t/\varepsilon}e^{-t\Delta}\left(u(0) -\bar u\mathbf{1}\right) + \frac{1}{\varepsilon}e^{t/\varepsilon}e^{-t\Delta}\int_0^t e^{-s/\varepsilon}e^{s\Delta}\left(\beta(s) - \bar \beta(s)\mathbf{1} \right) \; ds.
\ee
\end{thm}
\begin{proof} Let $(u,\beta)$ solve \eqref{mACobs}. Then $\beta - \bar\beta\mathbf{1}$ is a sum of a continuous function and the derivative of a $H^1_{loc}$ function and hence is locally integrable. We shall prove that $\beta - \bar\beta\mathbf{1}$ is globally essentially bounded in Lemma \ref{gammalem}. Finally, following \cite{Budd}, we rewrite \eqref{mACobs} to obtain \eqref{mACobsexplicit}. Consider the expression:\be \label{exp1} \varepsilon \frac{d}{dt} \left( e^{-t/\varepsilon}e^{t\Delta}(u-\bar u\mathbf{1})\right).\ee Applying the product rule we obtain that for a.e. $t \in T$, \[\begin{split}\eqref{exp1}&=-e^{-t/\varepsilon}e^{t\Delta}(u-\bar u\mathbf{1})+\varepsilon  e^{-t/\varepsilon}\frac{d}{dt} \left(e^{t\Delta}(u-\bar u\mathbf{1})\right)\\
&=-e^{-t/\varepsilon}e^{t\Delta}(u-\bar u\mathbf{1})+\varepsilon  e^{-t/\varepsilon}e^{t\Delta}\Delta(u-\bar u\mathbf{1})+\varepsilon  e^{-t/\varepsilon}e^{t\Delta}\frac{du}{dt}\\
&=e^{-t/\varepsilon}e^{t\Delta}\left(\varepsilon\frac{du}{dt} +\varepsilon\Delta u -u + \bar u \mathbf{1}\right)=e^{-t/\varepsilon}e^{t\Delta}\left(\beta(t) - \frac{\ip{\beta(t),\mathbf{1} }}{\ip{\mathbf{1} ,\mathbf{1} }}\mathbf{1} \right)\end{split}\]
and therefore integrating both sides and applying the `fundamental theorem of calculus' on $H^1$ \cite[Theorem 8.2]{Brezis} we obtain the integral form.

Now let $\xi:=\beta -\bar\beta \mathbf{1}$ be locally essentially bounded and locally integrable, let $\beta(t)\in\mathcal{B}(u(t))$ for a.e. $t\in T$, and for all $t\in T$ let \eqref{mACobsexplicit} hold. By differentiating and reversing the above steps we get that $(u,\beta)$ obeys the ODE in \eqref{mACobs}, and in particular the weak derivative of $u$ is given by: 
\[
\frac{du}{dt} = \left(\frac{1}{\varepsilon}I -\Delta\right)e^{t/\varepsilon}e^{-t\Delta}\left(u(0)-\bar u \mathbf{1}\right) + \frac{1}{\varepsilon}\xi(t) + \left(\frac{1}{\varepsilon}I -\Delta\right) \frac{1}{\varepsilon}\int_0^t e^{(t-s)/\varepsilon}e^{-(t-s)\Delta}\xi(s) \; ds.
\]
As $\xi$ is locally essentially bounded, by \eqref{mACobsexplicit} $u$ is continuous, and since $u$ is bounded it is locally $L^2$. Finally, by above $du/dt$ is a sum of (respectively) a smooth function, a locally essentially bounded function and the integral of a locally essentially bounded function, so is locally essentially bounded and hence locally $L^2$. Hence $u\in H^1_{loc}(T;\V)$. 
\end{proof}
%\be\label{mACobsexplicit}
%u(t) = \bar u \mathbf{1} + e^{t/\varepsilon}e^{-t\Delta}\left(u(0) -\bar u\mathbf{1}\right) + \frac{1}{\varepsilon}e^{t/\varepsilon}e^{-t\Delta}\int_0^t e^{-s/\varepsilon}e^{s\Delta}\left(\beta(s) - \bar \beta(s)\mathbf{1} \right) \; ds.
%\ee
\begin{nb}
The forward reference to Lemma \ref{gammalem} does not introduce circularity here, because we do not use this aspect of the forward direction of this theorem until after proving that lemma. We will however use the converse direction in proving the convergence of the semi-discrete scheme \emph{(}Theorem \ref{SDlimit}\emph{)}.

Note also that by \eqref{mACobsexplicit}, if $\beta(t) = \mathbf{0}$ for a.e. $t\in [0,\infty)$, then 
\[ u(t) = \bar u \mathbf{1} + \sum_{k=1}^{|V|-1} e^{(1/\varepsilon - \mu_k)t}\ip{u(0),\xi_k}\xi_k  \] 
where $\{(\mu_k,\xi_k)\}_{k=0}^{|V|-1}$ are the orthonormal eigenpairs of $\Delta$ in increasing order of eigenvalue \emph{(}so $\mu_0 = 0$ and $\xi_0\propto \mathbf{1}$\emph{)}. Let $\ell$ be the least $k\geq 1$ such that $\ip{u(0),\xi_k}\neq 0$. Then to leading order 
\[
u(t) \approx \bar u \mathbf{1} + e^{(1/\varepsilon - \mu_\ell)t}\ip{u(0),\xi_\ell}\xi_\ell
\]
which if $\mu_\ell < 1/\varepsilon$ contradicts $u(t)\in\V_{[0,1]}$ for sufficiently large $t$. Hence in such a case we must have $\beta(t)\neq \mathbf{0}$ for a non-null subset of the time. In particular, if $\varepsilon < 1/||\Delta||$ then this holds unless $u(0)=\bar u\mathbf{1}$.
\end{nb}
\subsection{Existence and uniqueness}
Finally, we have the following existence and uniqueness theory for \eqref{mACobs}. 
\begin{thm}\label{uniqueness}
Let $(u,\beta),(v,\gamma)$ solve \eqref{mACobs} on $T:=[0,T_0]$ or $[0,\infty)$, with $u(0)=v(0)$. Then for all $t\in T$, $u(t) = v(t)$, and there exists $\tilde T$ such that $T\setminus \tilde T$ has zero measure and for all $t\in \tilde T$, $\beta(t) - \gamma(t)= (\bar\beta(t) -\bar\gamma(t))\mathbf{1}$. Furthermore, if $u_i(t)\in (0,1)$ for some $i\in V$ and $t\in \tilde T$, then $\beta(t)=\gamma(t)$. 
\end{thm}
\begin{proof}
As $u$ and $v$ solve \eqref{mACobs}, by subtracting and since $\bar u = \bar v$  we get for a.e. $t\in T$  
\[\varepsilon\frac{d}{dt}(v(t)-u(t)) +\varepsilon\Delta(v(t)-u(t)) -(v(t)-u(t)) = (\gamma(t)-\beta(t)) +(\bar\beta(t)-\bar\gamma(t))\mathbf{1}.\]
Let $w:=v-u$ and take the inner product with $w$, noting that $\ip{w,\mathbf{1}} =0$,
\[\varepsilon\left\langle\frac{dw}{dt},w(t)\right\rangle_\V + \varepsilon\ip{\Delta w(t),w(t)} - \ip{w(t),w(t)} =\ip{\gamma(t)-\beta(t),w(t)}. \]
Consider $(v_i(t)-u_i(t))(\gamma_i(t)-\beta_i(t))$. If $v_i(t)=u_i(t)$ this equals 0, if $v_i(t)> u_i(t)$ then a simple case check gives that therefore $\gamma_i(t) \leq \beta_i(t)$ and likewise if $v_i(t)<u_i(t)$ then $\gamma_i(t) \geq \beta_i(t)$. Hence $\ip{\gamma(t)-\beta(t),w(t)}\leq 0$. Furthermore since $\Delta$ is positive semi-definite we have $\ip{\Delta w(t),w(t)}\geq 0 $. Therefore by the above we have for a.e. $t\in T$,
\[\frac{1}{2}\varepsilon\frac{d}{dt}||w(t)||_\V^2 \leq ||w(t)||^2_\V \] and note that $w(0) = \mathbf{0}$. Hence by Gr\"{o}nwall's differential inequality we have that for all $t\in T$, $||w(t)||_\V^2 \leq 0$. Therefore, for all $t\in T$, $v(t) =u(t)$.

Finally by Theorem \ref{betathm}, since $u = v$ on $T$, at a.e. $t\in T$ (in particular, at $t\in\tilde T$ for some $\tilde T\subseteq T$ with $T\setminus\tilde T$ of zero measure): \[\beta_i(t)-\gamma_i(t)=\begin{cases} \bar\beta(t) -\bar\gamma(t), &\text{if }u_i(t) = 0,\\
0, &\text{if } u_i(t)\in(0,1),\\
\bar\beta(t) -\bar\gamma(t), &\text{if }u_i(t) = 1.
\end{cases}\] Therefore at $t\in\tilde T$, either $\beta(t)-\gamma(t) = (\bar\beta(t) -\bar\gamma(t))\mathbf{1}$ or, if $u_i(t)\in(0,1)$ for some $i\in V$, then taking the average value of both sides we get \[\bar\beta(t) -\bar\gamma(t) = (\bar\beta(t) -\bar\gamma(t))\frac{\mathcal{M}(\chi_{\{i\mid u_i(t)\in\{0,1\}\}})}{\mathcal{M}(\mathbf{1})}\] so $\bar\beta(t) -\bar\gamma(t) = 0$ and hence $\beta(t)=\gamma(t)$.
\end{proof} 
\begin{nb}
There are only $2^{|V|}$ distinct $u$ such that $u_i\in\{0,1\}$ for all $i\in V$. Hence if $\overline {u(0)}\in [0,1]\setminus \{ \bar u\mid u\in\V\text{ and } \forall i\in V, u_i\in\{0,1\}\}$, which is $[0,1]$ minus a finite set of points, then we must have $\beta(t)=\gamma(t)$ for a.e. $t\in T$ \emph{(}since $\overline{u(t)} = \overline{u(0)}$\emph{)}.
\end{nb}
\begin{thm}\label{existence}
Let $T=[0,\infty)$. Then for all  $u_0\in \V_{[0,1]}$ there exists $(u,\beta)\in\V_{[0,1],t\in T}\times\V_{t\in T}$ satisfying \eqref{mACobs} with $u\in H_{loc}^1(T;\V)\cap C^{0}(T;\V)$ and with $ u(0) = u_0$. 
\end{thm}
\begin{proof} 
We prove this as Theorem \ref{SDlimit}, by taking the limit as $\tau\downarrow 0$ of the semi-discrete approximations defined in \eqref{mSDobs}. (We avoid circularity as we do not use this theorem until after we have proved Theorem \ref{SDlimit}.)
\end{proof}
\section{Mass-conserving semi-discrete scheme and link to the MBO scheme}\label{mSD}
%\begin{mydef}[Mass-conserving MBO scheme]\label{mMBOdef}
%	Recall from \emph{\cite{vGGOB}} that the ordinary MBO scheme on a graph can be defined variationally by \[u_{n+1}\in \underset{u\in\V_{[0,1]}}{\argmin}\: \left\langle \mathbf{1}-2e^{-\tau\Delta}u_n,u\right\rangle_\V.\]
%	To define a mass-conserving MBO scheme, we add a mass constraint on the minimiser\emph{:} 
%	\[u_{n+1}\in \underset{\underset{\mathcal{M}(u) = \mathcal{M}(u_n)}{u\in \V_{[0,1]}}}{ \argmin }\: \left\langle \mathbf{1}-2e^{-\tau\Delta}u_n,u\right\rangle_\V.\]
%	Note that we can suppress the now constant $\ip{\mathbf{1},u}$ term. \end{mydef}
\begin{mydef}[Mass-conserving semi-discrete scheme]\label{mSDdef}
Building on the insight from \emph{\cite{Budd}}, we link the mass-conserving AC flow to the mass-conserving MBO scheme by defining the following \emph{mass-conserving semi-discrete scheme:}
\begin{equation}
	\label{mSDobs}
	u_{n+1} -e^{-\tau\Delta}u_n-\lambda u_{n+1}+\lambda\frac{\ip{u_{n+1},\mathbf{1} }}{\ip{\mathbf{1} ,\mathbf{1} }}\mathbf{1} =\lambda\beta_{n+1} -\lambda\frac{\ip{\beta_{n+1},\mathbf{1} }}{\ip{\mathbf{1} ,\mathbf{1} }}  \mathbf{1}
\end{equation} 
for $\beta_{n+1}\in\mathcal{B}(u_{n+1})$, recalling that $\lambda:=\tau/\varepsilon$. Recall that, by \eqref{oldbeta}, since $ \mathcal{B}(u_{n+1})$ is non-empty we must have $u_{n+1} \in \V_{[0,1]}$.  
\end{mydef} 
We check this conserves mass.
\begin{prop}
	For $u_{n+1}$ given by \eqref{mSDobs}, \[\mathcal{M}(u_{n+1}) = \mathcal{M}(u_n).\]
\end{prop}
\begin{proof}
	Taking the mass of both sides of \eqref{mSDobs} and cancelling gives \[\ip{u_{n+1},\mathbf{1} }=\ip{e^{-\tau\Delta}u_n,\mathbf{1} } = \ip{u_n,\mathbf{1} }\]
	with the final equality because $e^{-\tau\Delta}$ is self-adjoint and $e^{-\tau\Delta}\mathbf{1} = \mathbf{1} $.
\end{proof}
We express this scheme variationally, and link to the MBO scheme.
%So following the previous section $u_{n+1}$ is given by the $\V_{[0,1]}$ solutions to \be\begin{split}\forall\eta\in\V_{[0,1]},\: 0&\leq \ip{\beta(u_{n+1}),\eta - u_{n+1}}\\&=\left\langle\frac{\varepsilon}{\tau}\left(u_{n+1} -u_n\right)+\frac{1}{2},\eta-u_{n+1}\right\rangle_\V+\varepsilon\left\langle\nabla u_n,\nabla\eta-\nabla u_{n+1}\right\rangle_\mathcal{E}\end{split}\ee
\begin{thm}[Cf. \text{\cite[Theorem 12]{Budd}}]
\label{obsMMprop}
If $0\leq \tau\leq\varepsilon$ then the solutions to the semi-discrete scheme \eqref{mSDobs} obey
 \begin{equation}\begin{split}
	\label{mACobsMM}
	u_{n+1}\in \underset{\underset{\mathcal{M}(u) = \mathcal{M}(u_n)}{u\in \V_{[0,1]}}}{ \argmin }  &\: \lambda\left\langle u,\mathbf{1} -u\right \rangle_\V +  \left|\left|u-e^{-\tau\Delta}u_n\right|\right|^2_\V\\ & \simeq (1-\lambda)  \left|\left|u\right|\right|^2_\V-2\ip{u,e^{-\tau\Delta}u_n}.
\end{split}\end{equation}
In particular, when $\tau=\varepsilon$ we have 
\begin{equation}\label{mMBO}
	u_{n+1}\in \underset{\underset{\mathcal{M}(u) = \mathcal{M}(u_n)}{u\in \V_{[0,1]}}}{ \argmax }  \: \left\langle u,e^{-\tau\Delta}u_n\right\rangle_\V
\end{equation}
which is equivalent to the mass-conserving MBO scheme as in Definition \ref{mMBOdef}.
\end{thm}
\begin{proof}
Let $u_{n+1}$ solve \eqref{mSDobs}. First, note that $\mathcal{B}(u_{n+1})$ is non-empty and so $u_{n+1}\in\V_{[0,1]}$. Furthermore, we know that $\mathcal{M}(u_{n+1}) = \mathcal{M}(u_n) =: M$.  

Next, expanding out the functional for $\mathcal{M}(u)=M$ gives: 
\[ \begin{split}
	\lambda\left\langle u,\mathbf{1} -u\right \rangle_\V +  \left|\left|u-e^{-\tau\Delta}u_n\right|\right|^2_\V &= \lambda M+(1-\lambda)\left|\left|u\right|\right|^2_\V  -2\ip{u,e^{-\tau\Delta}u_n}+\left|\left|e^{-\tau\Delta}u_n\right|\right|^2_\V\\
 		&\simeq (1-\lambda)\left|\left|u\right|\right|^2_\V -2\ip{u,e^{-\tau\Delta}u_n}.
\end{split}  \]
We seek to prove that for $\lambda\leq 1$ and $\forall\eta\in\V_{[0,1]}$ such that $\ip{\eta,\mathbf{1} }=M=\ip{u_{n+1},\mathbf{1} }$:
%\[ \begin{split}\tau\left\langle u_{n+1},\Delta u_n \right \rangle_\V& + \frac{1}{2}\left\langle u_{n+1},1-u_{n+1} \right \rangle_\V + \frac{1}{2}\ip{u_{n+1}-u_n,u_{n+1}-u_n}\\ &\leq
%\tau\left\langle \eta,\Delta u_n \right \rangle_\V + \frac{1}{2}\left\langle \eta,1-\eta\right \rangle_\V+ \frac{1}{2}\ip{\eta-u_n,\eta-u_n}\end{split}\]
\[
(1-\lambda)\ip{u_{n+1},u_{n+1}}-2\ip{u_{n+1},e^{-\tau\Delta}u_n}\leq (1-\lambda)\ip{\eta,\eta}-2\ip{\eta,e^{-\tau\Delta}u_n}
\]
By rearranging and cancelling this is equivalent to (noting that $\ip{\eta-u_{n+1},\mathbf{1} }=0$) \[ \begin{split}
	0& \leq - \left\langle \eta-u_{n+1},2 e^{-\tau\Delta}u_n \right \rangle_\V + (1-\lambda)\left(\ip{\eta,\eta}-\ip{u_{n+1},u_{n+1}}\right)\\
	&= \left\langle \eta-u_{n+1},-2 e^{-\tau\Delta}u_n +(1-\lambda)(\eta+u_{n+1}) \right \rangle_\V\\
	&= \left\langle \eta-u_{n+1},2(1-\lambda)u_{n+1}-2 e^{-\tau\Delta}u_n +(1-\lambda)(\eta-u_{n+1}) \right \rangle_\V\\
	&= \left\langle \eta-u_{n+1},2\lambda\beta_{n+1}-2\lambda\overline{\beta_{n+1}} \mathbf{1}-2\lambda\overline{u_{n+1}}\mathbf{1}  +(1-\lambda)(\eta-u_{n+1}) \right  \rangle_\V \:\:\text{by \eqref{mSDobs}}\\
	&= 2\lambda\left\langle \eta-u_{n+1},\beta_{n+1}\right \rangle_\V +(1-\lambda)||\eta-u_{n+1}||^2_\V. 
\end{split}  \] 
As $\beta_{n+1}\in\mathcal{B}(u_{n+1})$ and $\eta_i\in [0,1]$: either $(\beta_{n+1})_i=0$, or $(\beta_{n+1})_i\geq 0$ when $\eta_i -(u_{n+1})_i=\eta_i\geq 0$, or $(\beta_{n+1})_i\leq 0$ when $\eta_i -(u_{n+1})_i=\eta_i -1\leq 0$. Thus $\left\langle \eta-u_{n+1},\beta_{n+1}\right \rangle_\V\geq 0$.

Finally, for $\lambda=1$ the quadratic term in \eqref{mACobsMM} cancels and we get the equation \eqref{mMBO}.\end{proof}

\subsection{Solving the variational equations}

Compared to \cite{Budd} the addition of the mass conservation constraint substantially increases the difficulty in solving the equations from Theorem \ref{mACobsMM}. We here employ the techniques of convex optimisation, particularly the Krein--Milman theorem, complementary slackness and strong duality, to help resolve this difficulty.

We consider the set of feasible solutions to \eqref{mACobsMM} and \eqref{mMBO}. \begin{mydef} For a given $M=\mathcal{M}(u_0)$ for some $u_0\in\V_{[0,1]}$, we define the hyperplane $S_M := \left\{ u\in \V\,\middle|\, \ip{u,\mathbf{1}} = M\right\}$. We can visualise this as the plane through $u_0$ with $\V$-normal vector $\mathbf{1}$. Then we write the set of feasible solutions to \eqref{mACobsMM} and \eqref{mMBO} 
\be X:=\V_{[0,1]}\cap S_M.\ee Note that $X$ is compact, and is the intersection of two convex sets, so is  convex. Furthermore, note that $X$ can be described as the set of solutions to the linear inequalities \begin{align*} & \forall i \in V \;\ip{u,\chi_{\{i\}}}\geq 0 \text{ and }\ip{u,\chi_{\{i\}}}\leq d_i^r& \text{ and }& &\ip{u,\mathbf{1}} \geq M \text{ and } \ip{u,\mathbf{1}} \leq M \end{align*}
and thus is said to be a \emph{polyhedral} set.

\end{mydef} 
%Recall the following definition from convex analysis.
\begin{mydef}
	For a convex set $C$, define $x\in C$ to be an \emph{extreme point} of $C$ when \[ \forall y,z\in C, \forall t\in(0,1) \:\:\: \big(x = ty+(1-t)z \Rightarrow y = z = x \big)\] and write $\Ext C$ for the subset of $C$ consisting of all such points.
\end{mydef}
We can then characterise the extreme points of the feasible set.
\begin{prop}\label{extprop}
The set $\Ext X$ of extreme points of $X$ is finite and is given by \[ \Ext X = \left \{u\in X\,\middle|\, \exists i^* \in V\: \forall j\in V\setminus\{i^*\} \: u_j\in\{0,1\} \right\}.\]
\end{prop}
\begin{proof} Since $X$ is polyhedral, $\Ext X$ is finite by a standard result \cite[Corollary 1.3.1]{AJ}. 
Suppose $u\in X$ and $\exists i, j \in V$ such that $i\neq j$ and $u_i,u_j\in(0,1)$. Now for $\delta >0$ let \[\begin{split}v_1&:=u - \delta d_i^{-r}\chi_{\{i\}} + \delta d_j^{-r}\chi_{\{j\}}, \\
	v_2&:= u+\delta d_i^{-r}\chi_{\{i\}} - \delta d_j^{-r}\chi_{\{j\}}.\end{split}\]
Then $\mathcal{M}(v_1) =\mathcal{M}(v_2) =  \mathcal{M}(u) -\delta +\delta = \mathcal{M}(u) = M$ so $v_1,v_2 \in S_M$. And for $\delta < \min\left\{d_i^r u_i,d_i^r(1- u_i),d_j^r u_j,d_j^r (1-u_j)\right \}$ we have $v_1,v_2\in \V_{[0,1]}$. Therefore we have $u = \frac{1}{2}v_1+\frac{1}{2}v_2$ for $v_1,v_2\in X\setminus\{u\}$. Therefore $u\notin\Ext X$.

Now let $u\in \left\{u\in X\,\middle|\, \exists i^* \in V\: \forall j\in V\setminus\{i^*\} \: u_j\in\{0,1\} \right\}$, and suppose $u= tv_1+(1-t)v_2$ for some $v_1,v_2\in X$ and $0<t<1$. As $\Ext( [0,1] )= \{0,1\}$ we have that $u_i = 0$ if and only if  $(v_1)_i=(v_2)_i = 0$ and likewise for $u_i=1$. So $v_1 - v_2 = \theta\chi_{\{i^*\}}$ for some $\theta$, and \[ 0 = \ip{v_1-v_2,\mathbf{1}} = \theta\ip{\chi_{\{i^*\}},\mathbf{1}} = \theta d_{i^*}^r\] and so $\theta =0$, i.e. $v_1 = v_2$. Thus $u= tv_1+(1-t)v_2\Rightarrow v_1=v_2=u$, so $u\in\Ext X$.
\end{proof}
For tidiness, we define some useful notation.
\begin{mydef}\label{Adef}
	For $u\in\V_{[0,1]}$ and $\tau > 0$ define the set \begin{equation}
		 A_{u,\tau} := \{ \alpha \in [0,1] \,|\, \exists i\in V\:(e^{-\tau\Delta}u)_i = \alpha\} 
	\end{equation} with ordering $\alpha_1<\alpha_2<...<\alpha_K$ for the elements of $A_{u,\tau}$, where $K = |A_{u,\tau}|$.
	Define the quantities \begin{equation}
 		a_{u,\tau,\alpha} := \sum_{i:(e^{-\tau\Delta}u)_i = \alpha} d_i^r.
 \end{equation}\end{mydef}
%We note a useful fact.
\begin{prop}\label{Aprop} If $\tau >0$, then $0\in A_{u,\tau}\Rightarrow u=\mathbf{0}$, and $1\in A_{u,\tau} \Rightarrow u=\mathbf{1}$. 
\end{prop}
\begin{proof}
Follows immediately from \cite[Lemma 2.6(d)]{vGGOB}.
\end{proof}
\subsection{The MBO case: $\lambda =1$}
\begin{mydef}\label{Sdef}
Define the set of solutions to \eqref{mMBO}
\begin{equation}
	S_{\tau,u_n}:= \underset{u\in X}{ \argmax }  \: \left\langle u,e^{-\tau\Delta}u_n\right\rangle_\V.
\end{equation}
This is convex as the objective function is linear and $X$ is convex, compact as it is a closed subset of $X$, and non-empty as $X$ is compact so the continuous objective function attains its maxima. \end{mydef}
\begin{prop}
	$S_{\tau,u_n}$ is a face of $X$, i.e. if $u,v\in X$ and $t\in(0,1)$, then \[tu + (1-t)v\in S_{\tau,u_n} \Rightarrow u,v\in S_{\tau,u_n}.\] 
\end{prop}
\begin{proof}
	Let $u,v\in X$, $t\in(0,1)$, and $tu + (1-t)v\in S_{\tau,u_n}$. Then \[t\left\langle u,e^{-\tau\Delta}u_n\right\rangle_\V +(1-t)\left\langle v,e^{-\tau\Delta}u_n\right\rangle_\V = \underset{w\in X}{ \max }  \: \left\langle w,e^{-\tau\Delta}u_n\right\rangle_\V\]
	and so 
\[
t\left\langle u,e^{-\tau\Delta}u_n\right\rangle_\V \geq \underset{w\in X}{ \max }  \: \left\langle w,e^{-\tau\Delta}u_n\right\rangle_\V - (1-t)\, \underset{w\in X}{ \max }  \: \left\langle w,e^{-\tau\Delta}u_n\right\rangle_\V = t \,\underset{w\in X}{ \max }  \: \left\langle w,e^{-\tau\Delta}u_n\right\rangle_\V
\]
and likewise for $\left\langle v,e^{-\tau\Delta}u_n\right\rangle_\V$. Hence 
\[\left\langle u,e^{-\tau\Delta}u_n\right\rangle_\V =\left\langle v,e^{-\tau\Delta}u_n\right\rangle_\V = \underset{w\in X}{ \max }  \: \left\langle w,e^{-\tau\Delta}u_n\right\rangle_\V,\] which is to say that $u,v\in S_{\tau,u_n}$.
\end{proof}
\begin{prop}\label{extSprop}
	The extreme points of $S_{\tau,u_n}$ are given by \[ \Ext S_{\tau,u_n} = S_{\tau,u_n}\cap \Ext X \] and the solutions to \eqref{mMBO} are given by the convex hull of the extremal solutions, i.e. \[ S_{\tau,u_n} = \operatorname{conv}(S_{\tau,u_n}\cap \Ext X ).\]
\end{prop}
\begin{proof}
	Let $u\in S_{\tau,u_n}\cap \Ext X $. Then if $v_1,v_2\in S_{\tau,u_n}\subseteq X$, $t\in (0,1)$ and $u = tv_1 +(1-t)v_2$, then $v_1 = v_2$ since $u\in \Ext X$. So $u\in \Ext S_{\tau,u_n}$.
	
	Next, let $u\in \Ext S_{\tau,u_n}\subseteq S_{\tau,u_n}$. Then if $v_1,v_2\in X$ and $u = tv_1 +(1-t)v_2$, then $v_1,v_2\in S_{\tau,u_n}$ as $S_{\tau,u_n}$ is a face, and so $v_1 = v_2$ since $u\in \Ext S_{\tau,u_n}$. Hence $u\in S_{\tau,u_n}\cap \Ext X $. 
	
	So $\Ext S_{\tau,u_n} = S_{\tau,u_n}\cap \Ext X $, and finally we apply the Krein--Milman Theorem (see e.g. \cite[3.23]{Rudin}), which entails in particular that a finite-dimensional compact convex set is the convex hull of its exteme points.
\end{proof}
\begin{cor} \label{abcor} For $\mathcal{M}(u_0)=M$, there exists a trajectory $u_n$ obeying \eqref{mMBO} such that \[\forall n\in\mathbb{N}, \: u_n\in \Ext X = \left \{u\in X\,\middle|\, \exists i^* \in V\: \forall j\in V\setminus\{i^*\} \: u_j\in\{0,1\} \right\}. \]
\end{cor}
\begin{proof}
Follows immediately from the fact that $S_{\tau,u_n}$ is non-empty, and so $S_{\tau,u_n}\cap \Ext X$ is non-empty as otherwise $S_{\tau,u_n} = \operatorname{conv}(\emptyset) = \emptyset$.
\end{proof}

In \cite{OKMBO}, Van Gennip considers a mass-conserving MBO scheme for minimising the Ohta--Kawasaki functional with a modified graph diffusion, which in the $\gamma =0$ special case reduces to ordinary graph diffusion and hence is the same problem as \eqref{mMBO}. We here repeat his form for the solutions to \eqref{mMBO} lying at extreme points. 
\begin{thm} \label{mMBOcondition}Let $u_{n+1}\in S_{\tau,u_n}\cap\Ext X$.  Then write \[ E:=\{i\in V \,|\,(u_{n+1})_i = 1\}, \:  F:=\{i\in V \,|\,(u_{n+1})_i = 0\}\] Then for each $i\in V\setminus F$,  $j\in V\setminus E$ we have $(e^{-\tau\Delta}u_n)_i\geq (e^{-\tau\Delta}u_n)_j$.\end{thm}
\begin{proof} By Proposition \ref{extprop} we have that $u_{n+1} = \chi_E + \theta \chi_{V\setminus(E\cup F)}$ where $\theta\in(0,1)$ and $V\setminus(E\cup F)$ has at most one element which we will denote $i^*$ (when it exists). Now choose some  $0 <\delta < \min_{i\in V}\{d_i^r,d_{i^*}^r\theta,d_{i^*}^r(1-\theta)\}$, and any $i\in V\setminus F$,  $j\in V\setminus E$.  Define \[ u := u_{n+1} - \delta d_i^{-r}\chi_{\{i\}} + \delta d_j^{-r}\chi_{\{j\}}\] where by choice of $\delta$ we ensure that $u \in X$. Therefore \[ 0\leq \ip{u_{n+1}-u,e^{-\tau\Delta}u_n} = \delta((e^{-\tau\Delta}u_n)_i - (e^{-\tau\Delta}u_n)_j)\] and so $(e^{-\tau\Delta}u_n)_i\geq (e^{-\tau\Delta}u_n)_j$ as desired.
\end{proof}

\subsubsection{Uniqueness conditions for the mass-conserving MBO scheme} 
We consider when \eqref{mMBO} has a unique solution, and characterise all solutions to \eqref{mMBO}.

\begin{cor}\label{uniquecor}$S_{\tau,u_n}$ has one element if and only if $S_{\tau,u_n}\cap \Ext X $ has one element.\end{cor} 
\begin{proof}
	As $S_{\tau,u_n}$ is non-empty, $S_{\tau,u_n}\cap \Ext X$ is non-empty as else $S_{\tau,u_n} = \operatorname{conv}(\emptyset) = \emptyset$.
	Thus, if $S_{\tau,u_n}=\{u\}$ then $S_{\tau,u_n}\cap \Ext X =\{u\}$ as this is the only non-empty subset of $S_{\tau,u_n}$. 
	Conversely, if $S_{\tau,u_n}\cap \Ext X =\{u\}$ then by Proposition \ref{extSprop} $S_{\tau,u_n} =  \operatorname{conv}(\{u\}) = \{u\}$.
\end{proof}
Usefully, Theorem \ref{mMBOcondition}  gives a necessary condition for $u\in S_{\tau,u_n}\cap \Ext X $. We demonstrate the following sufficient condition for uniqueness of solutions.
\begin{thm}
	Define the condition \be\label{diffusedcondition} \forall i, j\in V, \:\:\:\: i\neq j \Rightarrow (e^{-\tau\Delta}u_n)_i \neq  (e^{-\tau\Delta}u_n)_j. \ee
	Then if \eqref{diffusedcondition} holds, $S_{\tau,u_n}$ has a unique element \emph{(}i.e. \eqref{mMBO} has a unique solution\emph{)}.
\end{thm}
\begin{proof}
	WLOG, up to relabelling of $V$, we may write \eqref{diffusedcondition} as \[i<j \Leftrightarrow (e^{-\tau\Delta}u_n)_i < (e^{-\tau\Delta}u_n)_j.\]
	Let $u\in S\cap \Ext X$. By Theorem \ref{mMBOcondition} we thus have \[i<j\Rightarrow u_i = 0 \text{ or } u_j = 1\] and hence by Proposition \ref{extprop} $u$ must have the form \[u = \big(\underbrace{0,0,...,0}_{a-1},\theta,\underbrace{1,1,...,1}_{|V|-a}\big)\] where $\theta\in (0,1]$ so $(a,\theta)$ uniquely determines any element of $S_{\tau,u_n}\cap\Ext X$. Let \[ \mathcal{M}(a,\theta) := \mathcal{M}(u)\text{ for $u$ defined by $(a,\theta)$ as above.}\] 
	Then for $a<b$, \[  \mathcal{M}(a,\theta)-\mathcal{M}(b,\phi) = \theta d_a^r + \sum_{a<i<b} d_i^r + (1-\phi) d_{b}^r > 0\] and clearly $\mathcal{M}(a,\theta) = \mathcal{M}(a,\phi)$ if and only if $\theta = \phi$. If $u\in S_{\tau,u_n}\cap \Ext X$, $\mathcal{M}(u) = M$, and by the above we have that $\mathcal{M}(a,\theta) = M$ for a unique $(a,\theta)$. Thus $S_{\tau,u_n}\cap \Ext X$ has a unique element (as by the proof of Corollary \ref{abcor} $S_{\tau,u_n}\cap \Ext X$ is non-empty), so by Corollary \ref{uniquecor} $S_{\tau,u_n}$ has a unique element.
\end{proof}
%\begin{nb}
%	If \eqref{diffusedcondition} does not hold, then non-unique solutions are possible. In particular, if $u$ is a solution such that $u_i \neq u_j$, $d_i^r = d_j^r$ and $(e^{-\tau\Delta}u_n)_i = (e^{-\tau\Delta}u_n)_j$, then the $v\in X\setminus\{u\}$ obtained from switching the values of $u$ at $i$ and $j$ is also a solution. 
%\end{nb}
Following this idea, we get a characterisation of $S_{\tau,u_n}$ and a necessary and sufficient condition for uniqueness.
\begin{thm}\label{lamb1soln}
	Suppose $u_n\in\V_{[0,1]}$ and $M=\mathcal{M}(u_n)>0$, then there is a unique $k$ such that $1\leq k\leq K$ and \[\sum_{l=k+1}^K a_{u_n,\tau,\alpha_l} < M \leq  \sum_{l=k}^K a_{u_n,\tau,\alpha_l}\] recalling $K$ and $a_{u,\tau,\alpha}$ from Definition \ref{Adef}. Then $u\in S_{\tau,u_n}$ if and only if $u\in X$ and 
	\begin{subequations}\label{mMBOsoln}
		\begin{align}
		&u_i = 0, \text { if } (e^{-\tau\Delta}u_n)_i<\alpha_k,\\
		&u_i = 1, \text { if } (e^{-\tau\Delta}u_n)_i>\alpha_k,\\
		&M - \sum_{l=k+1}^K a_{u_n,\tau,\alpha_l} = \sum_{(e^{-\tau\Delta}u_n)_i=\alpha_k} d_i^r u_i.
	\end{align}
	\end{subequations}
	 Therefore $S_{\tau,u_n}$ has a unique element if and only if \begin{equation}\label{uniquecon}
		M =  \sum_{l=k}^K a_{u_n,\tau,\alpha_l}\:\:\:\text{ or } \:\:\:\exists !i\in V, \:(e^{-\tau\Delta}u_n)_i = \alpha_k .
	\end{equation}  
\end{thm}
\begin{proof}
First, we show that $k$ exists and is unique. Let $B_r:=\sum_{l=r}^K a_{u_n,\tau,\alpha_l}$. Then  as $a_{u_n,\tau,\alpha_l}>0$ the $B_r$ are strictly decreasing in $r$ and we observe that $B_1=\mathcal{M}(\mathbf{1})\geq M$ and $B_{K+1}=0<M$. Hence there exists a unique $k\in\{1,...,K\}$ such that $B_{k+1} < M \leq B_k$.

	Next, for $v\in \V$, define $\tilde v:\{1,...,K\} \rightarrow\mathbb{R}$ by \[\tilde v_l := a^{-1}_{u_n,\tau,\alpha_l}\sum_{i:(e^{-\tau\Delta}u_n)_i=\alpha_l} d_i^r v_i\] and define the inner product \[\langle\tilde v,\tilde w\rangle_\alpha :=\sum_{l=1}^K a_{u_n,\tau,\alpha_l} \tilde v_l\tilde w_l.\] 
	Then note by a simple calculation we have that \[\langle\tilde v,\mathbf{1}\rangle_\alpha = \mathcal{M}(v)\] and \[\langle\tilde v,\widetilde{ e^{-\tau\Delta}u_n}\rangle_\alpha = \ip{v, e^{-\tau\Delta}u_n}.\] Hence, defining $\tilde X = \{\tilde v|v\in X\}$, we have that $u\in S_{\tau,u_n}$ if and only if  
		\[\tilde u \in \underset{\tilde v\in \tilde X}{ \argmax }  \: \left\langle \tilde v,\widetilde{ e^{-\tau\Delta}u_n}\right\rangle_\alpha\]
	and note that \eqref{diffusedcondition} holds true of $\widetilde{ e^{-\tau\Delta}u_n}$ (i.e. $(\widetilde{ e^{-\tau\Delta}u_n})_l\neq (\widetilde{ e^{-\tau\Delta}u_n})_r$ for all $l\neq r \in \{1,2,...,K\}$). Therefore by the same argument as in the proof of the previous theorem \emph{mutatis mutandis} (i.e. replacing instances of $\ip{\cdot,\cdot}$ with $\langle\cdot,\cdot \rangle_\alpha$, of $d_i^r$ with $a_{u_n,\tau,\alpha_l}$ etc.) there is a unique such $\tilde u$ of the form 
	\[\tilde u = \big(\underbrace{0,0,...,0}_{b-1},\theta,\underbrace{1,1,...,1}_{K-b}\big)\] where $\theta\in (0,1]$. Then we have \[ M = \langle\tilde u,\mathbf{1}\rangle_\alpha = \theta a_{u_n,\tau,\alpha_b} + \sum_{l=b+1}^K a_{u_n,\tau,\alpha_l}\] so we must have $b = k$ and \[\theta = a^{-1}_{u_n,\tau,\alpha_k}\left(M - \sum_{l=k+1}^K a_{u_n,\tau,\alpha_l}\right).\]
	 Taking $l<k$, \[0 =\tilde u_l = a^{-1}_{u_n,\tau,\alpha_l}\sum_{i:(e^{-\tau\Delta}u_n)_i=\alpha_l} d_i^r u_i\] and so $u_i = 0$ if $ (e^{-\tau\Delta}u_n)_i<\alpha_k$, and taking $l>k$ \[1 =\tilde u_l = a^{-1}_{u_n,\tau,\alpha_l}\sum_{i:(e^{-\tau\Delta}u_n)_i=\alpha_l} d_i^r u_i\] and so $u_i = 1$ if $ (e^{-\tau\Delta}u_n)_i>\alpha_k$. Finally taking $l=k$ we get the equivalences \[\begin{split} u\in S_{\tau,u_n}&\text{ if and only if } \tilde u \in \underset{\tilde v\in \tilde X}{ \argmax }  \: \left\langle \tilde v,\widetilde{ e^{-\tau\Delta}u_n}\right\rangle_\alpha \\
&\text{ if and only if }\begin{cases}
		u_i = 0, \text { if } (e^{-\tau\Delta}u_n)_i<\alpha_k,\\
		u_i = 1, \text { if } (e^{-\tau\Delta}u_n)_i>\alpha_k,\\
		\theta = a^{-1}_{u_n,\tau,\alpha_k}\sum_{(e^{-\tau\Delta}u_n)_i=\alpha_k} d_i^r u_i.
	\end{cases} \end{split}\]	
	Hence we have a unique solution if and only if $(e^{-\tau\Delta}u_n)_i = \alpha_k$ at a unique $i\in V$ or $\theta =1$ (and therefore $u_i=1$ for $(e^{-\tau\Delta}u_n)_i = \alpha_k$),  i.e. when \eqref{uniquecon} holds.
	\end{proof}
\begin{nb}
If $M=0$ then $X=\{\mathbf{0}\}$, so uniqueness is trivial, hence supposing that $M>0$ incurs no loss of generality.
\end{nb}
\begin{nb}
 The solution in \eqref{mMBOsoln}, with an adjustible threshold level \emph{(}i.e. $\alpha_k$\emph{)} to ensure that mass is conserved, accords with the definition of the mass-conserving graph MBO scheme in \emph{\cite{OKMBO}} and with the definition of the mass-conserving continuum MBO scheme in \emph{\cite{RW}}.
\end{nb}
\subsection{The non-MBO case: $0\leq \lambda <1$}\label{s22}

To solve \eqref{mACobsMM} for $0\leq \lambda<1$, we use duality. Let $M:=\mathcal{M}(u_n)$ and define the functions \begin{align}
	&f_i(u) :=-d_i^r u_i, & g_i(u) := (u_i -1)d_i^r,& &h(u) := 2(\mathcal{M}(u) - M).&
\end{align} Then \eqref{mACobsMM} can be written as the primal problem: \[ \underset{u\in \V}{ \min }  \: (1-\lambda)  \left|\left|u\right|\right|^2_\V-2\ip{u,e^{-\tau\Delta}u_n}\: \text{ s.t. }f_i(u)\leq 0,\:g_i(u)\leq 0, \text{ and }h(u)=0.\] 
Hence for $\xi,\mu\in\V$ and $\nu\in\mathbb{R}$ dual variables, \eqref{mACobsMM} has Lagrangian: \be\label{Lagrangian}\begin{split} %\hspace*{-1.17em} 
&L(u,\xi,\mu,\nu):= (1-\lambda)  \left|\left|u\right|\right|^2_\V-2\ip{u,e^{-\tau\Delta}u_n}+\sum_i\left(\xi_i f_i(u) +\mu_i g_i(u)\right) +\nu h(u)\\
&=(1-\lambda)  \left|\left|u\right|\right|^2_\V-2\ip{u,e^{-\tau\Delta}u_n} +\ip{u,\mu-\xi} +\ip{2\nu u-\mu,\mathbf{1}}-2\nu M.
\end{split}\ee
We can rewrite this by making the following definition:
\be \label{ustar} u^*(\xi,\mu,\nu) := \frac{1}{2(1-\lambda)}\left( 2e^{-\tau\Delta}u_n + \xi-\mu -2\nu\mathbf{1}\right) \ee
so that 
\[\begin{split} L(u,\xi,\mu,\nu)
&=(1-\lambda)  \left|\left|u\right|\right|^2_\V-2(1-\lambda)\left\langle u, u^*(\xi,\mu,\nu)\right\rangle_\V - \ip{\mu,\mathbf{1}} - 2\nu M\\
&=(1-\lambda)  \left|\left|u-u^*(\xi,\mu,\nu)\right|\right|^2_\V-(1-\lambda)||u^*(\xi,\mu,\nu)||^2_V - \ip{\mu,\mathbf{1}} - 2\nu M\\
\end{split}\]
which we note is strictly convex, proper, and bounded below in $u$ (for fixed $\xi,\mu$, and $\nu$). 
Next, we define the dual objective function: \be G(\xi,\mu,\nu):=\inf_{u\in \V}  L(u,\xi,\mu,\nu) = L(u^*(\xi,\mu,\nu),\xi,\mu,\nu).\ee 
and therefore 
\begin{equation}\label{gform}
	G(\xi,\mu,\nu) = - \left( (1-\lambda)\left|\left|u^*(\xi,\mu,\nu)\right|\right|^2_\V +\ip{\mu,\mathbf{1}}+2\nu M \right).
\end{equation}
The dual problem to \eqref{mACobsMM} is given by \be \label{dual} \sup_{\xi\geq 0,\mu\geq 0,\nu} G(\xi,\mu,\nu). \ee
\begin{lem}\label{sduality} For  $u_n\in\V_{[0,1]}$, $M=\mathcal{M}(u_n)$, \eqref{mACobsMM} and \eqref{dual} have strong duality, i.e. 
\[  \sup_{\xi\geq 0,\mu\geq 0,\nu} G(\xi,\mu,\nu) = \underset{u\in X}{ \min }  \: (1-\lambda)  \left|\left|u\right|\right|^2_\V-2\ip{u,e^{-\tau\Delta}u_n}\] 
and if $\xi^*,\mu^*%\geq0
$ and $\nu^*$ optimise \eqref{dual}, then $u^*(\xi^*,\mu^*,\nu^*)\in X$ as in \eqref{ustar} optimises \eqref{mACobsMM}.
 \end{lem} 
\begin{proof}
We apply Slater's condition for strong duality (see \cite[§5.2.3]{convexbook}). As the $f_i$ and $g_i$ are affine on $\V$ and $\V$ is open and affine, Slater's condition is satisfied if $\exists u\in\V$ with $f_i(u)\leq 0,g_i(u)\leq 0$ and $h(u) = 0$, i.e. if $\exists u\in X$. As $u_n\in X$ we thus have strong duality. 

Now let $\xi^*\geq 0,\mu^*\geq0$, and $\nu^*$ be optimal for \eqref{dual}, and let $\tilde u\in X$ be optimal for \eqref{mACobsMM}, which we know exists since $X$ is compact and the objective function is continuous. Writing $f_0(u) := (1-\lambda)  \left|\left|u\right|\right|^2_\V-2\ip{u,e^{-\tau\Delta}u_n}$ we have by strong duality: \[\begin{split}
	f_0(\tilde u) &= G(\xi^*,\mu^*,\nu^*)= L(u^*, \xi^*,\mu^*,\nu^*)= \inf_{u\in \V} L(u, \xi^*,\mu^*,\nu^*)\leq L(\tilde u, \xi^*,\mu^*,\nu^*)\leq f_0(\tilde u) 
\end{split} \] 
where the final inequality holds by \eqref{Lagrangian}, as $\tilde u\in X$ and so $f_i(\tilde u),g_i(\tilde u)\leq 0$ and $h(\tilde u) =0$.  So the inequalities are equalities and $L(u, \xi^*,\mu^*,\nu^*)$ is minimised at $\tilde u$. As $L$ is strictly convex in $u$ it has a unique minimiser, so $u^*(\xi^*,\mu^*,\nu^*)=\tilde u$ is optimal for \eqref{mACobsMM}.  
\end{proof}
%Next, let us consider the exact form of $g$. Writing \[v = (1-\lambda)u^*(\xi,\mu,\nu)= e^{-\tau\Delta}u_n + \frac{1}{2}\xi-\frac{1}{2}\mu -\frac{1}{2}\nu\mathbf{1}\] we have that \begin{equation*}
%	\begin{split}
%	(1-\lambda)g(\xi,\mu,\nu) &= \left|\left|v\right|\right|^2_\V-2\ip{v,e^{-\tau\Delta}u_n} +\ip{v,\mu-\xi} +\nu\ip{ v,\mathbf{1}}-(1-\lambda)\ip{\mu,\mathbf{1}}-(1-\lambda)\nu M 	\\
%	&= \ip{v,v - 2e^{-\tau\Delta}u_n +\mu -\xi +\nu\mathbf{1}} -(1-\lambda)\ip{\mu,\mathbf{1}}-(1-\lambda)\nu M \\
%	&= -\left|\left|v\right|\right|^2_\V -(1-\lambda)\ip{\mu,\mathbf{1}}-(1-\lambda)\nu M 
%	\end{split}
%\end{equation*}
%and therefore 
%\begin{equation}\label{gform}
%	g(\xi,\mu,\nu) = - \left[ (1-\lambda)\left|\left|u^*(\xi,\mu,\nu)\right|\right|^2_\V +\ip{\mu,\mathbf{1}}+\nu M \right]
%\end{equation}
%so the dual problem is equivalent to finding\begin{equation}
%	\label{dual2}
%	-\frac{1}{1-\lambda} \inf_{\xi\geq 0,\mu\geq 0,\nu} \frac{1}{4}\left|\left|2e^{-\tau\Delta}u_n + \xi-\mu -\nu\mathbf{1}\right|\right|^2_\V +(1-\lambda)\ip{\mu,\mathbf{1}}+(1-\lambda)\nu M .
%\end{equation}
By Lemma \ref{sduality} we have that $u^*:=u^*(\xi^*,\mu^*,\nu^*)\in X$ for $(\xi^*,\mu^*,\nu^*)$ dual optimal, and by applying complementary slackness we have that \begin{align*}
	&u_i^* > 0\Rightarrow \xi^*_i = 0, &\text{and}& &u_i^* < 1\Rightarrow \mu^*_i = 0,&\\
	&\xi_i^* > 0\Rightarrow u^*_i = 0, &\text{and}& &\mu_i^* >0 \Rightarrow u^*_i = 1.&
\end{align*}
Thus at each $i\in V$, $\xi_i^*=0$ or $\mu_i^* =0$. So we have the necessary conditions \begin{equation*}
	u_i^*=\begin{cases}
		0&\Rightarrow \:\:\:\mu^*_i = 0,\\%\:\:\:\:\:\:\:\:\:\:\:\:\:\nu^*\geq(e^{-\tau\Delta}u_n)_i \\
		\in(0,1)&\Rightarrow\: \:\:\xi_i^*=\mu^*_i = 0,\\%\:\:\:0<(e^{-\tau\Delta}u_n)_i - \nu^* < 1-\lambda\\
		1&\Rightarrow \:\:\:\xi_i^* = 0.%\:\:\:\:\:\:\:\:\:\:\:\:\:\:\nu^* \leq (e^{-\tau\Delta}u_n)_i - (1-\lambda)
	\end{cases}
\end{equation*} Then by substituting into \eqref{ustar}
\begin{equation*}
	\label{ustarsoln0}
	%\hspace*{-1.9em}
u_i^*=\begin{cases}
		0, &\text{if and only if} \:\:\mu^*_i = 0,\:\:\xi^*_i=2\nu^*-2(e^{-\tau\Delta}u_n)_i\geq0, \\
		\begin{aligned}\frac{(e^{-\tau\Delta}u_n)_i - \nu^*}{1-\lambda}\\\in(0,1),\end{aligned} &\text{if and only if} \:\:\xi_i^*=\mu^*_i = 0,\:\:0<(e^{-\tau\Delta}u_n)_i - \nu^* < 1-\lambda,\\
		1, &\text{if and only if} \:\:\xi_i^* = 0,\:\:\mu_i^*= 2(e^{-\tau\Delta}u_n)_i - 2(1-\lambda) - 2\nu^*\geq 0.
	\end{cases}
\end{equation*} 
We simplify by noting that the $\nu^*$ inequality conditions are disjoint and exhaustive, so we need only consider those conditions (to see this, note that if for example $\nu^*\geq (e^{-\tau\Delta}u_n)_i$ then each of the $u_i^*>0$ cases are ruled out, so $u_i^*$ must equal zero): \begin{equation}
	\label{ustarsoln}
	%\hspace*{-1.9em}
u_i^*=\begin{cases}
		0, &\text{if and only if} \:\: \nu^*-(e^{-\tau\Delta}u_n)_i\geq0, \\
		\frac{(e^{-\tau\Delta}u_n)_i - \nu^*}{1-\lambda}\in (0,1), &\text{if and only if} \:\: 0<(e^{-\tau\Delta}u_n)_i - \nu^* < 1-\lambda,\\
		1, &\text{if and only if} \:\: \nu^*\leq (e^{-\tau\Delta}u_n)_i - (1-\lambda).
	\end{cases}
\end{equation}  
But by the above lemma $u^*\in X$, so we have $\mathcal{M}(u^*) = M$. Thus $\nu=\nu^*$ is a solution to:
\begin{equation}\label{nusoln} 0 = M + \sum_i d_i^r \begin{cases}
	-1, & \nu \leq (e^{-\tau\Delta}u_n)_i -(1-\lambda),\\
	\frac{\nu-(e^{-\tau\Delta}u_n)_i}{1-\lambda}, & (e^{-\tau\Delta}u_n)_i-(1-\lambda)<\nu < (e^{-\tau\Delta}u_n)_i,\\
	0, &\nu\geq (e^{-\tau\Delta}u_n)_i,
\end{cases}\end{equation}
which exists by the Intermediate Value Theorem. By Definition \ref{Adef}, we rewrite \eqref{nusoln}% as:
\begin{equation}\label{nusoln2} M = \sum_{\alpha \in A_{u_n,\tau}} a_{u_n,\tau,\alpha} \begin{cases}
	1, & \nu \leq \alpha -(1-\lambda),\\
	\frac{\alpha-\nu}{1-\lambda}, & \alpha-(1-\lambda)<\nu < \alpha,\\
	0, &\nu\geq \alpha.
\end{cases}\end{equation}
\begin{nb} 
Although $u^*$ is unique, $\nu^*$ is not in general unique. For example if $A_{u_n,\tau}=\{0\}$ \emph{(}i.e. $u_n=\mathbf{0}$\emph{)} then any $\nu\geq 0$ solves \eqref{nusoln2}, but by the same token in that case any $\nu\geq 0$ gives $u^*=\mathbf{0}$.   
In general, the right hand side of \eqref{nusoln2} is constant in $\nu$ for $\nu\in [\alpha_k, \alpha_{k+1}-(1-\lambda)]$, where $\alpha_k,\alpha_{k+1}$ are consecutive elements in $A_{u_n,\tau}$.\end{nb}
\begin{prop}\label{lambdaprop}
Let $u_n\in\V_{[0,1]}$, $M=\mathcal{M}(u_n)$, and suppose $0<M<\ip{\mathbf{1},\mathbf{1}}$ and $\tau >0$. If $\nu$ solves \eqref{nusoln2}, then $\nu \in [\lambda\min A_{u_n,\tau},\lambda\max A_{u_n,\tau}]\subseteq(0,\lambda)$.
\end{prop}
\begin{proof} By Propostion \ref{Aprop} and the condition on $M$, note that $A_{u_n,\tau}\subseteq(0,1)$.
Since diffusion preserves mass, $M = \mathcal{M}(e^{-\tau\Delta}u_n)$ and therefore \[M = \sum_{\alpha \in A_{u_n,\tau}} a_{u_n,\tau,\alpha} \alpha \]
and so we have by \eqref{nusoln2}:
\be
\label{nusoln3}
0=\sum_{\alpha \in A_{u_n,\tau}} a_{u_n,\tau,\alpha} \begin{cases}
	1-\alpha, & \nu \leq \alpha -(1-\lambda),\\
	\frac{\alpha-\nu}{1-\lambda}-\alpha, & \alpha-(1-\lambda)<\nu < \alpha,\\
	-\alpha, &\nu\geq \alpha,
\end{cases}\ee
i.e., $\nu$ is a solution to 
\[\begin{split} 0 = \sum_{\alpha \in [1-\lambda +\nu,1)\cap A_{u_n,\tau}} a_{u_n,\tau,\alpha} (1-\alpha)&+ \sum_{\alpha \in (\nu,1-\lambda +\nu) \cap A_{u_n,\tau}} a_{u_n,\tau,\alpha} \frac{\alpha\lambda-\nu}{1-\lambda} \\&+ \sum_{\alpha \in (0,\nu] \cap A_{u_n,\tau} } a_{u_n,\tau,\alpha} (-\alpha). \end{split}\]
First, suppose that $\nu < \lambda\min A_{u_n,\tau}< \min A_{u_n,\tau}$. Then \[ \sum_{\alpha \in (0,\nu] \cap A_{u_n,\tau} } a_{u_n,\tau,\alpha} (-\alpha) =0\] and $\alpha\lambda -\nu >\lambda(\alpha-\min A_{u_n,\tau})\geq 0$ for $\alpha \in A_{u_n,\tau}$ so
\[\sum_{\alpha \in [1-\lambda +\nu,1)\cap A_{u_n,\tau}} a_{u_n,\tau,\alpha} (1-\alpha)+ \sum_{\alpha \in (\nu,1-\lambda +\nu) \cap A_{u_n,\tau}} a_{u_n,\tau,\alpha} \frac{\alpha\lambda-\nu}{1-\lambda}  > 0\] hence $\nu$ does not solve \eqref{nusoln3}.
%with equality if and only if $A = \{1\}$, i.e. if and only if $M=\ip{\mathbf{1},\mathbf{1}}$. 
Next, suppose that $\nu >\lambda\max A_{u_n,\tau}$. Then we have $\max A_{u_n,\tau} = (1-\lambda)\max A_{u_n,\tau} + \lambda \max A_{u_n,\tau} < 1-\lambda + \nu$ so \[\sum_{\alpha \in [1-\lambda +\nu,1)\cap A_{u_n,\tau}} a_{u_n,\tau,\alpha} (1-\alpha)=0\] and $\alpha\lambda -\nu <\lambda(\alpha-\max A_{u_n,\tau})\leq 0$ for $\alpha \in A_{u_n,\tau}$ so
\[ \sum_{\alpha \in (\nu,1-\lambda +\nu) \cap A_{u_n,\tau}} a_{u_n,\tau,\alpha} \frac{\alpha\lambda-\nu}{1-\lambda} + \sum_{\alpha \in (0,\nu] \cap A_{u_n,\tau} } a_{u_n,\tau,\alpha} (-\alpha) < 0.\]
%with equality if and only if $A = \{0\}$, i.e. if and only if $M=0$. 
Thus if $\nu$ solves \eqref{nusoln3} we must have $\nu\in  [\lambda\min A_{u_n,\tau},\lambda\max A_{u_n,\tau}]$.
\end{proof}
\begin{nb}
If $M=0$ then $u^*=\mathbf{0}=u_n$, which is satisfied if and only if $\nu \geq  (e^{-\tau\Delta}u_n)_i = 0$. 
If $M=\ip{\mathbf{1},\mathbf{1}}$ then $u^*=\mathbf{1}=u_n$, which is satisfied if and only if $\nu \leq  (e^{-\tau\Delta}u_n)_i - 1 +\lambda = \lambda$. Hence we can always assume $\nu$ to lie in $[0,\lambda]$.  
\end{nb}
\subsection{Behaviour as $\lambda\uparrow 1$}
Usefully, for $\lambda<1$ \eqref{mACobsMM} is strictly convex, so it has a unique solution $u_{n+1}^\lambda$. In this section we show that as $\lambda\uparrow 1$ these solutions converge, yielding a choice function for solutions of \eqref{mMBO}.
By the discussion in section \ref{s22} we have the following theorem.
\begin{thm}\label{lamble1soln}
	For $0\leq\lambda<1$, \eqref{mACobsMM} has a unique solution 
	\begin{equation}\label{usoln1}\hspace{-0.1em}
	(u^\lambda _{n+1})_i=\begin{cases}
		0, &\text{if and only if} \:\: \nu\geq (e^{-\tau\Delta}u_n)_i, \\
		\frac{(e^{-\tau\Delta}u_n)_i - \nu}{1-\lambda}, &\text{if and only if} \:\:(e^{-\tau\Delta}u_n)_i - (1-\lambda)<\nu<(e^{-\tau\Delta}u_n)_i,\\
		1, &\text{if and only if} \:\: \nu \leq (e^{-\tau\Delta}u_n)_i - (1-\lambda),
	\end{cases}
\end{equation} 
where $\nu$ is a solution to \eqref{nusoln2} and hence $\nu\in  [0,\lambda]$.
\end{thm}
As a prelude to investigating the convergence properties of $u^\lambda_{n+1}$, we first show that convergence of solutions of \eqref{mACobsMM} as $\lambda\uparrow 1$ is relevant to solving \eqref{mMBO}.
\begin{thm}\label{GC}
	Fix $u_n$ and denote the objective function in \eqref{mACobsMM} by\emph{:} \[f_\lambda :u\mapsto (1-\lambda)\left|\left|u\right|\right|^2_\V -2\ip{u,e^{-\tau\Delta}u_n}.\] Then as $\lambda \uparrow 1$, $f_\lambda\rightarrow f_1$ uniformly on $X$, and note that $f_1$ is equivalent to the objective function in \eqref{mMBO}. Furthermore, if $(u^\lambda)\in X$ solve \eqref{mACobsMM} and $u^\lambda\rightarrow u$ as $\lambda\uparrow 1$, then $u\in X$ is a solution to \eqref{mMBO}.  
\end{thm} 
\begin{proof}
For any $u\in X$ and $\lambda \leq 1$, 
\[|f_\lambda(u) - f_1(u)| = (1-\lambda)||u||_\V^2 \leq (1-\lambda)||\mathbf{1}||_\V^2\]
which tends to zero uniformly as $\lambda\uparrow 1$. 

Next, suppose $u^\lambda\rightarrow u$  as above. Then $u\in X$ since $X$ is closed. By uniform convergence, for all $\varepsilon > 0$ we have some $\delta> 0$ such that for all $\lambda\in(1-\delta,1)$ and all $v\in X$ \[ |f_\lambda(v)-f_1(v)|\leq \varepsilon/2.\]
Therefore since the $u^\lambda$ minimise $f_\lambda$, for any $v\in X$ we have
\[ f_1(u^\lambda) -\varepsilon/2 \leq f_\lambda(u^\lambda) \leq  f_\lambda(v) \leq f_1(v) + \varepsilon/2.\] 
Since $f_1$ is continuous we can take $\lambda\uparrow 1$ and rearrange to get 
\[ f_1(u) \leq f_1(v) +\varepsilon \]
and since $\varepsilon$ was arbitrary we must have that $u$ is a minimiser of $f_1$.
%	As $\Gamma$-convergence is stable with respect to continuous perturbations, it is sufficient to show that in $X$
%	\[ 
%	\Glim_{\lambda\uparrow 1}\, (1-\lambda)\left|\left|u\right|\right|^2_\V = 0.
%	\]
%	Since $(1-\lambda)\left|\left|u\right|\right|^2_\V \geq 0$ on $X$, for all $\lambda\leq1$, the lim-inf condition immediately follows: 
%	\[ \forall (\lambda_m)\uparrow 1\: \forall \text{convergent }(u_m)\in X \: \:\: \:\: \: 0\leq \liminf_{m\rightarrow\infty} (1-\lambda_m)\left|\left|u_m\right|\right|^2_\V.\]
%	Next, for any $u\in X$ and $\lambda_m \uparrow 1$ take recovery sequence $u_m \equiv u$. Then \[ \lim_{m\rightarrow\infty} \: (1-\lambda_m)\left|\left|u_m\right|\right|^2_\V = 0\] as desired. This proves the $\Gamma$-convergence. 
%	
%	The second claim follows from the fundamental theorem of $\Gamma$-convergence, so long as the $(f_\lambda)$ are equi-coercive on $X$ and the $(u^\lambda)$ are precompact in $X$. But since $X$ is compact (as it is closed, bounded and finite-dimensional) these properties are trivial.  
\end{proof}
%We now proceed to the question of the convergence of $u^\lambda_{n+1}$. 
\begin{thm}\label{sdMBOconv}
	Suppose $M=\mathcal{M}(u_n)\in(0,\mathcal{M}(\mathbf{1}))$, and take $k$ as in Theorem~\ref{lamb1soln} with \be\label{kdef}\sum_{l=k+1}^K a_{u_n,\tau,\alpha_l} < M \leq  \sum_{l=k}^K a_{u_n,\tau,\alpha_l}.\ee
	Then for some sufficiently small $\delta>0$, depending only on $e^{-\tau\Delta} u_n$, and each $\lambda\in(1-\delta,1)$ \begin{equation}
		\label{usolnlim}\hspace{-0.5em}(u^\lambda _{n+1})_i=\begin{cases}
		0, &\text{if and only if} \:\: (e^{-\tau\Delta}u_n)_i\leq \alpha_{k-1}, \\
		a_{u_n,\tau,\alpha_k}^{-1}\left(M - \sum_{l=k+1}^K a_{u_n,\tau,\alpha_l}\right) , &\text{if and only if} \:\:(e^{-\tau\Delta}u_n)_i =\alpha_k,\\
		1, &\text{if and only if} \:\: (e^{-\tau\Delta}u_n)_i \geq \alpha_{k+1},
	\end{cases}
	\end{equation}
	and thus $u^\lambda_{n+1}$ converges to the RHS of \eqref{usolnlim} as $\lambda\uparrow 1$.
	\end{thm}
\begin{proof}As $A_{u_n,\tau}$ is a finite set, we can take $\delta>0$ sufficiently small so that the $\delta$-balls around the $\alpha\in A_{u_n,\tau}$ are disjoint. Let $\lambda \in( 1-\delta,1)$ and choose $\nu$ solving \eqref{nusoln2}. Then  by Proposition \ref{lambdaprop}, $\nu\in(0,\lambda)$, by \eqref{nusoln2} we have 
\[M = \sum_{\alpha \in A_{u_n,\tau}} a_{u_n,\tau,\alpha} \begin{cases}
	1, & \nu \leq \alpha -(1-\lambda),\\
	\frac{\alpha-\nu}{1-\lambda}, & \alpha-(1-\lambda)<\nu < \alpha,\\
	0, &\nu\geq \alpha,
\end{cases}\]
and by choice of $\delta$, $\nu$ is within $1-\lambda$ of at most one $\alpha$.
Let $\alpha_0 := 0$ and $\alpha_{K+1}:=1$. Then there exists $  1\leq m \leq K$ such that $\nu\in(\alpha_{m-1},\alpha_{m+1}-(1-\lambda))$, since these intervals cover $(0,\lambda)$, and we have 
\[M = \sum_{l=m+1}^K a_{u_n,\tau,\alpha_l} +a_{u_n,\tau,\alpha_m} \max\left\{\min\left\{\frac{\alpha_m-\nu}{1-\lambda},1\right\},0\right\}. \] 
Hence by \eqref{kdef} we must have either $m=k$ if $\nu<\alpha_m$ or $m = k - 1$ if $\nu\geq\alpha_m$. 
If $\nu<\alpha_m$,	\[\frac{\alpha_k-\nu}{1-\lambda} = a_{u_n,\tau,\alpha_k}^{-1}\left(M - \sum_{l=k+1}^K a_{u_n,\tau,\alpha_l}\right) \] which by \eqref{usoln1} gives \eqref{usolnlim}.  
If $\nu\in[\alpha_m,\alpha_{m+1}-(1-\lambda))$ then by \eqref{usoln1} and since $m = k-1$ 
\[
(u^\lambda _{n+1})_i=\begin{cases} 0, &\text{if and only if }  (e^{-\tau\Delta}u_n)_i\leq \alpha_{m}=\alpha_{k-1}, \\
		1, &\text{if and only if }  (e^{-\tau\Delta}u_n)_i \geq \alpha_{m+1}=\alpha_k.
	\end{cases}
\]
Therefore
\[M =  \sum_{l=k}^K a_{u_n,\tau,\alpha_l}, \]
so it follows that 
\[  a_{u_n,\tau,\alpha_k}^{-1}\left(M - \sum_{l=k+1}^K a_{u_n,\tau,\alpha_l}\right)= 1\]
and so \eqref{usolnlim} follows.
\end{proof}
\begin{nb}
If $M = 0$ or $M = \mathcal{M}(\mathbf{1})$ then the $u_{n+1}^\lambda$ can only be $\mathbf{0}$ or only be $\mathbf{1}$, respectively, and the convergence is trivial, so the supposition on $M$ incurs no loss of generality.
\end{nb}
\begin{nb}
The RHS of \eqref{usolnlim} can immediately be seen to solve \eqref{mMBO} as it satisfies the conditions of \eqref{mMBOsoln}. Furthermore, note that $u^\lambda_{n+1}$ converges to a point in $\Ext X$ \emph{(}i.e. the RHS of \eqref{usolnlim} is in $\Ext X$\emph{)} if and only if \eqref{uniquecon} holds, i.e. if and only if it converges to the unique solution of \eqref{mMBO}.
\end{nb}

\subsection{The converse of Theorem \ref{obsMMprop}}
In this section we prove the following theorem.
\begin{thm}\label{obsMMconverse}
If $u=u_{n+1}$ solves \eqref{mACobsMM}, then $\exists \beta\in\mathcal{B}(u)$ \emph{(}given by \eqref{betasoln1} when $\lambda =1$ and \eqref{betasoln2} when $0\leq \lambda <1$\emph{)}, such that $(u,\beta)$ is a solution to \eqref{mSDobs} \emph{(}for $\beta$ as $\beta_{n+1}$\emph{)}. \end{thm} 
\begin{nb} If $(u,\beta)$ and $(u,\beta')$ solve \eqref{mSDobs} then rearranging we get \[\beta -\beta' =\bar\beta \mathbf{1}-\bar{\beta'} \mathbf{1}\] i.e. $\beta$ and $\beta'$ differ only by a multiple of $\mathbf{1}$. So, for a given $u$ and $\beta\in\mathcal{B}(u)$, $(u,\beta)$ is a solution if and only if $(u,\beta')$ is a solution for all and only the $\beta'\in\{\beta + \theta\mathbf{1}\mid\theta\in\mathbb{R}\}\cap\mathcal{B}(u)$. If $u_i\in(0,1)$ for an $i\in V$ and $(u,\beta)$ and $(u,\beta')$ solve \eqref{mSDobs}, then $\beta = \beta'$ as $\beta_i=\beta'_i=0$. % by the definition of $\mathcal{B}(u)$.
\end{nb}
\subsubsection{$\lambda =1$}
If $M=0$ then $u=u_n=\mathbf{0}$, is trivially a solution to \eqref{mSDobs}, for e.g. $\beta = \mathbf{0}$, hence WLOG we can suppose $M=\mathcal{M}(u_n)>0$. Let $k$ be as in Theorem~\ref{lamb1soln}, such that \[\sum_{l=k+1}^K a_{u_n,\tau,\alpha_l} < M \leq  \sum_{l=k}^K a_{u_n,\tau,\alpha_l}.\] Then, recalling Theorem \ref{lamb1soln}, any solution $u$ to \eqref{mACobsMM} for $\lambda =1$ must satisfy 
		\begin{align*}
		&u_i = 0, \text { if } (e^{-\tau\Delta}u_n)_i<\alpha_k,\\
		&u_i = 1, \text { if } (e^{-\tau\Delta}u_n)_i>\alpha_k,\\
		&M - \sum_{l=k+1}^K a_{u_n,\tau,\alpha_l} = \sum_{(e^{-\tau\Delta}u_n)_i=\alpha_k} d_i^r u_i.
	\end{align*}
For $\lambda = 1$, \eqref{mSDobs} becomes 
\be\label{mSDobs1} -e^{-\tau\Delta}u_n +\frac{M}{\ip{\mathbf{1},\mathbf{1}}}\mathbf{1}=\beta-\frac{\ip{\beta,\mathbf{1} }}{\ip{\mathbf{1} ,\mathbf{1} }}  \mathbf{1}.\ee 
We seek to find a $\beta$ such that $\beta_i =0$ if $u_i\in(0,1)$. Note that if $u_i\in(0,1)$, then by Theorem \ref{lamb1soln} we have $(e^{-\tau\Delta}u_n)_i=\alpha_k$, so we desire to have 
\[-\alpha_k +\frac{M}{\ip{\mathbf{1},\mathbf{1}}}=-\frac{\ip{\beta,\mathbf{1} }}{\ip{\mathbf{1} ,\mathbf{1} }}.\]
Therefore substituting into \eqref{mSDobs1} we have candidate solution: \be\label{betasoln1}\beta = \alpha_k\mathbf{1}-e^{-\tau\Delta}u_n.\ee
We now verify that this candidate solution works even for binary $u$.
\begin{proof}[Proof of Theorem \ref{obsMMconverse} for  $\lambda = 1$. ]
We check that the $\beta$ as in \eqref{betasoln1} solves \eqref{mSDobs1}:
\[ -e^{-\tau\Delta}u_n +\frac{M}{\ip{\mathbf{1},\mathbf{1}}}\mathbf{1}=\alpha_k\mathbf{1}-e^{-\tau\Delta}u_n - \alpha_k\mathbf{1}+\frac{\ip{e^{-\tau\Delta}u_n,\mathbf{1} }}{\ip{\mathbf{1} ,\mathbf{1} }}  \mathbf{1}.\]
Moreover, by the form for $u$ from Theorem \ref{lamb1soln} it follows that $\beta\in\mathcal{B}(u)$.  %\qed
 \end{proof}
\subsubsection{$0\leq\lambda <1$}
%
%[[Let $S_1$ be the set of solutions to \eqref{mSDobs}, and $S_2$ the set of solutions to \eqref{mACobsMM}. From Theorem \ref{obsMMprop} $S_1\subseteq S_2$.  But for $\lambda<1$  \eqref{mACobsMM} is a strictly convex problem so has a unique solution $u^*$. Hence  $S_1\subseteq \{u^*\}$, so either $S_1=\emptyset$ or $S_1 = \{u^*\} = S_2$ as desired. Can we demonstrate existence non constructively?]]

For $0\leq\lambda<1$,  \eqref{mACobsMM} is strictly convex, so recalling \eqref{usoln1} it has unique solution \begin{equation*}
	u_i=\begin{cases}
		0, &\text{if and only if} \:\: \nu\geq (e^{-\tau\Delta}u_n)_i, \\
		\frac{(e^{-\tau\Delta}u_n)_i - \nu}{1-\lambda}, &\text{if and only if} \:\:(e^{-\tau\Delta}u_n)_i - (1-\lambda)<\nu<(e^{-\tau\Delta}u_n)_i,\\
		1, &\text{if and only if} \:\: \nu \leq (e^{-\tau\Delta}u_n)_i - (1-\lambda),
	\end{cases}
\end{equation*} 
where $\nu\in [0,\lambda]$ solving \eqref{nusoln2} is such that $\bar u = \overline{u_n}$. 
Hence \eqref{mSDobs} is satisfied if and only if for all $i\in V$ \be \label{mSDobs2} \lambda \beta_i -\lambda \bar\beta = \lambda\bar u + \begin{cases}
		-(e^{-\tau\Delta}u_n)_i, &\text{if} \:\: \nu\geq (e^{-\tau\Delta}u_n)_i, \\
		- \nu, &\text{if} \:\:(e^{-\tau\Delta}u_n)_i - (1-\lambda)<\nu<(e^{-\tau\Delta}u_n)_i,\\
		1-\lambda-(e^{-\tau\Delta}u_n)_i, &\text{if } \:\: \nu \leq (e^{-\tau\Delta}u_n)_i - (1-\lambda).
	\end{cases} 
%\\&=  \frac{\lambda M}{\ip{\mathbf{1},\mathbf{1}}}-\max\left\{\min\left\{(e^{-\tau\Delta}u_n)_i,\nu\right\},(e^{-\tau\Delta}u_n)_i-1 +\lambda\right\}
\ee
%Note as above we get that if $\beta$ and $\beta '$ are solutions then $\beta$ and $\beta'$ differ only by a multiple of $\mathbf{1}$, and conversely that if $\beta$ is a solution then so is $\beta + \theta\mathbf{1}$ for any $\theta\in\mathbb{R}$.

We seek a $\beta$ solving this with $\beta_i =0$ if $u_i\in(0,1)$. Suppose $\exists i\in V$ for which $u_i\in(0,1)$. This occurs when $(e^{-\tau\Delta}u_n)_i - (1-\lambda)<\nu<(e^{-\tau\Delta}u_n)_i$, and so at this $i$: \[ -\lambda\bar\beta = \lambda\bar u -\nu.\] Plugging into \eqref{mSDobs2} we get the candidate solution:
\be\label{betasoln2} \beta_i  = \lambda^{-1}\begin{cases}
		\nu-(e^{-\tau\Delta}u_n)_i, &\text{if} \:\: \nu\geq (e^{-\tau\Delta}u_n)_i, \\
		0, &\text{if} \:\:(e^{-\tau\Delta}u_n)_i - (1-\lambda)<\nu<(e^{-\tau\Delta}u_n)_i,\\
		\nu-(e^{-\tau\Delta}u_n)_i+1-\lambda, &\text{if } \:\: \nu \leq (e^{-\tau\Delta}u_n)_i - (1-\lambda),
	\end{cases} \ee
which obeys $\beta\in\mathcal{B}(u)$ since $u$ obeys \eqref{usoln1}.
\begin{proof}[Proof of Theorem \ref{obsMMconverse} for  $0\leq\lambda <1$. ]
By the above discussion, taking $u$ as in \eqref{usoln1} and $\beta$ as in \eqref{betasoln2} entails that $(u,\beta)$ is a solution to \eqref{mSDobs} \emph{if} $\exists i\in V$ with $u_i\in(0,1)$. We check the alternative case,
i.e. for all $i\in V$, $u_i\in \{0,1\}$. Take $\beta$ as in \eqref{betasoln2}. As $u$ is binary, either $\nu\geq (e^{-\tau\Delta}u_n)_i$ or $\nu\leq(e^{-\tau\Delta}u_n)_i - (1-\lambda)$ at each $i\in V$, so %by \eqref{betasoln2}
\[\lambda\beta = \nu\mathbf{1} - e^{-\tau\Delta}u_n + (1-\lambda)\chi_{\{i|(e^{-\tau\Delta}u_n)_i\geq \nu + 1 -\lambda \}}.\] 
But as $u$ is binary we have $u = \chi_{\{i|(e^{-\tau\Delta}u_n)_i\geq \nu + 1 -\lambda \}}$ and so 
\[\lambda\bar\beta= \nu -\bar u+ (1-\lambda)\bar u = \nu -\lambda\bar u.\]
Thus $\beta$ solves \eqref{mSDobs2}. Therefore $(u,\beta)$ is always a solution to \eqref{mSDobs}. \end{proof}
\subsection{A Lyapunov functional for the mass-conserving semi-discrete scheme}
In this section we show that the  Lyapunov functional for the semi-discrete scheme derived in \cite{Budd} is also a Lyapunov functional for the mass-conserving semi-discrete scheme. We then use this functional to examine the eventual behaviour of the scheme, extending the analysis in \cite{Budd} by accounting for the complications that arise due to the mass conservation condition. All results in this section assume only that the initial condition $u_0\in\V_{[0,1]}$, and are otherwise independent of the initial condition. 

Recall from \cite{vGGOB} the Lyapunov functional for the oridinary MBO scheme, i.e. the strictly concave functional $J:\V\rightarrow\mathbb{R}$\begin{equation*}
	J(u) := \ip{\mathbf{1}-u,e^{-\tau\Delta}u}
\end{equation*}
with first variation at $u$, $L_u:\V\rightarrow\mathbb{R}$\[ L_u(v) := \left\langle v,\mathbf{1}-2e^{-\tau\Delta}u\right \rangle_\V.\] 
\begin{thm}[Cf. \text{\cite[Theorem 14]{Budd}}]\label{Lyapthm}
	When $0\leq\lambda\leq 1$, the functional \emph{(}on $\V_{[0,1]}$\emph{)}
	\begin{equation}\label{Lyap}
	H(u):= J(u) + (\lambda-1)\ip{u,\mathbf{1}-u} = \lambda\ip{u,\mathbf{1}-u} +\left\langle u,\left(I-e^{-\tau\Delta}\right)u\right\rangle_\V
	\end{equation} 
	is non-negative, and furthermore the functional is a Lyapunov functional for \eqref{mSDobs}, in the sense that $H(u_{n+1}) \leq H(u_n)$ with equality if and only if $u_{n+1}=u_n$ for the sequence of $u_n\in\V_{[0,1]}$ defined by \eqref{mSDobs}. In particular, we have that \begin{equation}\label{Hstep}
		H(u_n)-H(u_{n+1}) \geq (1-\lambda)\left|\left|u_{n+1}-u_n\right|\right|^2_\V.
	\end{equation} 
\end{thm}
\begin{proof}
	Note that $I-e^{-\tau\Delta}$ has eigenvalues $1- e^{-\tau\lambda_k}\geq 0$, since the eigenvalues $\lambda_k$ of $\Delta$ are non-negative, and so $\left\langle u,\left(I-e^{-\tau\Delta}\right)u\right\rangle_\V
\geq 0$. As $u\in\V_{[0,1]}$, $H(u)\geq 0$ follows. 
	
	Next by the concavity of $J$ and linearity of $L_{u_n}$, recalling that $ \ip{u_{n} - u_{n+1},\mathbf{1}}=0$:
	\[\begin{split}
	H(u_{n}) - &H(u_{n+1}) = J(u_{n}) - J(u_{n+1}) + (1-\lambda)\ip{u_{n+1},\mathbf{1}-u_{n+1}} - (1-\lambda)\ip{u_n,\mathbf{1}-u_n} \\
	&\geq L_{u_n}(u_{n} - u_{n+1}) - (1-\lambda)\ip{u_{n+1},u_{n+1}} + (1-\lambda)\ip{u_n,u_n}\\
	&= \ip{u_{n} - u_{n+1},\mathbf{1}-2e^{-\tau\Delta}u_n} - (1-\lambda)\ip{u_{n+1},u_{n+1}} + (1-\lambda)\ip{u_n,u_n}\\
	&= \ip{u_{n} - u_{n+1},-2e^{-\tau\Delta}u_n + (1-\lambda)(u_{n+1}+u_n)}\\
	&= \ip{u_{n} - u_{n+1},2 (1-\lambda)u_{n+1}-2e^{-\tau\Delta}u_n + (1-\lambda)(u_n-u_{n+1})}\\
	&=\left\langle u_{n} - u_{n+1},2\lambda\beta_{n+1}-2\lambda\overline{u_{n+1}}\mathbf{1} -2\lambda\overline{\beta_{n+1}}  \mathbf{1}+ (1-\lambda)(u_n-u_{n+1})\right\rangle_\V \text{ by \eqref{mSDobs}}\\
	&=\ip{u_{n} - u_{n+1},2\lambda\beta_{n+1}}+(1-\lambda)\left|\left|u_{n+1}-u_n\right|\right|^2_\V \\
	&\geq (1-\lambda)\left|\left|u_{n+1}-u_n\right|\right|^2_\V
	\end{split}
	\]
where the final line follows from $\beta_{n+1}\in\mathcal{B}(u_{n+1})$ as in the proof of Theorem \ref{obsMMprop}. Note that if $u_{n+1}\neq u_n$ then $J(u_{n}) - J(u_{n+1}) > L_{u_n}(u_{n} - u_{n+1})$ by strict concavity, so even for $\lambda=1$ there is equality if and only if $u_{n+1}= u_n$.
\end{proof}
\begin{cor}[Cf. $\text{\cite[Lemma 5.18]{OKMBO}}$]\label{cor0}
	Recall from Definition \ref{Sdef} the notation $S_{\tau,u_n}$ for the set of valid MBO updates of $u_n$, i.e. the set of solutions to \eqref{mMBO}. For $\lambda = 1$, if an MBO sequence $u_n$ defined by \eqref{mMBO} satisfies either\emph{:}
\begin{enumerate}[\em (i)]
\item for eventually all $n$, $u_{n+1}\in \Ext S_{\tau,u_n}$, or
\item for eventually all $n$, $u_{n+1}$ is as in \eqref{usolnlim} \emph{(}i.e. the $\lambda\uparrow 1$ limit of the semi-discrete updates $u^\lambda_{n+1}$%of $u_n$
\emph{)},
\end{enumerate}
then there exists $u\in X$ such that for eventually all $n$, $u_n=u$.
\end{cor}
\begin{proof} For (i), recall that $\Ext S_{\tau,u_n} = S_{\tau,u_n}\cap \Ext X \subseteq \Ext X$ and that $\Ext X$ is a finite set. Hence $\{ u_n\mid n\in\mathbb{N}\}$ is a finite set, so if the $u_n$ are not eventually a single $u$ then we must have some $u, v\in X$ such that $u\neq v$, $u_n = u$  infinitely often, and $u_n = v$ infinitely often. Therefore we must have $n<m<k$ such that $u_n = u_k = u$ and $u_m = v$, and hence 
\[H(u)\geq H(u_{n+1}) \geq ... \geq H(u_{m-1}) \geq H(v) \geq H(u_{m+1}) \geq ... \geq H(u_{k-1}) \geq H(u).\]
All the inequalities are equalities, and therefore by the equality condition on $H$ from Theorem~\ref{Lyapthm} we have $u=v$, a contradiction. Thus the $u_n$ are eventually constant.

For (ii), we show that there are finitely many possible $u\in X$ of the form \eqref{usolnlim}. Each such $u$ has the form
\[ u = \frac{M-\mathcal{M}(\chi_{V_3})}{\mathcal{M}(\chi_{V_2})} \chi_{V_2}+\chi_{V_3} \]
for a partition $V = V_1 \cup V_2 \cup V_3$ with $0\leq M - \mathcal{M}(\chi_{V_3}) \leq \mathcal{M}(\chi_{V_2})$.
To see this, note that $u$ as in \eqref{usolnlim} has $V_1 = \{ i \mid (e^{-\tau\Delta}u_n)_i <\alpha_k \}$, $V_2 = \{ i \mid (e^{-\tau\Delta}u_n)_i =\alpha_k \}$ and $V_3 = \{ i \mid (e^{-\tau\Delta}u_n)_i >\alpha_k \}$.
But since $V$ is finite, there are only finitely many tripartitions of $V$. Hence $\{ u_n\mid n\in\mathbb{N}\}$ is a finite set, and the proof runs as above.
\end{proof}
\begin{cor}[Cf. \text{\cite[Corollary 15]{Budd}}]\label{cor1}
	If $\lambda \in(0,1)$ and the sequence $u_n$ obeys \eqref{mSDobs}, then \[\sum_{n=0}^\infty \left|\left|u_{n+1}-u_n\right|\right|^2_\V < \infty \] and therefore in particular \[\lim_{n\rightarrow\infty} \left|\left|u_{n+1}-u_n\right|\right|_\V = 0. \]
\end{cor}
\begin{proof}
	%If $\lambda=1$ the result follows directly from Corollary \ref{cor0}.%Theorem \ref{obsMMprop} and the fact that MBO trajectories are eventually constant \cite[Thm.]{OKMBO}.
%If $\lambda<1$ then 
By the non-negativity of $H$ and \eqref{Hstep} we have 
	\[(1-\lambda)\sum_{n=0}^N \left|\left|u_{n+1}-u_n\right|\right|^2_\V \leq H(u_0) -H(u_{N+1}) \leq H(u_0)\] so the result follows by taking $N\rightarrow\infty$.
\end{proof}
We wish to use the gradient of $H$ to investigate critical points of the flow. However as we restrict the flow to lie in $S_M$, a non-Hilbert space, we make the following definition. 
\begin{mydef}\label{restrictgrad}
Let $H_0$ be a Hilbert space, and $H_1\subseteq H_0$ be a closed subspace. Let $\tilde H := x + H_1$ for some $x\in H_0$. Then for any  Fr\'{e}chet differentiable map $f:H_0\rightarrow\mathbb{R}$ with Fr\'{e}chet derivative $Df$, we define the Fr\'{e}chet derivative of $f|_{\tilde H}$ at $u\in \tilde H$ by \[ Df|_{\tilde H}(u) := Df(u)|_{H_1} \] where the restriction of the argument to $H_1$ ensures that the $u+h$ terms that appear in the definition of the Fr\'{e}chet derivative satisfy $u+h\in \tilde H$. Then we define the gradient
\[\nabla_{\tilde H} f|_{\tilde H}(u) \in H_1 \] 
to be the Riesz representative of $Df|_{\tilde H}(u)$, i.e. the unique element of $H_1$ such that 
\[\forall v\in H_1 \:\: \langle \nabla_{\tilde H} f|_{\tilde H}(u),v\rangle = Df|_{\tilde H}(u)(v) =Df(u)(v).\] 
Note therefore that for $u\in \tilde H$, since $\nabla_{H_0} f(u)$ is the Riesz representative of $Df(u)$, \[\forall v\in H_1 \:\: \left\langle \nabla_{H_0} f(u),v\right\rangle=\left\langle \nabla_{\tilde H} f|_{\tilde H}(u),v\right\rangle\] and so $\nabla_{\tilde H} f|_{\tilde H}(u) - \nabla_{H_0} f(u) \bot H_1$. That is, for $u\in \tilde H$, $\nabla_{\tilde H} f|_{\tilde H}(u)$ is the orthogonal projection of $\nabla_{H_0} f(u)$ onto $H_1$.
\end{mydef}
\begin{prop}[Cf. \text{\cite[Proposition 16]{Budd}}]
Suppose $M\in (0,\mathcal{M}(\mathbf{1}))$. The Lyapunov functional has gradient \emph{(}for $u\in\V_{(0,1)}\cap X$\emph{)}
\be\label{Hgrad}
\nabla_{S_M} H|_{S_M}(u) = 2(u-e^{-\tau\Delta}u) -2\lambda u + 2\lambda\bar u\mathbf{1} 
\ee
and therefore\emph{:} \begin{enumerate}[\em i.]\item For $u_{n+1}\in\V_{(0,1)}\cap X$ obeying  \eqref{mSDobs}
%\be\begin{split}\label{Hgrad2}
%\nabla_{\V\cap X} H(u_{n+1}) &= 2(u_{n+1}-e^{-\tau\Delta}u_n) -2e^{-\tau\Delta}(u_{n+1}-u_n) -2\lambda u_{n+1} + 2\lambda\frac{\ip{u_{n+1},\mathbf{1} }}{\ip{\mathbf{1} ,\mathbf{1} }}\mathbf{1}\\&=2\lambda\beta_{n+1} -2\lambda\frac{\ip{\beta_{n+1},\mathbf{1} }}{\ip{\mathbf{1} ,\mathbf{1} }}  \mathbf{1} - 2e^{-\tau\Delta}(u_{n+1}-u_n)\\&=- 2e^{-\tau\Delta}(u_{n+1}-u_n).\end{split}
%\ee
\begin{equation}\label{Hgrad2}
	\nabla_{S_M} H|_{S_M}(u_{n+1}) = - 2e^{-\tau\Delta}(u_{n+1}-u_n).
\end{equation}
\item Define $\mathscr{E}$ to be the eigenspace of $\Delta$ with eigenvalue $-\tau^{-1}\log(1-\lambda)$, or $\{\mathbf{0}\}$ if there is no such eigenvalue. If $u\in\V_{(0,1)}\cap X$ then $\nabla_{S_M} H|_{S_M}(u) =\mathbf{0}$ \emph{(}i.e. $u$ is a critical point of $H$\emph{)} if and only if $u \in\left( \frac{M}{\ip{\mathbf{1},\mathbf{1}}}\mathbf{1}+ \mathscr{E}\right)\cap\V_{(0,1)}$.
\end{enumerate}
\end{prop}
\begin{proof}
It is straightforward to check that \[ \ip{\nabla_\V H(u),v}:= \lim_{t\rightarrow0}\frac{ H(u+tv) - H(u)}{t}   = \ip{1-2e^{-\tau\Delta}u,v}+(\lambda-1)\ip{1-2u,v} \] and therefore \[\nabla_\V H(u) =1-2e^{-\tau\Delta}u + (\lambda-1)(1-2u)= \lambda - 2e^{-\tau\Delta}u +2(1-\lambda)u.\] Restricting to $S_M = u + \{\mathbf{1}\}^\bot$, by definition $\nabla_{S_M} H|_{S_M}(u)\in \{\mathbf{1}\}^\bot$ and $\nabla_{S_M} H|_{S_M}(u)-\nabla_{\V} H(u)\in \operatorname{span}\{\mathbf{1}\}$. Thus $\nabla_{S_M} H|_{S_M}(u) = \nabla_{\V} H(u) - \overline{\nabla_{\V} H(u)}\mathbf{1} $, yielding \eqref{Hgrad}.
\begin{enumerate}[i. ]
\item Since $u_{n+1}\in \V_{(0,1)}$ we have  $\beta_{n+1}=\mathbf{0}$, so from \eqref{mSDobs} we have \[u_{n+1} -e^{-\tau\Delta}u_n-\lambda u_{n+1}+\lambda\overline{u_{n+1}}\mathbf{1} =\lambda\beta_{n+1} -\lambda\overline{\beta_{n+1}}  \mathbf{1} = \mathbf{0}\]
 and \eqref{Hgrad2} follows by substituting $u_{n+1}$ into \eqref{Hgrad} and subtracting twice the above expression.
\item Let $A:v\mapsto \bar v\mathbf{1}$ and define $B := 2e^{-\tau\Delta} +2(\lambda-1)I-2\lambda A$. Then $\nabla_{S_M} H|_{S_M}(u) =\mathbf{0}$ if and only if $Bu=\mathbf{0}$. Note that $B\mathbf{1} =2\mathbf{1} +2\lambda\mathbf{1}-2\mathbf{1}-2\lambda\mathbf{1}=\mathbf{0}$ so $u = \frac{M}{\ip{\mathbf{1},\mathbf{1}}}\mathbf{1}\in X$ is a solution. Taking $(\xi_k)_{k>0}$ the eigenvectors for $\Delta$ (with eigenvalues $\lambda_k$) as a basis for $\{\mathbf{1}\}^\bot$  we get that (recalling that $\xi_k \bot  \mathbf{1}$ for $k>0$ and $\lambda_k>0$ for $k>0$) \[B\xi_k = 2(e^{-\tau\lambda_k}+\lambda-1)\xi_k = \mathbf{0}\text{ if and only if }\xi_k\in\mathscr{E}.\] Thus $Bu = \mathbf{0}$ if and only if
\[B\left(u-\frac{M}{\ip{\mathbf{1} ,\mathbf{1} }}\mathbf{1} \right) = \mathbf{0}\] if and only if \[ u-\frac{M}{\ip{\mathbf{1} ,\mathbf{1} }}\mathbf{1} \in\mathscr{E}\] as desired.
\end{enumerate}
\vspace*{-2em}\end{proof}
%[[BEGIN REWRITE]]
\begin{nb}  
We identify when the above identified critical points are all global maximisers. Considering the quadratic terms we observe that for $\eta \bot \mathbf{1}$ \emph{(}recalling $\bar u=M/\ip{\mathbf{1},\mathbf{1}}$\emph{)} \[ \begin{split} H\left(\bar u\mathbf{1}+\eta\right) &=\lambda\left\langle\bar u\mathbf{1}+\eta,\mathbf{1}-\bar u\mathbf{1}-\eta\right\rangle_\V +\left\langle \bar u\mathbf{1}+\eta,\left(I-e^{-\tau\Delta}\right)\left(\bar u\mathbf{1}+\eta\right)\right\rangle_\V\\
	&= H\left(\bar u\mathbf{1}\right) - \left(\ip{\eta,e^{-\tau\Delta}\eta}-(1-\lambda)\ip{\eta,\eta}\right)
 \end{split}\] so, for $\lambda\leq 1$, $u= \frac{M}{\ip{\mathbf{1},\mathbf{1}}}\mathbf{1}$ is a global maximiser of $H$ in $S_M$ if and only if \[P: = e^{-\tau\Delta} -(1-\lambda)I\] is positive semi-definite on $S_M$, i.e. for $(\gamma_k)_{k>0}$ the eigenvalues of $\Delta$, we desire that\emph{:} \[ e^{-\tau\gamma_k} -(1-\lambda)\geq 0, \text{ i.e. }\tau \varepsilon^{-1} \geq 1 - e^{-\tau\gamma_k}  .\] Therefore we have for $\lambda\leq 1$ that  $u= \frac{M}{\ip{\mathbf{1},\mathbf{1}}}\mathbf{1}$ is a global maximiser of $H$ if and only if \[\varepsilon\in \left[\tau,\frac{\tau}{1-e^{-\tau||\Delta||}}\right].\]
Furthermore, note that $\mathscr{E} = \operatorname{ker} P$ so in that case $ \frac{M}{\ip{\mathbf{1},\mathbf{1}}}\mathbf{1} +\mathscr{E}$ are all global maxima since $H(\frac{M}{\ip{\mathbf{1},\mathbf{1}}}\mathbf{1} +\eta)= H(\frac{M}{\ip{\mathbf{1},\mathbf{1}}}\mathbf{1})$ for $\eta \in\operatorname{ker}P$.
\end{nb}
Since $H(u_n)$ is monotonically decreasing and bounded below, it follows that $H(u_n)\downarrow H_\infty$ for some $H_\infty\geq 0$. Furthermore, since the sequence $u_n$ is contained in $X$ and $X$ is compact, there exists a subsequence $u_{n_k}$ that converges to some $u^*\in X$ with $H(u^*)=H_\infty$, since $H$ is continuous. Unfortunately, just like \cite{LB2016} for graph AC flow with the standard quartic potential, or \cite{Budd} for AC flow with the double-obstacle potential, we are unable to infer convergence of the whole sequence from these facts. However by the same argument as in \cite[Lemma 5]{LB2016} if the set of accumulation points of the $u_n$ is finite then there is in fact only one such point and the whole sequence converges. Notably, if $u^*\in\V_{(0,1)}\cap X$ is an accumulation point of the $u_n$ then by Corollary~\ref{cor1} and \eqref{Hgrad2} we have that $\nabla_{S_M} H|_{S_M}(u^*) = \mathbf{0}$. Thus, if $H(u_0) < H(\frac{M}{\ip{\mathbf{1},\mathbf{1}}}\mathbf{1})$ then no accumulation points of the $u_n$ lie in $\V_{(0,1)}\cap X$.

%[[END REWRITE]]
\section{Convergence of the semi-discrete scheme}\label{SDconvsec}

We follow the method of \cite{Budd} to prove convergence of the semi-discrete iterates to the solution of the continuous-time flow \eqref{mACobs} for $T=[0,\infty)$ and $u(0)=u_0\in\V_{[0,1]}$. Note that for $\overline{u_0} \in\{0,1\}$ this result is trivial, since the semi-discrete scheme has $u_n\equiv u_0$ and \eqref{mACobs} has $u(t)\equiv u_0$. Therefore, for the rest of this section we shall assume $\bar u = \overline{u_0} \in (0,1)$.
\subsection{Asymptotics of the $n^\text{th}$ semi-discrete iterate}
We first note two important controls.
\begin{lem}\label{betalem} For $0< \lambda \leq 1$, and $(u_{n+1},\beta_{n+1})$ solving \eqref{mSDobs} for given $u_n\in\V_{[0,1]}$, suppose $\bar u:=\overline{u_n}=\overline{u_{n+1}}\in(0,1)$. Then \be\label{betaavgcontrol}\beta_{n+1}-\overline{\beta_{n+1}}\mathbf{1}\in\V_{[\bar u-1,\bar u]}%\subseteq \V_{[-1,1]}
\ee 
and 
\be\label{betacontrol}\beta_{n+1}\in \V_{[-1,1]}.\ee
\end{lem}
\begin{proof}
 First, suppose $\lambda=1$. Then by \eqref{betasoln1}, \[(\beta_{n+1})_i -\overline{\beta_{n+1}} = \bar u-( e^{-\tau\Delta}u_n)_i \in [\bar u-1,\bar u].\]
Next, suppose $\lambda \in (0,1)$. Then by \eqref{betasoln2}, 
\[\begin{split}(\beta_{n+1})_i -\overline{\beta_{n+1}} &= \bar u +\frac{1}{\lambda}\begin{cases}-(e^{-\tau\Delta}u_n)_i, &\text{if} \:\: \nu\geq (e^{-\tau\Delta}u_n)_i, \\
		- \nu, &\text{if} \:\:(e^{-\tau\Delta}u_n)_i - (1-\lambda)<\nu<(e^{-\tau\Delta}u_n)_i,\\
		1-\lambda-(e^{-\tau\Delta}u_n)_i, &\text{if } \:\: \nu \leq (e^{-\tau\Delta}u_n)_i - (1-\lambda),\end{cases}\\
&=\bar u -1 +\frac{1}{\lambda}\begin{cases}\lambda-(e^{-\tau\Delta}u_n)_i, &\text{if} \:\: \nu\geq (e^{-\tau\Delta}u_n)_i, \\
		\lambda - \nu, &\text{if} \:\:(e^{-\tau\Delta}u_n)_i - (1-\lambda)<\nu<(e^{-\tau\Delta}u_n)_i,\\
		1-(e^{-\tau\Delta}u_n)_i, &\text{if } \:\: \nu \leq (e^{-\tau\Delta}u_n)_i - (1-\lambda),\end{cases}\end{split}\]
where we recall from Proposition \ref{lambdaprop} that $\nu\in(0,\lambda)$. It is therefore easy to check that in the first line the conditional term is non-positive, and in the second line the conditional term is non-negative. Therefore we deduce \eqref{betaavgcontrol}.

Consider the set $B := \{ (\beta_{n+1})_i\mid i\in V\}$. By \eqref{betaavgcontrol}, $B -\overline{\beta_{n+1}}\subseteq[\bar u-1,\bar u]$, so we have that $\operatorname{diam} B \leq 1$. Furthermore, $u_{n+1}\notin\{\mathbf{0,1}\}$, so since $\beta_{n+1}\in\mathcal{B}(u_{n+1})$ we have $x,y\in B$ such that $x\geq 0$ and $y\leq 0$. Therefore $B\subseteq[x-1,x+1]\cap[y-1,y+1]\subseteq[-1,1]$.
\end{proof}

Recall that the semi-discrete scheme is defined by \[(1-\lambda)u_{n+1} =e^{-\tau\Delta}u_n-\lambda \bar u \mathbf{1} +\lambda\beta_{n+1}-\lambda\overline{\beta_{n+1}}\mathbf{1}.
\]
%Approximating $1-\lambda = e^{-\lambda} +\bigO(\lambda^2)$, we get 
%\[u_{n+1} =e^\lambda e^{-\tau\Delta}u_n-\frac{\lambda}{2}e^\lambda\mathbf{1} +\lambda e^\lambda \beta_{n+1} + \bigO(\lambda^2).
%\]
Iterating this formula, we get the following formula for the $n^\text{th}$ term. 
\begin{prop}For $0<\lambda <1$, the semi-discrete solution has $n^\text{th}$ iterate \be u_n =\bar u\mathbf{1}+(1-\lambda)^{-n}e^{-n\tau\Delta}\left(u_0 - \bar u\mathbf{1}\right)+\frac{\lambda}{1-\lambda}\sum_{k=1}^n (1-\lambda)^{-(n-k)}e^{-(n-k)\tau\Delta}\left(\beta_k-\overline{\beta_{k}}\mathbf{1}\right)\ee
and, understanding $\bigO$ to refer to the simultaneous limit of $\tau\downarrow 0$ and $n\rightarrow \infty$ with $n\tau-t\in [0,\tau)$ for some fixed $t\geq 0$ and for fixed $\varepsilon>0$\footnote{Formally, we will say $f(\tau,n)=\bigO(\tau)$ if and only if $\operatorname{limsup} |f(\tau,n)/\tau| <\infty$ as $(\tau,n)\rightarrow(0,\infty)$ in $\{(\rho,m)\mid \rho > 0,\: m\rho-t\in[0,\rho)\}$ with the subspace topology induced by the standard topology on $ (0,\infty)\times\mathbb{N}$.}, we therefore have
\begin{equation}\label{SDiterate}
	u_n =\bar u\mathbf{1}+e^{n\lambda}e^{-n\tau\Delta}\left(u_0 - \bar u\mathbf{1}\right)+\lambda\sum_{k=1}^n e^{(n-k)\lambda}e^{-(n-k)\tau\Delta}\left(\beta_k-\overline{\beta_{k}}\mathbf{1}\right)+\bigO(\tau).
\end{equation}
\end{prop}
\begin{proof}
	We prove by induction as in \cite{Budd}. The $n = 0$ base case is trivial. Then writing $\theta_n:=\beta_n -\overline{\beta_{n}}\mathbf{1}\in\V_{[\bar u-1,\bar u]}$ (by Lemma \ref{betalem}) and inducting we have: 
\vspace{-0.1em}\[\begin{split}u_{n+1} &= (1-\lambda)^{-1}e^{-\tau\Delta}u_n -\frac{\lambda}{1-\lambda}\bar u\mathbf{1} +\frac{\lambda}{1-\lambda}\theta_{n+1}\\
								   &=(1-\lambda)^{-1}e^{-\tau\Delta}\Big[\bar u\mathbf{1}+(1-\lambda)^{-n}e^{-n\tau\Delta}\left(u_0-\bar u\mathbf{1}\right) \\
								   &\hspace{1em } +\frac{\lambda}{1-\lambda}\sum_{k=1}^n (1-\lambda)^{-(n-k)}e^{-(n-k)\tau\Delta}\theta_k\Big]-\frac{\lambda}{1-\lambda}\bar u\mathbf{1} +\frac{\lambda}{1-\lambda}\theta_{n+1}\\
								   &=\left(\frac{1}{1-\lambda}-\frac{\lambda}{1-\lambda}\right)\bar u \mathbf{1}+(1-\lambda)^{-(n+1)}e^{-(n+1)\tau\Delta}\left(u_0-\bar u\mathbf{1}\right)\\
								   &\hspace{1em } + \frac{\lambda}{1-\lambda}\sum_{k=1}^n (1-\lambda)^{-(n-k+1)}e^{-(n-k+1)\tau\Delta}\theta_k +\frac{\lambda}{1-\lambda}\theta_{n+1}\\
								   &=\bar u\mathbf{1}+(1-\lambda)^{-(n+1)}e^{-(n+1)\tau\Delta}\left(u_0-\bar u\mathbf{1}\right)+\frac{\lambda}{1-\lambda}\sum_{k=1}^{n+1} (1-\lambda)^{-(n-k+1)}e^{-(n-k+1)\tau\Delta}\theta_k 
									 \end{split}\] 
	completing the induction. 

Then as in \cite{Budd}, we use $n\tau=t+\bigO(\tau)$ and $n\lambda=t/\varepsilon+\bigO(\tau)$ for fixed $t$ to control:
%\vspace{-0em}
\[\begin{split}&Q:=\left|\left|u_n - \bar u\mathbf{1} - e^{n\lambda} e^{-n\tau\Delta}\left(u_0-\bar u\mathbf{1}\right)-\lambda\sum_{k=1}^n e^{(n-k)\lambda}e^{-(n-k)\tau\Delta}\theta_k\right|\right|_\V \\
	&= \left|\left|\left((1-\lambda)^{-n}-e^{n\lambda}\right)\left(e^{-n\tau\Delta}u_0-\bar u\mathbf{1}\right)+ \lambda\sum_{k=1}^n \left((1-\lambda)^{-(n-k+1)}- e^{(n-k)\lambda}\right)e^{-(n-k)\tau\Delta}\theta_k\right|\right|_\V\\
&\leq \left((1-\lambda)^{-n}-e^{n\lambda}\right)\left|\left|e^{-n\tau\Delta}u_0-\bar u\mathbf{1} \right|\right|_\V + \lambda\sum_{k=1}^n \left((1-\lambda)^{-(n-k+1)}- e^{(n-k)\lambda}\right)\left|\left|e^{-(n-k)\tau\Delta}\theta_k\right|\right|_\V\\ 
	&\leq \left((1-\lambda)^{-n}-e^{n\lambda}\right)\left(\left|\left|e^{-t\Delta}u_0-\bar u\mathbf{1} \right|\right|_\V +\bigO(\tau)\right)+   \lambda C\sum_{k=1}^n \left((1-\lambda)^{-(n-k+1)}- e^{(n-k)\lambda}\right), 
\end{split}\] 
with the first equality by the triangle inequality since $(1-\lambda)^{-(r+1)} - e^{r\lambda}\geq 0$ as $ e^{-\lambda r/(r+1)}\geq 1-\lambda r/(r+1) \geq 1-\lambda $, and for the second inequality we can take $C:=\sup_{s\in[0,t+\tau]}\sup_{\theta\in\V_{[\bar u-1,\bar u]}}\left|\left|e^{-s\Delta}\theta \right|\right|_\V$. Note that we can bound $C \leq \sup_{s\in[0,t+\tau]}\left|\left|e^{-s\Delta} \right|\right|\cdot \max\{\bar u,1-\bar u\}\left|\left|\mathbf{1} \right|\right|_\V=\max\{\bar u,1-\bar u\}\left|\left|\mathbf{1} \right|\right|_\V$. Then 
\[\begin{split}
	RHS&= \left(D+\bigO(\tau)\right)\left(\left(1-\lambda\right)^{-n}-e^{t/\varepsilon} \right)+ C\frac{\lambda}{1-\lambda} \frac{(1-\lambda)^{-n}-1}{(1-\lambda)^{-1}-1}- \lambda C \frac{e^{n\lambda}-1}{e^{\lambda}-1}\\
	\intertext{where $D:=\left|\left|e^{-t\Delta}u_0-\bar u\mathbf{1} \right|\right|_\V$, so noting that $(1-\lambda)^{-n} = (1-(t/\varepsilon+\bigO(\tau))/n)^{-n}=e^{t/\varepsilon+\bigO(\tau)} +\bigO(1/n) = e^{t/\varepsilon} +\bigO(\tau) $,}
	RHS&=\bigO(\tau) + C \left(e^{t/\varepsilon}  - 1 +\bigO(\tau)-\lambda \frac{e^{t/\varepsilon} -1}{e^{\lambda}-1}\right)\\
	& =\bigO(\tau)  +C\left(e^{t/\varepsilon}  - 1\right)\frac{e^\lambda-1-\lambda}{e^\lambda-1} \\
	&=\bigO(\tau) + C\left(e^{t/\varepsilon}  - 1\right)\bigO(\lambda) = \bigO(\tau) \end{split}\] as desired. We have used here that $\frac{e^\lambda-1-\lambda}{e^\lambda-1}$ has Taylor series $\frac{1}{2}\lambda -\frac{1}{12}\lambda^2 + \bigO(\lambda^4)$, as can be checked by direct calculation. \end{proof}
\subsection{Proof of convergence}
We consider the limit of \eqref{SDiterate} as $\tau\downarrow 0$, $n\rightarrow \infty$ with $n\tau\rightarrow t$ for some fixed $t$ and $\tau\in(0,\varepsilon)$. The key insight, as in \cite{Budd}, is noticing that \eqref{SDiterate} strongly resembles a Riemann sum for the integral form for the mass-conserving AC flow from Theorem~\ref{explicitthm}. To exploit this, we define the piecewise constant function $z_\tau:[0,\infty)\rightarrow\V$,
\[z_\tau(s): =\begin{cases} e^{-\tau/\varepsilon} e^{\tau\Delta}\beta^{[\tau]}_1, & 0\leq s \leq \tau,\\
									 e^{-k\tau/\varepsilon} e^{k\tau\Delta} \beta^{[\tau]}_k,	& (k-1)\tau<s\leq k\tau \text{ for } k\in\mathbb{N}	,					
\end{cases}\]
and the function \[\gamma_\tau(s):= e^{s/\varepsilon}e^{-s\Delta} z_\tau(s)=\begin{cases} e^{-(\tau-s)/\varepsilon} e^{(\tau-s)\Delta}\beta^{[\tau]}_1, & 0\leq s \leq \tau,\\
									 e^{-(k\tau-s)/\varepsilon} e^{(k\tau-s)\Delta} \beta^{[\tau]}_k,	& (k-1)\tau<s\leq k\tau			\text{ for } k\in\mathbb{N},				
\end{cases}\]
(note that for bookkeeping we introduce the superscript $[\tau]$ to keep track of the time step governing a particular sequence of $u_n$ and $\beta_n$). We note an important convergence result.
\begin{prop}\label{BAprop}
For any sequence $\tau'_n\rightarrow 0$ with $\tau'_n<\varepsilon$ for all $n$,  there exists a function $z:[0,\infty) \rightarrow \V$ and a subsequence $\tau_n$ of $\tau'_n$ such that $z_{\tau_n}$ converges weakly to $z$ in $L^2_{loc}([0,\infty);\V)$ and $z_{\tau_n}$ weak*-converges to $z$ in $L^\infty_{loc}([0,\infty);\V)$.
\end{prop}
\begin{proof} For $N\in \mathbb{N}$, consider  $z_\tau|_{[0,N]}$. As the $\beta^{[\tau]}_k\in\V_{[-1,1]}$ for all $k$ and $\tau$ by Lemma \ref{betalem}, we have for all $s\in[0,N]$ and $\tau<\varepsilon$
\[\left|\left|z_\tau(s)\right|\right|_\V \leq \sup_{s'\in[0,N+\varepsilon]}\left|\left|e^{-s'(\frac{1}{\varepsilon}I-\Delta)}\right|\right| \cdot \left|\left|\mathbf{1}\right|\right|_\V \leq \max\left\{1, e^{(N+\varepsilon)(||\Delta||-\varepsilon^{-1})}\right\} \cdot \left|\left|\mathbf{1}\right|\right|_\V\] 
where we have used that for $s\leq N$ the corresponding $k\tau$ in the exponent of $z_\tau(s)$ is less than $N+\tau$, and that $||e^{-s'(\frac{1}{\varepsilon}I-\Delta)}||=e^{s'(||\Delta||-\varepsilon^{-1})}$ is maximised at the endpoints of $[0,N+\varepsilon]$. Therefore the $z_\tau|_{[0,N]}$ are uniformly bounded in $||\cdot||_\V$ (and therefore in $||\cdot||_\infty$ since all norms on $\V$ are equivalent) for $\tau<\varepsilon$, and hence they lie in a closed ball in $L^2([0,N];\V)$ and in $L^\infty([0,N];\V)$. By the Banach--Alaoglu theorem the former  ball is weak-compact and the latter ball is weak*-compact. Hence for any $\tau'_n\downarrow 0$ there exists $\tau''_n$ a subsequence of $\tau'_n$ and  $ z\in L^2([0,N];\V)$ and $w \in L^\infty([0,N];\V)$ such that 
\begin{align*}
& z_{\tau''_n}|_{[0,N]}\rightharpoonup z\text{ in } L^2([0,N];\V), 
& z_{\tau''_n}|_{[0,N]}\rightharpoonup^* w\text{ in } L^\infty([0,N];\V) .
\end{align*}
We claim that $z = w$ a.e. on $[0,N]$. By the definitions of the weak and weak* topologies we have that for all $f \in L^2([0,N];\V)$ and $g \in L^1([0,N];\V)$
\begin{align*}
& \int_0^N \ip{z_{\tau''_n}(t),f(t)}\; dt \rightarrow \int_0^N \ip{z(t),f(t)}\; dt, 
& \hspace{-0.8em}\int_0^N \ip{z_{\tau''_n}(t),g(t)}\; dt \rightarrow \int_0^N \ip{w(t),g(t)}\; dt.
\end{align*}
Hence for any $A\subseteq [0,N]$ (measurable) and $i\in V$ consider $f(t):=\chi_A(t) \chi_i$. Then $f\in L^2([0,N];\V)\cap L^1([0,N];\V)$ and so for all measurable $A\subseteq[0,N]$, 
\[
\int_A z_i(t) -w_i(t) \; dt = 0 .
\]
Hence $z_i = w_i$ a.e. for each $i\in V$, so $ z = w$ a.e. on $[0,N]$. 

Finally, we extend to $[0,\infty)$ by a ``local-to-global" diagonal argument. First, we take $N = 1$: by above we can choose a subsequence $\tau^{(1)}$ of $\tau'$ such that $z_{\tau^{(1)}_n}$ converges in both the weak topology on $L^2$ and the weak* topology on $L^\infty$ to some $z$ on $[0,1]$ . Then to move from $N$ to $N+1$ we likewise choose a subsequence $\tau^{(N+1)}$ of $\tau^{(N)}$ such that $z_{\tau^{(N+1)}_n}$ converges in both senses to $z$ on $[0,N+1]$. Finally, define $\tau_n : = \tau_n^{(n)}$. Then for all bounded $T\subseteq [0,\infty)$, we have $T\subseteq [0,M]$ for some $M\in \mathbb{N}$ and hence $z_{\tau_n}|_T$ is eventually a subsequence of $z_{\tau^{(M)}_n}|_T$ and so converges in both senses to $z|_T$.
\end{proof}
\begin{cor} \label{zcor} From $z_{\tau_n} \rightharpoonup z$ in $L^2_{loc}([0,\infty);\V)$ we infer\emph{:}
\begin{enumerate}[ A.]
\item $\gamma_{\tau_n}\rightharpoonup \gamma $, where $\gamma(s) := e^{s/\varepsilon}e^{-s\Delta} z$, and $z_{\tau_n} -\bar z_{\tau_n}\mathbf{1}\rightharpoonup z-\bar z\mathbf{1}$ \emph{(}both in $L^2_{loc}([0,\infty);\V)$\emph{)}.
\item For all $t\geq 0$, \[\int_0^t z_{\tau_n}(s)-\bar z_{\tau_n}(s)\mathbf{1}\; ds \rightarrow \int_0^t  z(s)-\bar z(s)\mathbf{1}\; ds.\]
\item Replacing $\tau_n$ by an appropriate subsequence, we have strong convergence of the Cesàro sums, i.e. for all bounded $T\subseteq[0,\infty)$ \begin{align*}&\frac{1}{N}\sum_{n=1}^N z_{\tau_n} \rightarrow z &\text{ and }&&\frac{1}{N}\sum_{n=1}^N \gamma_{\tau_n} \rightarrow \gamma && \text { in } L^2(T;\V)\end{align*} as $N\rightarrow \infty$. 
\end{enumerate}
And from $z_{\tau_n} \rightharpoonup^* z$ in $L^\infty_{loc}([0,\infty);\V)$ we infer\emph{:}
\begin{enumerate}[A.]\setcounter{enumi}{3}
\item $\gamma_{\tau_n}\rightharpoonup^* \gamma$ in $L^\infty_{loc}([0,\infty);\V)$.
\end{enumerate}
\end{cor}
\begin{proof}
Claim (A) follows since $f\mapsto e^{s/\varepsilon}e^{-s\Delta} f$ (where $s$ is the argument of $f$) and $f\mapsto f-\bar f\mathbf{1}$ are continuous self-adjoint maps on $L^2(T;\V)$ for $T$ bounded. Hence for all $f\in L^2(T;\V)$, 
\[(\gamma_{\tau_n},f)_{s\in T} = (z_{\tau_n},e^{s/\varepsilon}e^{-s\Delta}f)_{s\in T}\rightarrow (z,e^{s/\varepsilon}e^{-s\Delta}f)_{s\in T} = (\gamma,f)_{s\in T}\]
and 
\[(z_{\tau_n}-\bar z_{\tau_n}\mathbf{1}, f)_{t\in T}=(z_{\tau_n}, f-\bar f\mathbf{1})_{t\in T}\rightarrow (z,f-\bar f\mathbf{1})_{t\in T} = (z-\bar z\mathbf{1}, f)_{t\in T}.\]
Claim (B) is a direct consequence of weak convergence. Claim (C) follows by the Banach--Saks theorem \cite{BS}, which states that weak $L^p$ convergence on a bounded interval entails strong convergence of Cesàro sums on that interval along an appropriate subsequence, and a ``local-to-global" diagonal argument as in the above proof to extract a subsequence that works on all of $[0,\infty)$. Claim (D) follows since $f\mapsto e^{s/\varepsilon}e^{-s\Delta} f$ is continuous on $L^\infty(T;\V)$ and on $L^1(T;\V),$ for $T$ bounded, and with respect to  the pairing of $L^\infty$ with $L^1$, for all $f\in L^\infty(T;\V)$ and $g\in L^1(T;\V)$
\[
\int_T \ip{e^{s/\varepsilon}e^{-s\Delta} f(s),g(s)} \; ds = 
\int_T \ip{f(s),e^{s/\varepsilon}e^{-s\Delta} g(s)} \; ds
\]
so the map is ``self-adjoint" and so (D) follows by the same argument as (A).
\end{proof}
%\begin{nb} Since $L^2(T;\V)$ is separable, the Banach--Alaoglu and Banach--Saks theorems have constructive proofs. Hence, we can in principle extract an explicit $\tau_n$ from the above discussion. We omit the details of this construction. 
%\end{nb}
We now return to the question of convergence of the semi-discrete iterates. Taking $\tau$ to zero along the sequence $\tau_n$, we define for all $t\geq 0$ the continuum-time function 
\be \label{uhat}\hat u(t) := \lim_{n\rightarrow\infty, m =\ceil{t/\tau_n}} u^{[\tau_n]}_m. \ee
Therefore by \eqref{SDiterate} (note that $m$ depends on both $t$ and $n$, but for the sake of readability we will write $m$ rather than $m_n(t)$)
\[\begin{split}
%\hspace{-3em}
\hat u(t) = \bar u\mathbf{1} +  \lim_{n\rightarrow \infty} \bigg( & e^{m\tau_n/\varepsilon} e^{-m\tau_n\Delta}\left(u_0 -\bar u\mathbf{1}\right)\\ &+\frac{1}{\varepsilon}e^{m\tau_n/\varepsilon} e^{-m\tau_n\Delta} \tau_n\sum_{k=1}^m e^{-k\tau_n/\varepsilon} e^{k\tau_n\Delta}\left(\beta^{[\tau_n]}_k-\bar\beta^{[\tau_n]}_k\mathbf{1}\right)\bigg)\end{split}\]
and by rewriting the sum term via the definition of $z_{\tau_n}$:
\[\hat u(t)= \bar u\mathbf{1} + \lim_{n\rightarrow \infty} e^{m\tau_n/\varepsilon} e^{-m\tau_n\Delta}\left(u_0 -\bar u\mathbf{1}\right) +\frac{1}{\varepsilon}e^{m\tau_n/\varepsilon} e^{-m\tau_n\Delta}\int_0^{m\tau_n} z_{\tau_n}(s)-\bar z_{\tau_n}(s)\mathbf{1} \; ds.
\]
Then finally to prove global convergence we must show the desiderata:
\begin{enumerate}[ (i) ]
\item $\hat u(t)$ exists for all $t\geq 0$,
\item $\hat u(t)\in\V_{[0,1]}$ for all $t\geq 0$,
\item $\hat u$ is continuous and $H^1_{loc}([0,\infty);\V)$, and $\hat u(t)$ is a solution to the AC flow.
\end{enumerate}

Towards (i), let $A:= \varepsilon^{-1}I-\Delta$ and $e_n :=  m\tau_n - t \in [0,\tau_n)$. Then 
\[e^{m\tau_n/\varepsilon} e^{-m\tau_n\Delta} = e^{(t+e_n)A} = e^{tA}(I + \bigO(e_n)) = e^{tA} + \bigO(\tau_n)\]
and so 
\[\begin{split}\hat u(t) =  \bar u\mathbf{1} +& \lim_{n\rightarrow\infty} e^{tA}\left(u_0 -\bar u \mathbf{1}\right)\\&+\frac{1}{\varepsilon}\left(e^{tA} + \bigO(\tau_n)\right)\left(\int_0^{t}z_{\tau_n}(s)-\bar z_{\tau_n}(s)\mathbf{1}\; ds+\int_t^{t+e_n} z_{\tau_n}(s)-\bar z_{\tau_n}(s)\mathbf{1} \; ds\right).  \end{split}\]
Hence, as  by Proposition \ref{BAprop} the $z_{\tau_n}-\bar z_{\tau_n}\mathbf{1}$ are uniformly bounded on any compact interval (and so on $[0,t+\max_n e_n]$), it follows by Corollary \ref{zcor}(B)
\be \label{uhatsoln} \hat u(t) =  \bar u\mathbf{1} + e^{t/\varepsilon} e^{-t\Delta}\left(u_0 -\bar u\mathbf{1}\right) + \frac{1}{\varepsilon}e^{t/\varepsilon} e^{-t\Delta}\int_0^{t} z(s)-\bar z(s)\mathbf{1}\; ds. \ee

To show (ii): $\hat u(t)$ is a limit of semi-discrete iterates, each of which lies in $\V_{[0,1]}$. 

To show (iii): We verify the sufficient conditions in Theorem~\ref{explicitthm}.  By \eqref{uhatsoln} we have that $\hat u$ has the desired integral form, and since $\gamma-\bar\gamma \mathbf{1}$ is a weak limit of locally bounded and locally integrable functions we have that $\gamma-\bar\gamma \mathbf{1}$ is locally bounded a.e. and is locally integrable. 
%$z-\bar z(s)\mathbf{1}$ is locally bounded (i.e. bounded on compact subsets of $[0,\infty)$) as it is a weak limit of locally bounded functions. Thus $\int_0^t z(s)-\bar z(s)\mathbf{1}\; ds$, and so by \eqref{uhatsoln} $\hat u$, are continuous. Next by (ii), $\hat u$ is bounded so is $L^2_{loc}$. Lastly $\hat u$ has weak derivative \[\frac{d\hat u}{dt} = \left(\frac{1}{\varepsilon}I -\Delta\right)\left(u_0 -\bar u\mathbf{1}\right) + \frac{1}{\varepsilon}e^{t/\varepsilon} e^{-t\Delta}\left( z(t)-\bar z(t)\mathbf{1} + \left(\frac{1}{\varepsilon}I -\Delta\right) \int_0^t  z(s)-\bar z(s)\mathbf{1}\; ds\right)\] 
%which is $L^2_{loc}$ since $ z-\bar z\mathbf{1}$ is a weak limit of $L^2_{loc}$ functions (so is $L^2_{loc}$) and $\int_0^t  z(s)-\bar z(s)\mathbf{1}\; ds$ is a pointwise limit of locally bounded functions, so is locally bounded.
%
%Finally, to show (iv), we recall the rearrangement of the AC ODE: \[\frac{d}{dt}\left(e^{-t/\varepsilon}e^{t\Delta}\left(u(t)-\bar u\mathbf{1}\right)\right)=\varepsilon^{-1}e^{-t/\varepsilon}e^{t\Delta}\left(\beta(t)-\bar\beta(t)\mathbf{1}\right).\] 
%Inspection of \eqref{uhatsoln} shows that plugging in $\hat u$ into this ODE gives \[\gamma(t) -\bar\gamma(t)\mathbf{1} = \beta(t)-\bar\beta(t)\mathbf{1}\] so to show that $\hat u$ solves the ODE a.e. 
Thus it suffices to check the subdifferential condition.
\begin{lem} $\gamma(t) \in \mathcal{B}(\hat u(t))$ for a.e. $t\geq 0$.
\end{lem}
We give two proofs of this result. We first recap part of the proof in \cite{Budd} in order to derive the characterisation of $\gamma$ in \eqref{gammacesaro}, and next we give a novel proof of this result.
\begin{proof}[Proof (A), cf. \cite{Budd}.]
 By Corollary \ref{zcor}(C), on each bounded $T\subseteq[0,\infty)$ $\gamma$ is the $L^2(T;\V)$ limit of \[S_N:= \frac{1}{N} \sum_{n=1}^N \gamma_{\tau_n}\] as $N\rightarrow \infty$. As $L^2$ convergence implies a.e. pointwise convergence along a subsequence, by a ``local-to-global"
diagonal argument there exists a sequence $N_k \rightarrow \infty$ such that for a.e. $t\geq 0$ 
\[\gamma(t) = \lim_{k\rightarrow\infty}  \frac{1}{N_k} \sum_{n=1}^{N_k} \gamma_{\tau_n}(t).\]
Recall $A:= \varepsilon^{-1}I-\Delta$,  $m := \ceil{t/\tau_n}$, and $e_n := m\tau_n -t \in [0,\tau_n)$. Then 
by Lemma \ref{betalem}
\[\begin{split}||\gamma_{\tau_n}(t) - \beta^{[\tau_n]}_m||_\V=|| (e^{-e_nA}-I)\beta^{[\tau_n]}_m||_\V& \leq (1-e^{-e_n||A||})||\mathbf{1}||_\V \\& < (1-e^{-\tau_n||A||})||\mathbf{1}||_\V<\tau_n||A||\,||\mathbf{1}||_\V.\end{split}\] Therefore for a.e. $t\geq 0$,%\footnote{ We show that  $\frac{1}{N_k} \sum_{n=1}^{N_k}\bigO(\tau_n)\rightarrow 0$. Let $\varepsilon >0$. Then  $\exists N$ s.t. $\forall n>N$ $\bigO(\tau_n) <\varepsilon/2$. Thus for all $k$ s.t. $N_k >\max\{N, \frac{2}{\varepsilon} \sum_{n=1}^{N}\bigO(\tau_n)\}$,  $\frac{1}{N_k} \sum_{n=1}^{N_k}\bigO(\tau_n)=\frac{1}{N_k} \sum_{n=1}^{N}\bigO(\tau_n)+\frac{1}{N_k} \sum_{n=N+1}^{N_k}\bigO(\tau_n)<\varepsilon$.}
\[
\left|\left |\gamma(t) -\frac{1}{N_k} \sum_{n=1}^{N_k} \beta^{[\tau_n]}_m \right |\right|_\V \leq \left|\left|\gamma(t)- \frac{1}{N_k} \sum_{n=1}^{N_k} \gamma_{\tau_n}(t) \right|\right|_\V + ||A||\, ||\mathbf{1}||_\V \frac{1}{N_k}\sum_{n=1}^{N_k} \tau_n\rightarrow 0
\]
as $k\rightarrow\infty$ (since $\tau_n\rightarrow 0$ and the convergence of a sequence implies the convergence of its Ces\`aro sums to the same limit), so for a.e. $t\geq 0$
\be \label{gammacesaro}
\gamma(t) = \lim_{k\rightarrow\infty} \frac{1}{N_k} \sum_{n=1}^{N_k} \beta_m^{[\tau_n]}.
\ee
The result then follows as in \cite{Budd}; we omit the details as they are identical to those in \cite{Budd}.
%%Recall that $u^{[\tau_n]}_m\rightarrow\hat u(t)$ and $\beta_m^{[\tau_n]}\in \mathcal{B}( u^{[\tau_n]}_m)$. Suppose first that $\hat u_i(t) \in (0,1)$.  Then we have some $M$ such that for all $n>M$, $(u^{[\tau_n]}_m)_i\in (0,1)$ and so $(\beta^{[\tau_n]}_m)_i=0$. Hence \[\gamma_i(t) = \lim_{k\rightarrow\infty}  \frac{1}{N_k}\left( \sum_{n=1}^{M} (\beta_m^{[\tau_n]})_i + \sum_{n=M+1}^{N_k} 0\right) = 0\] as desired. Next suppose $\hat u_i(t) = 0$. Then we have some $M$ such that for all $n>M$, $(u^{[\tau_n]}_n)_i\in [0,1)$ and so $ (\beta^{[\tau_n]}_n)_i \geq 0$. Hence 
%\[\gamma_i(t) \geq \lim_{k\rightarrow\infty}  \frac{1}{N_k}\left( \sum_{n=1}^{M} (\beta_m^{[\tau_n]})_i + \sum_{n=M+1}^{N_k} 0 \right) = 0 \] 
%as desired. Likewise for $\hat u_i(t) = 1$, $\gamma_i(t) \leq 0$. Hence we have $\gamma(t)\in\mathcal{B}(\hat u (t))$.
\end{proof}
\begin{proof}[Proof (B).]
Fix $i\in V$ and bounded $T\subseteq[0,\infty)$. For tidyness of notation, we define $x_n(t):=u^{[\tau_n]}_{\ceil{t/\tau_n}}$ and $x(t) := \hat u_i(t)$, and likewise $\xi_n(t):=  (\beta^{[\tau_n]}_{\ceil{t/\tau_n}})_i$ and $\xi(t):=\gamma_i(t)$. 
Let \begin{align*} &T_1:=\{t\in T \mid  x(t) = 0\}, 
&T_2:=\{t\in T \mid  x(t) \in (0,1)\},  & &T_3:=\{t\in T \mid  x(t) = 1\}.\end{align*} Then it suffices to show that $\xi \geq 0$ a.e. on $T_1$, $\xi = 0$ a.e. on $T_2$, and $\xi \leq 0 $ a.e. on $T_3$. 

By Corollary~\ref{zcor}(D) we have that $(\gamma_{\tau_n})_i\rightharpoonup^*\xi$ in $L^\infty_{loc}([0,\infty);\V)$ and therefore $(\gamma_{\tau_n})_i\rightharpoonup^*\xi$ in $L^\infty(T,\mathbb{R})$, i.e. for all $f\in L^1(T,\mathbb{R})$, as  $n\rightarrow\infty$
\[
\int_T (\gamma_{\tau_n})_i(t)f(t)\; dt \rightarrow \int_T \xi(t)f(t) \; dt.
\] 
Recalling from Proof (A) that $(\gamma_{\tau_n})_i(t) =\xi_n(t)+\bigO(\tau_n)$, we infer that as $n\rightarrow\infty$
\[
\int_T \xi_n(t)f(t)\; dt \rightarrow \int_T \xi(t)f(t) \; dt.
\] 

By \eqref{uhat} we have by definition that for all $t \in T_1$, $x_n(t) \rightarrow 0$. We define the (measurable) sets  $A_N:=\{t\in T_1\mid \forall n\geq N\: x_n(t) < 1/2\}$. Then by the pointwise convergence of the $x_n$,  $T_1 = \bigcup_N A_N$. Suppose for contradiction that for some $X\subseteq T_1$ of positive measure, $\xi < 0$ on $X$. So there exists $\delta >0$ and $Y\subseteq X$ of positive measure such that $\xi \leq -\delta$ on $Y$. As $T_1$ is the union of the $A_N$ there exists $N\in\mathbb{N}$ such that $Y\cap A_N$ is of positive measure. Taking test function $f = \chi_{Y\cap A_N}$ we infer that as $n\rightarrow\infty$ (and $\mu$ the Lebesgue measure)
\[
\int_{Y\cap A_N} \xi_n(t) \; dt \rightarrow \int_{Y\cap A_N} \xi(t)\; dt \leq -\delta\mu(Y\cap A_N) < 0
\]
but since $\beta_{\ceil{t/\tau_n}}^{[\tau_n]}\in \mathcal{B}( u^{[\tau_n]}_{\ceil{t/\tau_n}})$ we have that if $t\in A_N$ then for all $n\geq N$, $\xi_n(t)\geq 0$, so this is a contradiction. Hence $\xi\geq 0$ a.e. on $T_1$. By the same argument, $\xi \leq 0$ a.e. on $T_3$.

Finally, for all $t\in T_2$, since $x_n(t)\rightarrow x(t)$, $x_n(t)$ is eventually in $(0,1)$. Define $B_N:= \{t\in T_2\mid \forall n\geq N\: x_n(t) \in (0,1)\}$, and note that $T_2 = \bigcup_N B_N$ and that for $t\in B_N$ and $n\geq N$, $\xi_n(t)= 0$ since $\beta_{\ceil{t/\tau_n}}^{[\tau_n]}\in \mathcal{B}( u^{[\tau_n]}_{\ceil{t/\tau_n}})$. Suppose for contradiction that for some $X\subseteq T_2$ of positive measure, $\xi \neq 0$ on $X$. Then WLOG  there exists $\delta>0$ and $Y\subseteq X$ of positive measure such that $\xi \geq \delta$ on $Y$. As before there exists $N\in\mathbb{N}$ such that $Y\cap B_N$ is of positive measure. Taking $f = \chi_{Y\cap B_N}$ we infer that as $n\rightarrow\infty$ (for $n\geq N$)
\[
0 =\int_{Y\cap B_N} \xi_n(t) \; dt \rightarrow \int_{Y\cap B_N} \xi(t)\; dt \geq \delta\mu(Y\cap A_N) > 0
\]
a contradiction. Therefore $\xi =0$ a.e. on $T_2$. 
\end{proof}
\begin{nb}
Proof \emph{(}B\emph{)} is also valid in the non-mass-conserving case in \emph{\cite{Budd}}. We thank Dr. Carolin Kreisbeck for her suggestion of using weak* $L^\infty$ convergence which led to the development of this proof.
\end{nb}
In summary, we have the following convergence result.
\begin{thm}[\text{Cf. \cite[Theorem 21]{Budd}}]\label{SDlimit}
For any given $u_0\in\V_{[0,1]}\setminus\{\mathbf{0,1}\}$, $\varepsilon>0$ and $\tau_n\downarrow 0$, there exists a subsequence $\tau'_n$ of $\tau_n$ with $\tau'_n<\varepsilon$ for all $n$, such that along this subsequence the semi-discrete iterates $(u^{[\tau'_n]}_m,\beta^{[\tau'_n]}_m)$ given by \eqref{mSDobs} with initial state $u_0$ converge to an AC solution. That is, there exists $(\hat u,\gamma)$ a solution to \eqref{mACobs} with $\hat u(0) = u_0$, such that\emph{:} \begin{itemize}
\item $\beta^{[\tau'_n]}_{\ceil{\cdot/\tau'_n}}$ converges to $\gamma$ weakly in $L^2_{loc}([0,\infty);\V)$ and weakly* in $L^\infty_{loc}([0,\infty);\V)$, 
\item for each $t\geq 0$ and for $m:=\ceil{t/\tau'_n}$, $u^{[\tau'_n]}_m\rightarrow \hat u(t)$ as $n\rightarrow\infty$, and 
\item there is a sequence $N_k\rightarrow\infty$ such that for almost every $t\geq 0$, $\frac{1}{N_k}\sum_{n=1}^{N_k}\beta^{[\tau'_n]}_m\rightarrow \gamma(t)$ as $k\rightarrow\infty$.
\end{itemize}
\end{thm}
\begin{nb}
This result proves Theorem \ref{existence}, i.e. the existence of AC solutions.
\end{nb}
\begin{nb}
For $u_0=\mathbf{0}$ or $\mathbf{1}$, the $u^{[\tau_n]}_m\equiv u_0$ trivially converge but the $\beta^{[\tau_n]}_m$ need not converge.
\end{nb}
\begin{nb} As in \emph{\cite{Budd}}, we can avoid passing to a subsequence in all but the last of these convergences because of Theorem \ref{uniqueness}. 
Recall the fact noted in \emph{\cite{Budd}}\emph{:} if $(X,\rho)$ is a topological space, $x_n,x\in X$, and every subsequence of $x_n$ has a further subsequence converging to $x$ in $\rho$, then $x_n\rightarrow x$ in $\rho$. 
Let $\tau_n\downarrow 0$, with $\tau_n<\varepsilon$ for all $n$, and $x_n:=t\mapsto u^{[\tau_n]}_{\ceil{t/\tau_n}}\in (\V_{t\in[0,\infty)},\rho)$ for $\rho$ the topology of pointwise convergence. By Theorem \ref{SDlimit} applied to $\tau_{n_k}$, every subsequence $x_{n_k}$ has a subsequence converging to an AC solution with initial condition $u_0$. By uniqueness, these must equal $\hat u$. Therefore $x_n \rightarrow \hat u$ pointwise, without passing to a subsequence.
Likewise, the corresponding $\gamma-\bar\gamma\mathbf{1}$ is unique up to a.e. equivalence, so $z_{\tau_n}-\overline{z_{\tau_n}}\mathbf{1}\rightharpoonup z-\bar z\mathbf{1}$ and $\gamma_{\tau_n}-\overline{\gamma_{\tau_n}}\mathbf{1}\rightharpoonup \gamma-\bar \gamma\mathbf{1}$ in $L^2_{loc}([0,\infty;\V)$ without passing to a subsequence. Finally, when $\gamma$ is unique up to a.e. equivalence \emph{(}i.e. when $\hat u(t)\notin \V_{\{0,1\}}$ for all $t\geq 0$\emph{)} then $\gamma_{\tau_n}\rightharpoonup \gamma$ in $L^2_{loc}([0,\infty;\V)$ and $\gamma_{\tau_n}\rightharpoonup^* \gamma$ in $L^\infty_{loc}([0,\infty;\V)$ without passing to a subsequence. 
\end{nb}
\subsection{Consequences of Theorem \ref{SDlimit}}
Given this representation of the unique solution to \eqref{mACobs} as a limit of semi-discrete approximations, we can deduce a number of properties of this solution. 

First, following \cite{Budd}, we verify that the unique AC solution is a decreasing flow of $\GL$ by considering the Lyapunov functional $H$ for the semi-discrete scheme defined in \eqref{Lyap}, and in doing so obtain a control on the behaviour of $\GL(\hat u(t))$. 

We recall from \cite{Budd}, for $u\in\V_{[0,1]}$, the following scaling of the Lyapunov functional \[ H_\tau(u) :=\frac{1}{2\tau}H(u) = \GL(u)-\frac{1}{2}\tau \ip{u,Q_\tau u}\]
where $\tau^2 Q_\tau := e^{-\tau\Delta}-I + \tau\Delta$. Furthermore we recall the result.
\begin{prop}[$\text{\cite[Proposition 22]{Budd}}$] \label{Htauprop}
Let $u_\tau, u\in \V_{[0,1]}$ satisfy $||u_\tau-u||_\V \rightarrow 0$ as $\tau\rightarrow 0$. Then it follows that $H_\tau(u_\tau)\rightarrow \GL(u)$. 
\end{prop}
%We therefore deduce the desired result exactly as in \cite{Budd}.
\begin{thm}\label{GLthm}
The AC trajectory $\hat u$ defined by \eqref{uhat} has $\GL(\hat u(t))$ monotonically decreasing in $t$. More precisely\emph{:} for all $t > s \geq 0 $, \be \label{GLstep} 
 \GL(\hat u(s)) -  \GL(\hat u(t)) \geq \frac{1}{2(t-s)} \left|\left|\hat u(s) -\hat u(t) \right|\right|_\V^2.
\ee 
\end{thm}
\begin{proof} We reproduce the proof from \cite{Budd}.
Let  $t > s \geq 0 $ and $m:= \ceil{s/\tau_n}$ and $\ell := \ceil{t/\tau_n}$. 
Recall from \cite{Budd}: for all sequences $v_n\in \V$,  \be \label{Plaw} \sum_{n=1}^N  \left|\left|v_n \right|\right|_\V^2  = \frac{1}{N} \left|\left|\sum_{n=1}^N v_n \right|\right|_\V^2 + \frac{1}{N}\sum_{k<n} \left|\left|v_n - v_k \right|\right|_\V^2   \geq \frac{1}{N} \left|\left|\sum_{n=1}^N v_n \right|\right|_\V^2 .%\tag{$*$} 
\ee
Now by \eqref{uhat}, we have $u_m^{[\tau_n]}\rightarrow \hat u(s)$ and $u_\ell^{[\tau_n]}\rightarrow \hat u(t)$. It follows that:
\begin{align*}
 \GL(\hat u(s)) -  \GL(\hat u(t)) &=  \lim_{n\rightarrow\infty} H_{\tau_n}\left(u_m^{[\tau_n]}\right) -  H_{\tau_n}\left(u_\ell^{[\tau_n]}\right)&&\text{ by Proposition \ref{Htauprop}}&\\
&\geq \lim_{n\rightarrow\infty} \frac{1}{2\tau_n}\left(1-\frac{\tau_n}{\varepsilon}\right) \sum_{k=m}^{\ell-1} \left|\left|u_{k+1}^{[\tau_n]}-u_k^{[\tau_n]}\right|\right|_\V^2 &&\text{ by \eqref{Hstep}}&\\
&\geq \lim_{n\rightarrow\infty}  \frac{1}{2\tau_n}\left(1-\frac{\tau_n}{\varepsilon}\right) \frac{1}{\ell-m}\left|\left|u_{\ell}^{[\tau_n]}-u_m^{[\tau_n]}\right|\right|_\V^2&&\text{ by \eqref{Plaw}}& \\
&=\frac{1}{2(t-s)} \left|\left|\hat u(s) -\hat u(t) \right|\right|_\V^2 \geq 0 &&&
\end{align*}
as desired, since $\tau_n(\ell-m)\rightarrow t - s$.\end{proof}
\begin{nb}
Since $\GL(\hat u(s))-\GL(\hat u(t))\leq \GL(\hat u(s))\leq\GL(\hat u(0))$ it follows by \eqref{GLstep} that \[ \left|\left|\hat u(s) -\hat u(t) \right|\right|_\V\leq \sqrt{|t-s|}\sqrt{2\GL(\hat u(0))}\] which as in \emph{\cite{Budd}} gives an explicit $C^{0,1/2}$ condition for $\hat u$. 
\end{nb}
Next, we derive some controls on $\gamma$ and thereby infer a Lipschitz condition on $\hat u$. 
\begin{lem}\label{gammalem}
For $\gamma(t)$ given at a.e. $t\in T$ by \eqref{gammacesaro}, at each such $t$ \[\gamma(t) -\bar\gamma(t)\mathbf{1} \in \V_{[\bar u-1,\bar u]} %\subseteq \V_{[-1,1]}
\text{ and } \gamma(t) \in \V_{[-1,1]}.\]
\end{lem}
\begin{proof}
Follows immediately from \eqref{gammacesaro} and the controls in Lemma \ref{betalem}.
\end{proof}
%Following \cite{Budd}, we therefore infer a Lipschitz condition on $\hat u$.
\begin{thm}\label{mACLips}
The AC trajectory $\hat u$ defined by \eqref{uhat} has $ \hat u \in C^{0,1}([0,\infty);\V)$.
\end{thm}
\begin{proof}
By \eqref{mACobsexplicit} and the argument in \cite[Theorem 11]{Budd} we have for $t_1<t_2$
\[\begin{split}  \hat u(t_2)-\hat u(t_1) =
%= \left(e^{t_2 A} -e^{t_1 A}\right)\left(u_0 -\bar u\mathbf{1}\right)+\frac{1}{\varepsilon} \int_0^{t_1}  \left( e^{(t_2-s)A} - e^{(t_1-s)A}\right)\left(\gamma(s) - \bar \gamma(s)\mathbf{1} \right)\; ds\\&\hspace{1em}+\frac{1}{\varepsilon}\int_{t_1}^{t_2}  e^{(t_2-s)A}\left(\gamma(s) - \bar \gamma(s)\mathbf{1} \right)\; ds \\
%&=\left(e^{(t_2-t_1)A}-I\right)\left[e^{t_1A}\left(u_0 -\bar u\mathbf{1}\right)+\frac{1}{\varepsilon} \int_0^{t_1}  e^{(t_1-s)A}\left(\gamma(s) - \bar \gamma(s)\mathbf{1} \right)\; ds\right] \\&\hspace{1em}+\frac{1}{\varepsilon}\int_{0}^{t_2-t_1}  e^{sA}\left(\gamma(t_2-s) - \bar \gamma(t_2-s)\mathbf{1} \right)\; ds\\
\left(e^{(t_2-t_1)A}-I\right)\left( \hat u(t_1)-\bar u \mathbf{1}\right) +\frac{1}{\varepsilon}\int_{0}^{t_2-t_1}  e^{sA}\left(\gamma(t_2-s) - \bar \gamma(t_2-s)\mathbf{1} \right)\; ds
\end{split} \] where $A:= \varepsilon^{-1}I - \Delta$ and $\gamma(t)$ is given at a.e. $t\geq 0$ by \eqref{gammacesaro}.
By Lemma \ref{gammalem}, $\gamma(t)-\bar\gamma(t)\mathbf{1}\in\V_{[\bar u-1,\bar u]}$ for a.e. $t\geq0$, and \[||\hat u(t)-\bar u\mathbf{1}||_\V\leq \max\{\bar u,1-\bar u\}||\mathbf{1}||_\V =: \rho ||\mathbf{1}||_\V\] since $ \hat u(t) \in \V_{[0,1]}$ for all $t\geq 0$. Let $B_{\delta t}: = (e^{\delta t A}- I)/\delta t $. Then $||B_{\delta t}|| =  {(e^{\delta t /\varepsilon}- 1)}/{ \delta t }$, which monotonically increases in $\delta t$.
We thus have for $t_2 -t_1 < 1$,
 \[\begin{split}\frac{\left|\left| \hat u(t_2)- \hat u(t_1)\right|\right|_\V}{t_2-t_1}&\leq ||B_{t_2-t_1}||\cdot \rho||\mathbf{1}||_\V +\frac{1}{\varepsilon} \underset{s\in [0,t_2-t_1]}{\text{ess sup}} \left|\left|e^{sA}\left(\gamma(t_2-s) - \bar \gamma(t_2-s)\mathbf{1} \right) \right|\right|_\V \\
 &\leq  \frac{e^{(t_2-t_1)/\varepsilon}- 1}{ t_2-t_1 }\cdot\rho ||\mathbf{1}||_\V + \frac{1}{\varepsilon}\sup_{s\in [0,t_2-t_1]} \left|\left|e^{sA}\right|\right| \cdot \rho||\mathbf{1}||_\V \\
&\leq \frac{e^{(t_2-t_1)/\varepsilon}- 1}{ t_2-t_1 }\cdot\rho  ||\mathbf{1}||_\V + \frac{1}{\varepsilon} e^{(t_2-t_1)/\varepsilon}  \cdot\rho ||\mathbf{1}||_\V \\
&< \rho ||\mathbf{1}||_\V \left (e^{1/\varepsilon}  -1 + \frac{1}{\varepsilon}e^{1/\varepsilon}\right) 
\end{split}\] and for $t_2 -t_1\geq 1$ we have 
\[\frac{\left|\left| \hat u(t_2)- \hat u(t_1)\right|\right|_\V}{t_2-t_1}\leq \left|\left| \hat u(t_2)- \hat u(t_1)\right|\right|_\V \leq ||\mathbf{1}||_\V \] completing the proof.
\end{proof} 

\section{Towards the multi-class case}\label{mcAC}

%
%
%Discuss Auction Dynamics and multiclass MBO
%
%Multi-obstacle potential: look at AC limit.

So far, we have considered only when $u$ separates into two phases (``classes") in the Allen--Cahn flow, and likewise the MBO scheme applies a binary threshold. In this section we begin to generalise to multiple classes, defining an Allen--Cahn flow against the \emph{multi}-obstacle potential, and a corresponding multi-class semi-discrete scheme, with and without mass-conservation. In future work, we hope to use the above framework to link this to a multi-class MBO scheme, as a special case of the semi-discrete scheme, and investigate the semi-discrete scheme for $\lambda\uparrow 1$ as a choice function for, and an algorithm for computing, multi-class mass-conserving MBO solutions.
%
%Thus far, we have considered only the case where $u$ separates into two phases (``classes") in the Allen--Cahn flow, and likewise where MBO applies a binary threshold. The eventual goal of this work is to generalise to the case with multiple classes, defining an Allen--Cahn flow against the \emph{multi}-obstacle potential, and by the above techniques deriving multi-class MBO and semi-discrete schemes. We shall then compare this method to the ``auction dynamics" method for multi-class MBO studied by Jacobs, Merkurjev and Esedo\=glu \cite{Auction}.
\subsection{Set-up}\label{mcsetup}
Let $K$ denote the number of classes we wish to divide into. Then define the simplex \[ \Sigma := \left\{x\in\mathbb{R}^K\,\middle |\,\sum_{k=1}^K x_k =1, x_k\geq 0\right\}\] and define the function spaces \[\V_{\Sigma} := \left\{U:V\rightarrow\Sigma\right\}\subseteq \left\{U:V\rightarrow\mathbb{R}^K\right\}=:\V^K. \] We will treat $U=(U_{ik})_{i\in V,k=1..K}\in\V^K$ interchangeably as functions on $V$ and as real matrices in $\mathbb{R}^{|V|\times K}$. Then we define the action of operators like $\Delta$ by matrix multiplication, i.e. $\Delta U$ applies $\Delta$ to each column of $U$. 
\begin{nb} For $U \in\V_\Sigma$, $U_{ik}$ describes the proportion of class $k$ at vertex $i$, e.g. if $K=3$ and row $i$ were $(0.2,0.5,0.3)$ that would describe vertex $i$ as being 20\% class 1, 50\% class 2, and 30\% class 3 \emph{(}for an interpretation of this, imagine the classes as red, green, and blue\emph{)}. For $K=2$, each vertex is mapped to some $(x,1-x)$ for $x\in[0,1]$, and we recover the setting of the earlier parts of this paper by projecting onto the first coordinate. 
\end{nb}

As we will be concerned with $U\in\V_\Sigma$, we often restrict attention to the hyperplane \[ \Pi  := \left\{x\in\mathbb{R}^K\,\middle |\,\sum_{k=1}^K x_k =1 \right\}\] and associated space $\V_{\Pi} := \left\{U:V\rightarrow\Pi\right\}$, as $\Sigma$ has non-empty interior in the induced topology on $\Pi$. By restricting from $\mathbb{R}^K$ to $\Pi$ we can eliminate behaviour in the normal direction to $\Pi$, i.e. parallel to $(1,...,1)\in\mathbb{R}^K$. 

%Letting $e_k$ denote the standard basis for $\mathbb{R}^K$, we define the orthonormal system of vectors: \[f_k = \sqrt{\frac{k}{k+1}}\left(\frac{1}{k}\sum_{l=1}^{k}e_l -e_{k+1}\right),\:\: 1\leq k\leq K-1.\] Note that for any $x\in\Pi$ we have that $\Pi = x + \operatorname{span}\{f_k\}$.
%

We define the multi-class inner product: \[\ipp{U,V}:=\text{tr}\left(U^T D^rV\right)\] where we recall that $D$ is the diagonal matrix of degrees. As in the two-class case we define $\V^K_{t\in T}:=\{U:T\rightarrow\V^K\}$, and likewise $\V_{\Sigma,t\in T}$ and $\V_{\Pi,t\in T}$, with inner product \[ ((U,V))_{t\in T}:=\int_T \ipp{U(t),V(t)} \; dt\] and via this define $L^2(T;\V^K)$ and the Sobolev space $H^1(T;\V^K)$ in like manner as before. Thus, $U\in H^1(T;\V^K)$ if and only if $U_i\in H^1(T;\mathbb{R}^K)$ for all $i\in V$, and we define the generalised derivative of $U\in H^1(T;\V^K)$ by, for all $\Phi\in C^\infty_c(T;\V^K)$, \[ \left(\left(\frac{dU}{dt},\Phi\right)\right)_{t\in T}=-\left(\left( U,\frac{d\Phi}{dt}\right)\right)_{t\in T} .\] Finally we define 
 \[ H^1_{loc}(T;\V^K) :=\left\{u\in \V^K_{t\in T}\,\middle|\, u\in H^1(I;\V^K) \text{ for all bounded open }I\subseteq T\right\}\]
and likewise we define $L^2_{loc}(T;\V^K)$.

\subsection{Allen--Cahn flow and the multi-obstacle potential}

We seek to extend \eqref{mACobs} to the multi-class case. First, we define the ``multi-obstacle" potential $\mathcal{W}:\V^K\rightarrow[0,\infty]$ \be \label{multiW} \mathcal{W}(U):= \begin{cases}\sum_{i\in V}d_i^r\prod_{k=1}^K(1-U_{ik}), &\text{if } U\in\V_\Sigma,\\
\infty, &\text{otherwise,}\end{cases}\ee and we define $W:\mathbb{R}^K\rightarrow[0,\infty]$ by \be W(x):= \begin{cases}\prod_{k=1}^K(1-x_k), &\text{if } x\in \Sigma,\\
\infty, &\text{otherwise},\end{cases}\ee and note that $\mathcal{W}(U) = \ip{W\circ U,\mathbf{1}}$. We define the multi-class Ginzburg--Landau energy:\be\label{mcGL} \GL(U) := \frac{1}{2} \ipp{U,\Delta U} + \frac{1}{\varepsilon}\mathcal{W}(U).\ee
As this energy is infinite outside $\V_\Pi$, trajectories of its AC gradient flow will be contained in $\V_\Pi$, so we define formally an AC flow restricted to  $\V_\Pi$ with a multi-well potential as \be\label{formalmAC} \frac{dU}{dt}=-\Delta U -\frac{1}{\varepsilon}\nabla_{\V_\Pi}\mathcal{W}|_{\V_\Pi}(U)\ee and we can define mass-conserving multi-well AC flow as the system of equations \be\label{formalmcmAC}\frac{dU^k}{dt}=-\Delta U^k -\frac{1}{\varepsilon}\left(\nabla_{\V_\Pi}\mathcal{W}|_{\V_\Pi}(U)\right)^k +\frac{1}{\varepsilon}\overline{\left(\nabla_{\V_\Pi}\mathcal{W}|_{\V_\Pi}(U)\right)^k}\mathbf{1}  \ee
where $(U^k)_i:=U_{ik}$. 
\begin{nb}
We restrict to $\V_\Pi$ here and not all the way to $\V_\Sigma$ because $\V_\Pi$ is a translated subspace of $\V^K$, which makes the analysis run smoother. We will handle the restriction of the flow to $\V_\Sigma$ via subdifferential terms. 
\end{nb}
We can then define a mass-conserving semi-discrete scheme:
\[\frac{U^k_{n+1}-e^{-\tau\Delta} U^k_n}{\tau}= -\frac{1}{\varepsilon}\left(\nabla_{\V_\Pi}\mathcal{W}|_{\V_\Pi}(U_{n+1})\right)^k +\frac{1}{\varepsilon}\overline{\left(\nabla_{\V_\Pi}\mathcal{W}|_{\V_\Pi}(U_{n+1})\right)^k}\mathbf{1}.   \]  
A simple calculation gives that (equipping $\Pi\subseteq\mathbb{R}^K$ with the standard inner product)
\[
\left(\nabla_{\V_\Pi}\mathcal{W}|_{\V_\Pi}(U)\right)_{ik} = \left(\nabla_{\Pi} W|_\Pi(U_i)\right)_k. 
\]
However, to treat this rigourously we must account for the non-differentiability of $\mathcal{W}$ and $W$, by considering the subgradient. We make the following definition.
\begin{mydef}[Subgradient and restricted subgradient]
Let $H_0$ be a Hilbert space, and let $f:H_0\rightarrow\bar{\mathbb{R}}$ be a convex function. Then the subgradient of $f$ is given by 
\[
\partial_{H_0}f(x) := \{ v\in H_0 \mid \forall y\in H_0, \: \langle v,y-x\rangle\leq f(y)-f(x)\} .
\] 
Let $H_1\subseteq H_0$ be a closed subspace, and let $\tilde H := x_0 + H_1$ for some $x_0 \in H_0$. Then we define the subgradient of $f|_{\tilde H}$ at $x\in \tilde H$ by 
\[
\partial_{\tilde H}f|_{\tilde H}(x) := \{ v\in H_1 \mid \forall y\in \tilde H, \: \langle v,y-x\rangle \leq f(y)-f(x)\} .
\] 
Note that if $v\in \partial_{H_0}f(x)$ and $v'$ is the orthogonal projection of $v$ onto $H_1$ then for $x,y\in\tilde H$, $y-x\in H_1$ and so $\langle v,y-x\rangle = \langle v',y-x\rangle$ and $v'\in\partial_{\tilde H}f|_{\tilde H}(x)$. 
\end{mydef}
\begin{nb}
Let $A\subseteq \tilde H$ be a convex set, and consider the indicator function
\[
I_A(x):=\begin{cases}0, &\text{if } x\in A,\\
\infty, &\text{if } x\in H_0\setminus A. \end{cases}
\]
Then for $x\notin A$, $\partial_{H_0}I_A(x) = \emptyset$. For $x\in A$ 
\[ \partial_{H_0}I_A(x) = \{ v\in H_0 \mid \forall y\in H_0, \: \langle v,y-x\rangle\leq I_A(y)\}  = \{ v\in H_0 \mid \forall y\in A, \: \langle v,y-x\rangle\leq 0\} .
\] Note that for $x\in \tilde H$, $ \partial_{\tilde H}I_A|_{\tilde H}(x) =  \partial_{H_0}I_A(x) \cap H_1$.
\end{nb}
With this framework in mind, for $x\in\Pi$ we can write, 
\[
W|_\Pi(x) = \prod_{k=1}^K(1-x_k) + I_\Sigma|_\Pi(x)=: W_0|_\Pi(x) + I_\Sigma|_\Pi(x).
\]
%where 
%\[
%I_\Sigma(x):=\begin{cases}0, &\text{if } x\in \Sigma\\
%\infty, &\text{if } x\in\Pi\setminus\Sigma \end{cases} = \begin{cases}0, &\text{if } \forall k \: x_k\geq 0\\
%\infty, &\text{otherwise} \end{cases} = \sum_{k=1}^K I_{[0,\infty)}(x_k). 
%\]
We consider the subgradient of the non-differentiable $I_\Sigma|_\Pi$ term for $x\in\Sigma$
\[
\partial_{\Pi}I_\Sigma|_\Pi(x) = \left\{ v\in \mathbb{R}^K \: \middle | \:  \forall y \in \Sigma, \: \sum_{k=1}^K v_k(y_k-x_k) \leq 0\text{ and } \sum_{k=1}^K v_k = 0 \right\}.
\]
Note that the latter condition comes from $\Pi$ being of the form $x_0 + (1,...,1)^\bot$ for $x_0\in\Pi$.
\begin{prop}\label{subdiff}
For $x\in\Sigma$, $v\in \partial_{\Pi}I_\Sigma|_\Pi(x)$ if and only if $\sum_{k=1}^K v_k = 0$ and for some $\alpha\in\mathbb{R}$ and all $k\in\{1,...,K\}$,
\[
v_k \in \begin{cases}
\{\alpha\}, &\text{if }x_k >0\\
(-\infty,\alpha], &\text{if } x_k = 0
\end{cases}
\]
i.e. for all $k$, $v_k-\alpha \in \partial I_{[0,\infty)}(x_k)$. Equivalently, for all $k\in\{1,...,K\}$, \[
v_k = -\beta_k +\frac{1}{K} \sum_{q=1}^K \beta_q 
\]
for $\beta_k  \in -\partial I_{[0,\infty)}(x_k)$.
\end{prop}
\begin{proof} (``If")
For any such $v$ and $y\in\Sigma$
\[
 \sum_{k=1}^K v_k(y_k-x_k)  =  \sum_{k:x_k=0} v_ky_k + \sum_{k:x_k>0} \alpha(y_k-x_k) \leq \alpha  \sum_{k=1}^K (y_k-x_k) = 0  
\]
so $v\in \partial_{\Pi}I_\Sigma|_\Pi(x)$.

(``Only if") 
%Note that $y\in \Sigma$ if and only if $\xi := y-x$ has $\xi_k\geq -x_k$ and $\sum_{k=1}^K \xi_k = 0$. So $v\in \partial_{\Pi}I_\Sigma(x)$ only if 
%\[
%\sum_{k=1}^K v_k\xi_k\leq 0
%\]
%for all such $\xi$. 
There must exist at least one $\ell$ with $x_\ell>0$. Choose an arbitrary such $\ell$ and an arbitrary $k\neq\ell$, and define $y\in\Sigma$ with $y_\ell = 0$, $y_k = x_k + x_\ell$ and $y_q = x_q$ for $q\neq k,\ell$. Then $v\in \partial_{\Pi}I_\Sigma|_\Pi(x)$ only if $x_\ell(v_k-v_\ell)\leq 0$, i.e. $v_k\leq v_\ell$. As $k$ was arbitrary and $\ell$ was an arbitrary $\ell$ such that $x_\ell > 0$, the desired form for $v$ follows.

For the equivalence, $v_k-\alpha \in \partial I_{[0,\infty)}(x_k)$ and $\sum_k v_k =0$ if and only if $v_k = -\beta_k + \alpha$ and $K\alpha = \sum_k\beta_k$, for some $\beta_k  \in -\partial I_{[0,\infty)}(x_k)$.
\end{proof}
%We immediately infer that 
%\[
%  \{ v\in \mathbb{R}^K\mid \forall k \: v_k \in \partial I_{[0,\infty)}(x_k)\} \subseteq \partial_{\Pi}I_\Sigma(x) 
%\]
Then, we handle the non-differentiability of $W$ via the subgradient, i.e.
\[
\nabla_{\Pi} W|_\Pi(x) \in \nabla_{\Pi} W_0|_\Pi(x) + \partial_{\Pi}I_\Sigma|_\Pi(x)
\]
so we get for $x\in\Sigma$ (recalling that Definition \ref{restrictgrad} requires $\nabla_{\Pi} W_0|_\Pi(x) \in (1,...,1)^\bot$)
\[
\left(\nabla_{\Pi} W|_\Pi(x)\right)_k =  \left(- \prod_{\ell\neq k}(1-x_\ell) + \frac{1}{K}\sum_{q=1}^K\prod_{\ell\neq q}(1-x_\ell) \right) +\left(- \beta_k(x) +\frac{1}{K}\sum_{q=1}^K \beta_q(x)\right)
\]
 for some $\beta_k(x) \in -\partial I_{[0,\infty)}(x_k)$, i.e. 
\[
\beta_k(x) \in 
\begin{cases}\{0\}, &\text{if } x_k > 0, \\
[0,\infty), &\text{if }x_k = 0. \end{cases} 
\] 
Accordingly, for $U\in\V_{\Sigma}$ we define the set \[\mathcal{B}(U):=\{\beta\in \V^K \mid \beta_{ik}\in -\partial I_{[0,\infty)}(U_{ik})\}\]
and the function
\[
f_{ik}(U) := \prod_{\ell\neq k}(1-U_{i\ell}) - \frac{1}{K}\sum_{q=1}^K\prod_{\ell\neq q}(1-U_{i\ell}) 
\]
so that 
\be\label{gradW}
\left(\nabla_{\V_\Pi}\mathcal{W}|_{\V_\Pi}(U)\right)_{ik} = \left(\nabla_{\Pi} W|_\Pi(U_i)\right)_k = - f_{ik}(U) - \beta_{ik} + \frac{1}{K}\sum_{q=1}^K\beta_{iq} =: - f_{ik}(U) - \tilde\beta_{ik} 
\ee
where $\beta\in\mathcal{B}(U)$ and $\tilde \beta_{ik} :=  \beta_{ik} - \frac{1}{K}\sum_{q=1}^K\beta_{iq}$, i.e. $\tilde \beta = \beta (I-\frac{1}{K}J_K)$ where $J_K$ is the $K\times K$ matrix of ones and $I$ is the identity. Define $\tilde{\mathcal{B}}(U):=\{ \beta (I-\frac{1}{K}J_K)\mid\beta\in\mathcal{B}(U)\}$.

Thus plugging \eqref{gradW} into the formal expressions \eqref{formalmAC} and \eqref{formalmcmAC} we get the multi-obstacle AC flow:
\begin{align*}%\label{mcAC}
&\frac{dU}{dt}=-\Delta U(t)+\frac{1}{\varepsilon}f(U(t))+ \frac{1}{\varepsilon}\tilde\beta(t), &\tilde\beta(t)\in\tilde{\mathcal{B}}(U(t)) \end{align*}
and mass-conserving variant:
\begin{align*}%\label{mcmAC}
&\frac{dU^k}{dt}=-\Delta U^k(t) +\frac{1}{\varepsilon}f^k(U(t))+ \frac{1}{\varepsilon}\tilde\beta^k(t)  -\frac{1}{\varepsilon}\overline{f^k(U(t)) + \tilde\beta^k(t)}\mathbf{1},  &\tilde\beta(t)\in\tilde{\mathcal{B}}(U(t)) \end{align*}
As in the two-class case, we therefore tidy up and make the following definitions.
\begin{mydef}[Multi-obstacle AC flow]\label{mcACdef}
Let $T$ be any interval. A pair $(U,\beta)\in\V_{\Sigma,t\in T}\times\V^K_{t\in T}$ with $U\in H^1_{loc}(T;\V^K)\cap C^0(T;\V^K)$ is a solution to \emph{multi-obstacle AC flow} on $T$ when for a.e. $t\in T$
\begin{align}\label{mcAC1}&\frac{dU}{dt}=-\Delta U(t)+\frac{1}{\varepsilon}f(U(t))+ \frac{1}{\varepsilon}\tilde\beta(t), &\tilde\beta(t)\in\tilde{\mathcal{B}}(U(t)) \end{align}
and is a solution to \emph{mass-conserving multi-obstacle AC flow} on $T$ when for a.e. $t\in T$ and all $k\in\{1,2,...,K\}$
\begin{align}\label{mcmAC1}&\hspace{-0.5em}\frac{dU^k}{dt}=-\Delta U^k(t) +\frac{1}{\varepsilon}f^k(U(t))+ \frac{1}{\varepsilon}\tilde\beta^k(t)  -\frac{1}{\varepsilon}\overline{f^k(U(t)) + \tilde\beta^k(t)}\mathbf{1},  &\tilde\beta(t)\in\tilde{\mathcal{B}}(U(t)). \end{align}
\end{mydef}
And therefore we can define the corresponding semi-discrete schemes.
\begin{mydef}[Multi-class semi-discrete scheme]\label{mcSDdef}
We define the \emph{multi-class semi-discrete scheme} by\emph{:} 
\begin{align}\label{mcSD}&U_{n+1}=e^{-\tau\Delta} U_n +\lambda f(U_{n+1}) + \lambda\tilde\beta_{n+1} ,    &\tilde\beta_{n+1}\in\tilde{\mathcal{B}}(U_{n+1}) \end{align}
and the \emph{mass conserving multi-class semi-discrete scheme} by, for all $k\in\{1,2,...,K\}$\emph{:}
\begin{align}\label{mcmSD}
&\hspace{-0.1em}U^k_{n+1}=e^{-\tau\Delta} U^k_n +\lambda f^k(U_{n+1}) + \lambda\tilde\beta^k_{n+1} - \lambda\overline{ f^k(U_{n+1}) + \tilde\beta^k_{n+1}}\mathbf{1},    &\beta_{n+1}\in\tilde{\mathcal{B}}(U_{n+1}). \end{align}
\end{mydef}
\begin{nb}
To check that this all makes sense, we confirm that all of these processes stay inside $\V_\Pi$. For \eqref{mcAC1} and \eqref{mcmAC1}, this is equivalent to checking that\emph{:} 
\[
\sum_{k=1}^K \frac{dU_{ik}}{dt} = 0.
\]
By definition, $\sum_{k=1}^K  f_{ik}(U) = \sum_{k=1}^K \tilde \beta_{ik} = 0$ and so $\sum_{k=1}^K  \overline{ f^k (U)} = \sum_{k=1}^K\overline{ \tilde \beta^k} = 0$. Thus it suffices to check the diffusion term \emph{(}recall that $U\in\V_{\Pi}$ and that $\Delta$ acts as a matrix\emph{)}
\[
\sum_{k=1}^K (\Delta U)_{ik} = \sum_{k=1}^K \sum_{j\in V} \Delta_{ij}U_{jk} = \sum_{j\in V} \Delta_{ij} = (\Delta\mathbf{1})_i= 0.
\] 
For \eqref{mcSD} and \eqref{mcmSD}, if $U_n\in\V_\Pi$ then summing over $k$ the $f$ and $\tilde \beta$ again vanish, so
\[
\sum_{k=1}^K (U_{n+1})_{ik} = \sum_{k=1}^K(e^{-\tau\Delta} U_n)_{ik} = \sum_{k=1}^K\sum_{j\in V} (e^{-\tau\Delta})_{ij} (U_n)_{jk} = \sum_{j\in V} (e^{-\tau\Delta})_{ij} = (e^{-\tau\Delta}\mathbf{1})_i= 1.
\]
\end{nb}
\section{Conclusion}
In this paper, we have translated Rubinstein and Sternberg's mass-conserving Allen--Cahn flow \cite{RS} into the context of dynamics on graphs, and have proved existence, uniqueness and regularity properties of the resulting differential equation with a double-obstacle potential. Following \cite{Budd}, we have formulated a semi-discrete scheme for mass-conserving graph Allen--Cahn flow, proved that the mass-conserving graph MBO scheme emerges exactly as the $\lambda = 1$ special case of this semi-discrete scheme, and shown that the Lyapunov functional from \cite{Budd} remains a Lyapunov functional in the mass-conserving case.

Using the tools of convex optimisation, we have characterised the solutions of this mass-conserving semi-discrete scheme, allowing us to prove that: 
\begin{itemize}
\item As $\lambda\uparrow 1$, the semi-discrete solutions for a given $\lambda$ converge in $\V$ to a solution for $\lambda = 1$, yielding a choice function for MBO solutions.
\item As $\tau,\lambda \downarrow 0$ (and $\varepsilon$ fixed) the semi-discrete solutions converge pointwise to the Allen--Cahn solution.
\end{itemize} 

Finally, we have devised a formulation of mass-conserving multi-class Allen--Cahn flow on graphs (with the multi-obstacle potential). In future work, we seek to extend the results of this paper to the multi-class case, and so devise a choice function for the multi-class MBO scheme as a limit of the multi-class semi-discrete solutions. We shall then compare this method for solving the multi-class MBO scheme on graphs with others in the literature, e.g. the auction dynamics of Jacobs, Merkurjev and Esedo\=glu \cite{Auction}.
\section*{Acknowledgements}
This project has received funding from the European Union’s Horizon 2020 research and innovation programme under the Marie Skłodowska-Curie grant agreement No 777826.

Thanks to Dr. Carolin Kreisbeck for her suggestion of using weak* $L^\infty$ convergence for an alternative method for proving the convergence of the semi-discrete iterates. 

\end{document}